\include{BoxedEPS}

\documentclass[12pt]{amsart}
\usepackage{amsmath,amsthm}
\usepackage{graphicx,amsmath,amsfonts}

\chardef\bslash=`\\ % p. 424, TeXbook
\makeatletter
\def\verbatim{\interlinepenalty\@M \@verbatim
  \leftskip\@totalleftmargin\advance\leftskip2pc
  \frenchspacing\@vobeyspaces \@xverbatim}
\makeatother
\newtheorem{thm}{Theorem}[section]
\newtheorem{cor}[thm]{Corollary}
\newtheorem{lem}[thm]{Lemma}
\newtheorem{prop}[thm]{Proposition}

\newtheorem{rem}[thm]{Remark} %[section]

\numberwithin{equation}{section}

\newcommand{\ZZ}{{\mathbb Z}}% {Z\kern-.45em Z}}
\newcommand{\RR}{{\mathbb R}}%{{I\kern-.3em R}}
\newcommand{\Rd}{ {\Bbb R}^d}
\newcommand{\Zd} {{\Bbb Z}^d}

\begin{document}
\bibliographystyle{plain}

\title[Nonlinear functional equations and   sampling] {Localized nonlinear functional equations
and two  sampling problems in signal processing
}

%%%%%%%%%%%%%%%%%%%%%%%%%%%%%%%%%%%%%%%%%%%%%%%%%%%%%%%%%%%%%%%%%%%%%%%%
\author{ Qiyu Sun }

\address{
Sun: Department of Mathematics,  University of Central Florida,
Orlando, FL 32816}
%% Note the doubled @@:
\email{qiyu.sun@ucf.edu}
\thanks{Research of the  author was  supported in part by  the National Science Foundation
(DMS-1109063).}
%%%%%%%%%%%%%%%%%%%%%%%%%%%%%%%%%%%%%%%%%%%%%%%%%%%%%%%%%%%%%%%%%%%%%%%%

\date{\today }



%\dedicatory{}

%%%%%%%%%%%%%%%%%%%%%%%%%%%%%%%%%%%%%%%%%%%%%%%%%%%%%%%%%%%%%%%%%%%%%%%%
%%%%%%%%%%%%%%%%%%%%%%%%%%%%%%%%%%%%%%%
%%%%%%%%%%%%%%%%%%%%%%%%%%%%%%%%%
\begin{abstract}
Let $1\le p\le \infty$. In this paper, we consider solving a nonlinear functional equation
 $$f(x)=y,$$
where $x, y$ %
belong to  $\ell^p$ and
$f$ has continuous bounded gradient
 in an inverse-closed  subalgebra of ${\mathcal B}(\ell^2)$, the Banach algebra of all bounded linear operators on the Hilbert space $\ell^2$. We introduce  strict monotonicity property for  functions $f$ on Banach spaces  $\ell^p$ so that the above nonlinear functional equation is solvable and
  the solution $x$ depends continuously on the given data $y$ in $\ell^p$. We show that the Van-Cittert iteration
 converges in $\ell^p$ with exponential rate and hence it could be used to locate the true solution of the above nonlinear functional equation. We apply the above theory 
 to handle two  problems in signal processing:  nonlinear sampling termed with instantaneous companding and subsequently average sampling; and
local identification of  innovation positions and qualification of amplitudes of signals  with finite rate of innovation.
\end{abstract}

%\contents
\maketitle

\section{Introduction}

Let  $\ell^p:=\ell^p({\mathcal N}), 1\le p\le \infty$,
be  the  space of all real-valued $p$-summable column vectors  $x:=(x(n))_{n\in {\mathcal N}}$ on a countable index set ${\mathcal N}$
 with its  %standard
norm   denoted by $\|\cdot\|_p:=\|\cdot\|_{\ell^p({\mathcal N})}$.
In the first part of this paper, we consider solvability and stability of a  % localized
nonlinear functional equation
\begin{equation}\label{nfe.def}
f(x)=y
\end{equation}
in Banach spaces $\ell^p, 1\le p\le \infty$. We assume that both  the original signal (image) $x$
 to be recovered and the  distorted signal (image) $y$ to be observed belong to  $\ell^p, 1\le p\le \infty$. We also assume that the function $f$
in the nonlinear functional equation
\eqref{nfe.def}
 has continuous bounded gradient $\nabla f$  in an
inverse-closed Banach  algebra %${\mathcal A}$
of infinite matrices with certain off-diagonal decay (hence being localized and near sparse). Here
the gradient $\nabla f$ of a  function $f$ on $\ell^p$ is defined by
$$\nabla f(x) c:=\lim_{t\to 0} \frac{ f(x+tc)-f(x)}{t} \quad {\rm for} \ x, c\in \ell^p.$$
%for $x, c\in \ell^p$.
Such a hypothesis is satisfied for   functions $f$ modeling a distortion procedure that has strong neighborhood dependency; i.e., the change in amplitude  of the distorted signal (image)  $f(x)$ at each position depends essentially on the change in amplitudes of the original signal (image) $x$
at neighboring  positions. %, such as (non)linear sampling of signals with finite rate of innovation.

For the nonlinear functional equation \eqref{nfe.def} in the Hilbert space $\ell^2$, a well-known result is  that it
 is solvable and stable %for $p=2$ %and stable in the presence of noise with finite energy
 provided that the function  $f$ %in \eqref{nfe.def}
is
 strictly monotonic  on $\ell^2$  and has
     continuous bounded gradient   in  ${\mathcal B}(\ell^2)$, see for instance \cite{minty62, nonlinearbook}
or
Lemma \ref{inverse.lm}.
   Here ${\mathcal B}(\ell^p), 1\le p\le \infty$, is the Banach algebra of all bounded linear operators on $\ell^p$,
    and a function $f$ on $\ell^2$ is said to be {\em strictly monotonic} \cite{nonlinearbook}
if  there exists a positive constant $m_0$ such that
\begin{equation}\label{monotone.def}
(x-x')^T\big (f(x)-f(x')\big)\ge m_0 (x-x')^T (x-x')\quad {\rm for \ all} \  x, x'\in \ell^2.
\end{equation}
The {\em strict monotonicity constant} of a function $f$ on $\ell^2$ is the largest constant $m_0$ such that \eqref{monotone.def} holds.
If a function $f$ on $\ell^2$ has continuous gradient $\nabla f$ in ${\mathcal B}(\ell^2)$,
an equivalent formulation of its strict monotonicity
is that
\begin{equation}\label{monotoneequiv.def}
c^T \nabla f(x)\!\ c\ge m_0 c^T c \quad  {\rm for\ all} \ c\in \ell^2 \ {\rm and} \  x\in \ell^2.
\end{equation}

 In  this paper, we consider  the nonlinear functional equation \eqref{nfe.def} in Banach spaces $\ell^p, p\ne 2$, see Remark \ref{sfm.rem} and Section \ref{friinnovation.section}
for some of our motivations related to signal fidelity measures and local identification of innovations.
We  introduce the  following ``strict monotonicity"  property for functions on $\ell^p, 1\le p\le \infty$,
\begin{equation} \label{pmonotone.def}
 c^T \nabla f(x) c\ge m_0 c^T c\quad {\rm for \ all} \ c\in \ell^2\ {\rm and}\  x\in \ell^{p},
  \end{equation}
where $m_0>0$.
 We call  \eqref{pmonotone.def}  the strict monotonicity property  because
of its similarity to  the equivalent formulation \eqref{monotoneequiv.def} of the strict monotonicity for a function  on $\ell^2$.
We remark that the strict monotonicity  property \eqref{pmonotone.def} for functions on Banach spaces  could be
verified more easily than  the strong accretivity defined by Browder \cite{browder67} and Kato \cite{kato67}
 via normalized duality mapping, 
especially for functions arisen in nonlinear sampling and reconstruction.
We show in Theorems \ref{pinvertibility.tm3} and \ref{errorestimate.tm} that   the nonlinear functional equation \eqref{nfe.def} is solvable and stable
 in $\ell^p, 1\le p\le \infty$, if 
 $f$ has continuous bounded gradient
 in  some Banach subalgebra 
  of ${\mathcal B}(\ell^2)$ and satisfies the strict monotonicity  property \eqref{pmonotone.def}.
The proofs of Theorems \ref{pinvertibility.tm3} and \ref{errorestimate.tm} rely heavily
on  a nonlinear version of Wiener's lemma   (Theorem \ref{nonlinearwiener.tm}) and  a unique   Lipschitz extension
theorem (Theorem \ref{pinvertibility.tm})  for
strictly monotonic functions.

\bigskip

In the second part of this paper, we consider finding the true solution of
 the nonlinear functional equation
\eqref{nfe.def} with $1\le p\le \infty$. For a strictly monotonic function $f$ on $\ell^2$, 
consider the  Van-Cittert  iteration:
\begin{equation} \label{oldvancittert.def}
%\left\{\begin{array}{l} x_1=0\\
x_{n+1}=x_{n}-\alpha  (f(x_{n})-y), \quad   n\ge 0,
%\end{array}\right.
\end{equation}
where  $x_0\in \ell^2$ is the initial guess
and $\alpha>0$ is the relaxation factor. By the strict monotonicity property \eqref{monotoneequiv.def}, we obtain that
$$\|x_{n+1}-x_n\|_2^2\le \big(1-2 \alpha m_0+ \alpha^2\big (\sup_{x\in \ell^2} \|\nabla f(x)\|_{{\mathcal B}(\ell^2)}\big)\big)\|x_{n}-x_{n-1}\|_2^2\quad {\rm for \ all} \ n\ge 1.$$
 Thus we conclude that if $0<\alpha< 2 m_0/ (\sup_{x\in \ell^2} \|\nabla f(x)\|_{{\mathcal B}(\ell^2)})$
then for any initial guess $x_0$,  the sequence $x_n, n\ge 0$, in the  Van-Cittert  iteration \eqref{oldvancittert.def} converges exponentially
 to the true solution $f^{-1}(y)$  of the nonlinear functional equation \eqref{nfe.def}, i.e., there exist constants $C\in (0,\infty)$ and $r\in (0, 1)$ such that
  $$\|x_n-f^{-1}(y)\|_2\le C r^n \quad {\rm for \ all} \ n\ge 1.$$
In Theorem \ref{vancittertiteration.tm}, we borrow a technique  in \cite{suncasp05, suntams07} used for establishing  Wiener's lemma for infinite matrices,
 and we show that the sequence $x_n, n\ge 0$,  in the
Van-Cittert  iteration method converges to the true solution $f^{-1}(y)$ in $\ell^p, 1\le p\le \infty$,  exponentially,
%there exist positive constants $C\in (0,\infty)$ and $r\in (0,1)$ such that
%\begin{equation}
%\|x_n-f^{-1}(y)\|_{p}\le C  \|x_0-f^{-1}(y)\|_{p} \ \! r^n\quad {\rm for \ all}\
%n\ge 1,
%\end{equation}
 provided that
$0<\alpha< m_0/(  \sup_{x\in \ell^\infty} \|\nabla f(x)\|_{{\mathcal B}(\ell^2)}+\sup_{x\in \ell^\infty} \|\nabla f(x)\|_{{\mathcal B}(\ell^2)}^2) $
 and
 the function $f$  has  continuous bounded gradient in some inverse-closed
Banach algebra ${\mathcal B}$ satisfying the differential norm property \eqref{paracompactcondition}.
 Therefore  the
Van-Cittert  iteration method
 could be applied to locate the true solution of the (infinite-dimensional) nonlinear functional equation \eqref{nfe.def} with $p\ne 2$.
We are working on other % easily-implementable
 methods that could solve the nonlinear functional equation \eqref{nfe.def} with
 $p\ne 2$ in a more stable and efficient way. %numerically. % in a stable way.

\bigskip
In the third part of this paper, we consider applying
 the above theory about solving the nonlinear functional equation \eqref{nfe.def}
  to  two sampling problems in  signal processing.
We first apply it to  a nonlinear sampling problem
related to
instantaneous companding
\begin{equation} \label{companding.def} F:\ h(t)\longmapsto F(h(t)),\end{equation}
see Figure \ref{einstein.fig}.
The  technique to transmit (process) the companded signal  $F(h(t))$ instead of the original signal $h(t)$
 is  useful in a transmission
system  that does not respond well to very high or very low
signal levels, since that instantaneous companding technique  may amplify the weaker parts of a signal  and de-emphasize its stronger parts.
\begin{figure}[hbt]
\centering
\begin{tabular}{cc}
         \includegraphics[width=70mm]{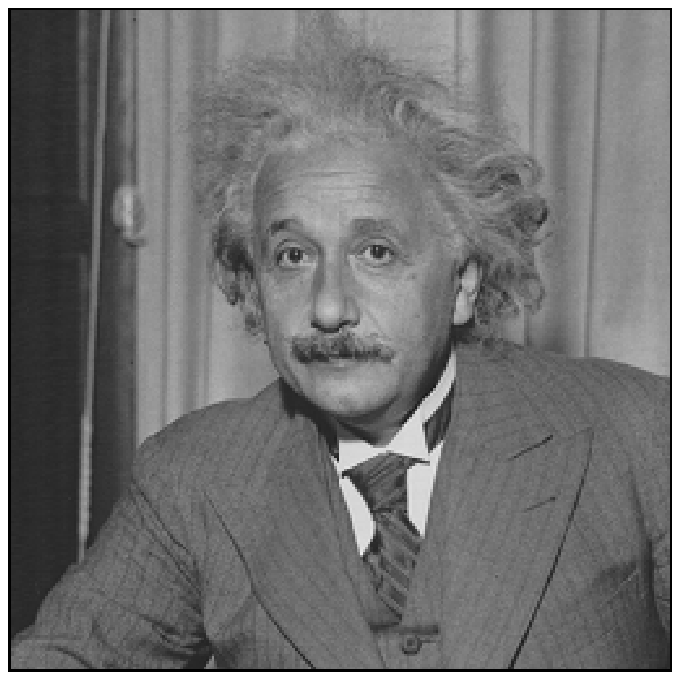} & \includegraphics[width=70mm]{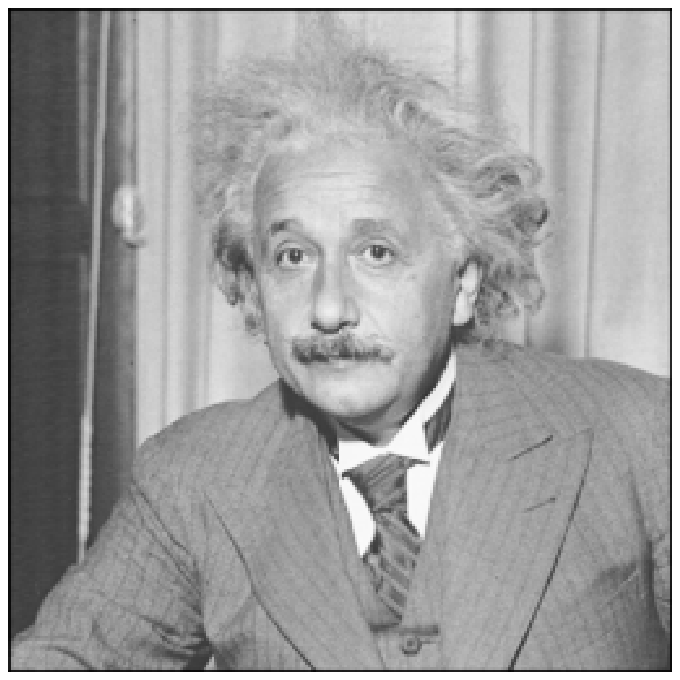}
   \end{tabular}
\caption{\small The left hand side is the original image, while the right hand side is the distorted image via $F(t)= 255\sin( \pi t/510)$.}
\label{einstein.fig}
\end{figure}
There are some theoretical obstacles in the path of instantaneous companding.
For instance,
the instantaneous companding \eqref{companding.def} does not preserve the bandlimitedness. %  in general.
A good news  by Beurling \cite{landau61} is that
 no information of a signal $h(t)$  bandlimited to $[-\pi\Omega, \pi\Omega], \Omega>0$, is lost
   by transmitting the companded signal $F(h(t))$ over %idealized
   bandlimiting channel
   $[-\pi \Omega,\pi \Omega]$, provided  that the companding function $F$  is monotonic increasing. % \cite{landau61}.
Precisely,
the procedure  of instantaneous companding and subsequently  bandlimiting,
 \begin{equation*} %\label{icbl.def}
 B_{\pi\Omega}\ni h\overset{\rm instantaneous\ companding}\longrightarrow F(h) \overset{\rm bandlimiting}\longrightarrow P_{\pi\Omega} (F(h))\in B_{\pi\Omega},\end{equation*}
 is one-to-one on $B_{\pi\Omega}$ \cite{landau61, sandberg94}. Here    $B_{\pi\Omega}$ is the Paley-Wiener space containing all
 square-integrable functions  bandlimited to $[-\pi\Omega, \pi\Omega]$,
    $P_{\pi\Omega}$ is the projection operator from $L^2$ to the Paley-Wiener space $B_{\pi\Omega}$, and
 $L^p:=L^p(\Rd), 1\le p\le \infty$, is the space of all $p$-integrable functions on $\Rd$ with
its %standard
norm  denoted by $\|\cdot\|_p:=\|\cdot\|_{L^p(\Rd)}$, while the inner product on $L^2$
 is denoted by $\langle \cdot, \cdot\rangle$.

   In this paper,
    we consider  a nonlinear sampling procedure of instantaneous  companding
    and subsequently  average sampling \cite{eldar08}:
      \begin{equation}\label{sfpsi.def}
    S_{F, \Psi}: \ h\overset{\rm instantaneous\ companding }
    \longrightarrow F(h) %\overset{\rm bandlimiting}
%   F(h)
   \overset{\rm average\ sampling
 }\longrightarrow  \langle F(h), \Psi\rangle, %\int_{\Rd} F(h(t)) \Psi(t) dt,
%   \longrightarrow P_{\pi\Omega} (F(h))
%    \overset{\rm sampling}\longrightarrow
%    (P_{\pi\Omega} (F(h))(k/\Omega))_{k\in \ZZ}
     \end{equation}
  where $\Psi=(\psi_\gamma)_{\gamma\in \Gamma}$ is the average sampler,
   $\Gamma\subset \Rd$ is the index set representing
the location of the nonideal acquisition devices, and
the sampling functional $\psi_\gamma$
  at the sampling position $\gamma\in \Gamma$ reflects the
characteristics of the  acquisition device  at that %sampling
location.
  Such a  nonlinear sampling procedure %scenario
  appears in  a variety of setups and applications, such as
 charge-coupled device  image sensors \cite{Holst96, Kawai95}, power electronics \cite{Pejovic95}, radiometric photography,
magnetic resonance imaging \cite{Faris01,  Kose90, Shaf04}, and band-limited signal transmission
%and instantaneous companding and subsequently bandlimiting
 \cite{landau61, sandberg94}.
 % Here $\langle f, g\rangle:=\int_{\Rd} f(t) g(t) dt$ when $fg$ is integrable.
 It would be an ill-posed problem to reconstruct a signal $h$ from its
 nonlinear sampling data $S_{F, \Psi}(h)$ %in general
 if
 no a priori information  on the original signal $h$ is given.
 In this paper, we assume that  signals  to be recovered
are superposition of impulse response of varying positions and amplitudes,
\begin{equation}\label{firstmodel.def}
h(t)=\sum_{\lambda\in \Lambda} c(\lambda)\phi_\lambda(t),\end{equation}
where $\Lambda$ is the index set with each $\lambda\in \Lambda$ representing an  innovation position  of the signal,
   $\Phi:=(\phi_\lambda)_{\lambda\in \Lambda}$ is the generator with each entry $\phi_\lambda$ being the impulse response of the signal generating
device at the  location $\lambda$,
and $c:=(c(\lambda))_{\lambda\in \Lambda}$ is the amplitude vector with each entry $c(\lambda)$ being the amplitude of the signal  at  the innovation position $\lambda$, which is  assumed
to be $p$-summable for some $1\le p\le \infty$ in this paper. In other words, signals to be recovered reside in the  space
 \begin{equation}\label{vp.def}
 V_p(\Phi):=\{c^T\Phi |\ c\in \ell^p(\Lambda)\}, \ 1\le p\le\infty
  \end{equation}
(\cite{sunaicm08}).
We remark that signals, such as bandlimited signals in a Paley-Wiener space, time signals in a finitely-generated shift-invariant space, time-frequency signals in Gabor
analysis, and signals in a reproducing kernel subspace of $L^p$,   have the parametric representation \eqref{firstmodel.def}
% by selecting the generator $\Phi$ appropriately
\cite{nashedsun10, sunsiam06, sunaicm08}.

 Associated with the nonlinear sampling procedure \eqref{sfpsi.def} for signals in the space $V_p(\Phi)$ is the function
 \begin{equation}
 \label{ffphipsi.def} f_{F, \Phi, \Psi}:\ \ell^p(\Lambda)\ni c\longrightarrow  c^T\Phi\
  \overset{\rm instantaneous\ companding \ and \  average\ sampling
 }\longrightarrow  \langle F(c^T\Phi), \Psi\rangle\in \ell^p(\Gamma). %\int_{\Rd} F( c^T\Phi(t)) \Psi(t) dt.
 \end{equation}
For the function $f_{F, \Phi, \Psi}$  just defined, % from $\ell^p(\Lambda)$ to $\ell^p(\Gamma)$,
its gradient function $$\nabla f_{F, \Phi, \Psi}(x):\ \ell^p(\Lambda)\longmapsto \ell^p(\Gamma), \ x\in \ell^p(\Lambda)$$
is a family of infinite matrices $(a_{\lambda,\gamma}(x))_{\lambda\in \Lambda, \gamma\in \Gamma}$
 with uniform off-diagonal decay, provided that
 the companding function $F$ has continuous bounded gradient and that both the generator $\Phi$ and the average sampler $\Psi$
 are well-localized,  such as  when they have polynomial decay of order $\beta>d$,
 see Lemma \ref{frirecovery.lem1}. Here we say that $\Phi=(\phi_\lambda)_{\lambda\in \Lambda}$ has {\em polynomial decay} of order $\beta$ if
   \begin{equation} \label{polynomialdecay.def}
  \|\Phi\|_{\infty,\beta}:= \sup_{\lambda\in \Lambda} \|\phi_\lambda(\cdot) (1+|\cdot-\lambda|)^\beta\|_\infty<\infty.
   \end{equation}
 So  we may apply the  established solvability and stability of
 the nonlinear functional equation \eqref{nfe.def} to study the stability of
 the nonlinear  sampling procedure \eqref{sfpsi.def}, and the  Van-Cittert iteration method to reconstruct
 signals $h\in V_p(\Phi)$ from its nonlinear sampling data $\langle F(h), \Psi\rangle$, %$S_{F, \Psi}(h)$,
 see Theorems \ref{frirecovery.tm1} and \ref{frirecovery.tm2} for the theoretical results and Section \ref{fri3.subsection} for the numerical simulations.
 %in Section \ref{fri.section},
The readers may refer to \cite{eldar08, Faktor10} for the study of the
 nonlinear  sampling procedure \eqref{sfpsi.def}  from an engineering viewpoint.

\bigskip

A signal is said to have {\em finite rate of
innovation} if it has finitely many degrees of freedom per unit of time \cite{vetterli02}.
In the second application of our theory about solving the nonlinear functional  equation \eqref{nfe.def}, we consider
%locally
precise identification of  innovation positions  and  accurate qualification of amplitudes
of signals with finite rate of innovation. %The , which
These are important to reach meaningful conclusions
in many applications, such as global positioning system, ultra wide-band
communication, and mass spectrometry.
The readers may refer to \cite{bns09, dvb07,  kusuma03,   mv05, mvb06,  michaeli11, sd07, %sunsiam06,
vetterli02} and references therein for
various techniques, such as the annihilating filter technique and the polynomial reproducing technique,
that have been developed by Vetterli and his school to attack that challenging
problem.

Signals living in  the space $V_p(\Phi)$ have finite rate of innovation provided that
the generator $\Phi=(\phi_\lambda)_{\lambda\in \Lambda}$ is well-localized (for instance, $\Phi$ has  polynomial decay of order $\beta>d$) and
the index set $\Lambda$ is
a   relatively-separated subset of $\Rd$ \cite{sunaicm08}. Here
 a subset $\Lambda$ of $\Rd$  is  said to be {\em relatively-separated} if
 \begin{equation} \label{rsset.def}
 %C(\Lambda):=
\sup_{t\in \Rd} \sum_{\lambda\in \Lambda} \chi_{\lambda+[0,1)^d}(t)<\infty,
\end{equation}
where $\chi_E$ is a characteristic function on a measurable set $E$.
In applications such as global positioning system, cellular radio, ultra wide-band
communication, and mass spectrometry \cite{coombes05, win02,   donoho06, kusuma03}, signals  have
the parametric  representation \eqref{firstmodel.def} but
 both the innovation position set $\Lambda$  and the amplitude vector  $c:=(c(\lambda))_{\lambda\in \Lambda}$
need to be determined.
The dependence of the impulse response $\phi_\lambda$ at the  innovation position $\lambda$  could be given or determined.
For instance, the impulse response $\phi_\lambda$ is assumed to be approximately the shift of a given impulse $\varphi$ (i.e., $\phi_\lambda\approx \varphi(\cdot-\lambda)$) in \cite{coombes05, win02, kusuma03},
and
  the dilated shift of a given shape $\varphi$ with unknown width $w_\lambda$
  (i.e., $\phi_\lambda \approx \varphi((\cdot-\lambda)/w_\lambda)$ ) in \cite{donoho06}.
In this paper, we assume that the dependence of
the impulse response $\phi_\lambda$ on the  innovation position $\lambda$  is the shift of a given impulse response $\varphi$, i.e.,
  $\phi_\lambda=\varphi(\cdot-\lambda)$. So
 signals %to be recovered
 have the following
 parametric representation
  \begin{equation}\label{secondmodel.def}
h(t)=\sum_{\lambda\in \Lambda} c(\lambda) \varphi(\cdot-\lambda).\end{equation}
 We could also say that signals live in an {\em almost-shift-invariant space} generated by $\varphi$, as in the Fourier domain
\begin{equation} \label{secondmodelfourier.def}\hat h(\xi)=C(\xi)\hat \phi(\xi)\end{equation} is the product of an almost periodic function
$C(\xi):=\sum_{\lambda\in \Lambda} c(\lambda) e^{-i\lambda \xi}$
 and the Fourier transform $\hat \varphi$ of the impulse response $\varphi$.
Here
the {\em Fourier transform}  $\hat f$ of an integrable function $f$  on $\Rd$ is defined by
$ \hat f(\xi)=\int_{\Rd} e^{-ix\xi} f(x) dx$.

%where $(c(\lambda))_{\lambda\in \Lambda}$ is a bounded sequence.
In Section \ref{friinnovation.section}, %this paper,
we consider recovering a signal $h$
in the parametric representation \eqref{secondmodel.def} accurately from  its nonlinear  sampling data
$\langle F(h), \Psi\rangle$, with a priori approximate information on  innovation positions %$\Lambda$
 and
 amplitudes. 
For that purpose, we require that the linearization
of the   nonlinear sampling process
\begin{equation} \label{localnonlinearsampling.def}
\Big(\begin{matrix}\Lambda\\  c\end{matrix}\Big) %(c(\lambda))_{\lambda\in \Lambda} )
\overset{\rm parametric\ representation \ \eqref{secondmodel.def}}
 \longrightarrow
%\longmapsto
h %(t):=\sum_{\lambda\in \Lambda} c(\lambda) \varphi(t-\lambda)
\overset{\rm %instantaneous \
companding\ and \ %average \
sampling}  \longrightarrow
%\longmapsto
\langle F(h), \Psi\rangle
\end{equation}
is stable at
the given  approximate innovation position vector $\Lambda_0$ and
amplitude vector $c_0=(c_0(\lambda))_{\lambda\in \Lambda_0}$; i.e.,
there exist positive constants $A$ and $B$ such that
\begin{equation}\label{slambda0.stability}
A\|e\|_2\le \|S_{\Lambda_0, c_0} e\|_2\le B\|e\|_{2}\quad {\rm for \ all} \ e\in (\ell^2(\Lambda_0))^{d+1},
\end{equation}
where
\begin{equation}\label{slambda0.def}
S_{\Lambda_0, c_0}:=\left (\begin{array}{c}
-c_0(\lambda) \langle F'(h_0)  \nabla \varphi(\cdot-\lambda), \psi_\gamma\rangle\\
\langle F'(h_0) \varphi(\cdot-\lambda), \psi_\gamma\rangle\end{array}\right )_{\gamma\in \Gamma, \lambda\in \Lambda_0}\end{equation}
and $h_0=\sum_{\lambda\in \Lambda_0} c_0(\lambda) \varphi(\cdot-\lambda)$.
Associated with the linearization  of the   nonlinear sampling process \eqref{localnonlinearsampling.def} at $(\Lambda_0, c_0)$ is the function %$f_{\Lambda_0, c_0, \Psi}$ on $(\ell^\infty(\Lambda_0))^{d+1}$ defined by
\begin{equation} \label{flambda0.def}
f_{\Lambda_0, c_0, \Psi}:\  %(\ell^\infty(\Lambda_0))^{d+1}\ni
\Big(\begin{matrix} \sigma\\ c\end{matrix}\Big)\longmapsto %\left (\begin{array}{cc} I_{\Lambda_0, \Lambda_0} & \\  & (1/c_0(\lambda) I_d)_{\lambda\in \Lambda_0} \end{array} \right )_{\gamma\in \Gamma, \lambda\in \Lambda_0}
 (S_{\Lambda_0, c_0}^T S_{\Lambda_0, c_0})^{-1} S_{\Lambda_0, c_0}^T
\langle F(h)-F(h_0), \Psi\rangle,
%\in \ell^\infty(\Gamma)
\end{equation}
where  $\sigma=(\sigma(\lambda))_{\lambda\in \Lambda_0}, c=(c(\lambda))_{\lambda\in \Lambda_0}$ and $h=\sum_{\lambda\in \Lambda_0} (c_0(\lambda)+c(\lambda)) \varphi(\cdot-\lambda-\sigma(\lambda))$.
% and  $h_0=\sum_{\lambda\in \Lambda_0} c_0(\lambda) \varphi(\cdot-\lambda)$.
We observe  that the gradient of  the function $f_{\Lambda_0, c_0, \Psi}$
is a  family of infinite matrices with certain off-diagonal decay, and that
$f_{\Lambda_0, c_0, \Psi}$ has strict monotonicity property in a small neighborhood of the origin; i.e.,
\begin{equation*}
\inf_{\|e\|_{(\ell^2(\Lambda_0))^{d+1}}=1} e^T \nabla f_{\Lambda_0, c_0, \Psi} \Big(\begin{matrix} \sigma\\  c\end{matrix}\Big)  e>0
\end{equation*}
for all $\sigma\in (\ell^\infty(\Lambda_0))^d$ and $ c\in \ell^\infty(\Lambda_0)$ with $\|c\|_\infty+\|\sigma\|_\infty\le \delta_0$,
where $\delta_0>0$ is a sufficiently small number.
Thus we may apply our theory
about solving the nonlinear functional equation \eqref{nfe.def} with $p=\infty$ indirectly to local identification of
innovation  positions  and qualification of  amplitudes, see Theorem  \ref{friidentification.tm}.

Finally we consider the following natural questions: 1) how to find approximate innovation positions and amplitudes
from the given sampling data and  2) how to verify the stability condition
\eqref{slambda0.stability} for the linearization  matrix. %$S_{\Lambda_0, c_0}$
%at the approximate innovation position  set and amplitude vector.
From  the stability  condition  \eqref{slambda0.stability} %for the linearization matrix
%$S_{\Lambda_0, c_0}$,
it follows  that  approximate innovation positions should be separated  from each other and that
approximate amplitudes at  innovation positions should be above certain level.
So we may model those  signals as  superposition of impulse response of
{\em active} and {\em nonactive} generating devices located at an unknown neighborhood of a uniform grid;
i.e.,
after appropriate scaling,  we may assume that signals  live in a
perturbed shift-invariant space  %$V_{\infty, 0}(\varphi; \sigma), \sigma\in \ell^\infty$, where
\begin{equation}\label{vinfinity0.def}
V_{\infty,\oslash }(\varphi; \sigma):=\Big\{\sum_{k\in \Zd} c(k) \varphi(\cdot-k-\sigma(k)) \ \big| \ (c(k))_{k\in \Zd}\in \ell^\infty_{\oslash}(\Zd)\Big\}
\end{equation}
with {\em unknown}  perturbation $\sigma:=
(\sigma(k))_{k\in \Zd}$, where
\begin{equation}\label{linftyzero.def}
\ell^\infty_{\oslash}(\Zd)=\Big\{c:=(c(k))_{k\in \Zd}  \big |\
 \|c\|_{\ell^\infty_\oslash}:= \sup_{c(k)\ne 0} |c(k)|+ |c(k)|^{-1}<\infty\Big\}.\end{equation}
We may assume that signals live in the infinite unions of almost-shift-invariant spaces
$\cup_{\sigma} V_{\infty,\oslash }(\varphi; \sigma)$, see \cite{dolu08} for sampling in finite unions of subspaces.
A negative result for sampling in the  perturbed shift-invariant space
 with unknown perturbations is that not all signals in such a space
can be recovered
from their samples provided that
 $\varphi$ %is the sinc function ${\rm sinc}(t)=\sin (\pi t)/(\pi t)$ or
satisfies the popular and traditional Strang-Fix condition. The reason is that
 in this case, one cannot determine the perturbation $\sigma_0$ of signals $\sum_{k\in\Zd}
\varphi(\cdot-k-\sigma_0), \sigma_0\in \Rd$, %in $V_{\infty, \oslash}(\varphi)$
as they have  same constant amplitudes and are identical. % for any fixed perturbation $\sigma_0\in \Rd$. We also remark that the requirement \eqref{blindcondition} excludes
In Theorem  \ref{blindsampling.tm}, we  provide a positive result for
 sampling in a perturbed shift-invariant space
 with unknown perturbations. We show that any signal $h
$ %(t)=\sum_{k\in \Zd} c(k) \varphi(t-k-\sigma(k))$
in a perturbed shift-invariant space $V_{\infty, \oslash}(\varphi; \sigma)$  with small perturbation error $\|\sigma\|_\infty$ %(\sigma(k))_{k\in \Zd}\|_\infty$
can be  recovered from its  sampling data
$\{\langle h, \psi_m(\cdot-k)\rangle| 1\le m\le M, k\in \Zd\}$ provided that
\begin{equation}\label{blindcondition} {\rm rank}
\begin{pmatrix} [\widehat{\nabla \varphi}, \widehat\psi_1](\xi) & \cdots & [\widehat{\nabla \varphi}, \widehat\psi_M](\xi)
%\sum_{l\in \Zd} (\xi+2l\pi)\widehat \varphi(\xi+2\pi l) \overline{\hat \psi_1(\xi+2l\pi)} & \cdots &
%\sum_{l\in \Zd} (\xi+2l\pi)\hat \varphi(\xi+2\pi l) \overline{\widehat \psi_M(\xi+2l\pi)}
\\
[\hat\varphi, \widehat\psi_1](\xi) & \cdots & [\hat{\varphi}, \widehat\psi_M](\xi)
\end{pmatrix}=d+1 \quad {\rm for \ all} \ \xi\in [-\pi, \pi]^d,
\end{equation}
see  Remark \ref{blindsampling.averagesamplingremark} for the connection of the above assumption with sampling in shift-invariant spaces.
Here  the {\em bracket product} $[f, g]$
of two square-integrable functions $f$ and $g$ is given by
$[f, g](\xi)=\sum_{l\in \Zd} f(\xi+2l\pi) \overline{ g(\xi+2l\pi)}$.

\bigskip
The paper is organized as follows. The first part of this paper contains four sections concerning %the
 solvability and stability of the % localized
nonlinear functional equation \eqref{nfe.def}. The starting point of  solvability
 is  Wiener's lemma for strictly monotone functions in Section \ref{wiener.section}.
 A sufficient condition for the solvability of the nonlinear functional equation \eqref{nfe.def} in $\ell^p$
is introduced in Section \ref{invertibility.section}, while
 the unique Lipschitz extension theorem  in Section \ref{extension.section} and  Wiener's lemma in Section \ref{wiener.section} are crucial in the proof. The stability of the nonlinear functional equation \eqref{nfe.def} in $\ell^p$
%standard error estimate to solve the nonlinear functional equation \eqref{nfe.def}
is studied in Section \ref{error.section}.
 The central pieces of the second part of this paper  are
 the global  exponential convergence of the Van-Cittert iteration  in $\ell^p$
  and the local R-quadratic convergence of the quasi-Newton iteration  in $\ell^p, 1\le p\le \infty$, see Section \ref{algorithm.section} for details.
Our proof of the global exponential convergence depends heavily on the paracompactness idea used in \cite{suncasp05, suntams07} to establish Wiener's lemma for infinite matrices.
In the third part of this paper, we apply  solvability and stability
 of the nonlinear functional equation \eqref{nfe.def} in the first part
and  numerical implementation in the second part to two sampling problems in signal processing:
1)  stable recovery of signals $h$ living
 in the space $V_p(\Phi)$  from  their nonlinear sampling data $\langle F(h), \Psi\rangle$ in
Sections \ref{fri.section}; and  2)  local identification of innovation positions and qualification of amplitudes of signals $h$ having the parametric representation \eqref{secondmodel.def} from their nonlinear sampling data $\langle F(h), \Psi\rangle$
in Section  \ref{friinnovation.section}.

%\bigskip

In this paper, the capital letter $C$ denotes an absolute constant, which may be different at different occurrences.
% We say that $A \lesssim B$ if
%$A\le C B$ for some absolute constant $C$.

\section{Nonlinear Wiener's lemma}\label{wiener.section}

A Banach subalgebra ${\mathcal A}$ of ${\mathcal B}$ is said to be  {\em
inverse-closed}   if an element in ${\mathcal A}$, that is invertible in ${\mathcal B}$, is also invertible in ${\mathcal A}$.
Inverse-closedness occurs  in many fields of mathematics
under various names, such as spectral invariance, Wiener pair, and local subalgebra.
Inverse-closedness (= Wiener's lemma)  has been established
for infinite matrices satisfying
various off-diagonal decay conditions (or equivalently for localized linear functions on $\ell^p$), see for instance \cite{balan,
 baskakov90, gkwieot89, grochenigklotz10,
gltams06,  jaffard90, shincjfa09, sjostrand94, suncasp05,
suntams07,  sunca11} and   the  survey papers \cite{grochenigsurvey, Krishtal11}.
 Wiener's lemma for infinite matrices plays  crucial  roles for  well-localization of dual wavelet frames and  Gabor
frames \cite{balanchl04, gkwieot89, jaffard90, krishtal08},  algebra of pseudo-differential operators
\cite{grochenig06, grochenigs07, sjostrand94},
fast numerical implementation  \cite{christensen05, dahlkejat10, grochenigr10},
stable signal recovery % with finite rate of innovation
\cite{akramjfa09, bns09, grochenigr10, halljin10,  sunsiam06, sunaicm08}, and optimization \cite{mjieee08, mjieee09}.

A quantitative version of inverse-closedness is to admit {\em norm control}, which is fundamental
in our study of   the  nonlinear functional equation \eqref{nfe.def} and its applications to sampling in signal processing.
Here an inverse-closed Banach subalgebra ${\mathcal A}$   of ${\mathcal B}$ is said to
 admit  norm control in ${\mathcal B}$
if there exists a continuous function $h$ from $[0, \infty)\times [0, \infty)$ to $[0, \infty)$ such that
$$\|A^{-1} \|_{\mathcal A}\le h(\|A\|_{\mathcal A}, \|A^{-1}\|_{\mathcal B})$$
for all $A\in {\mathcal A}$ with $A^{-1}\in {\mathcal B}$.
Unlike the inverse-closedness,  there are not many papers devoted to
the above quantitative version of inverse-closedness  \cite{grochenigklotz10, grochenigklotz12, nikolski99, tao05}.
The algebra ${\mathcal W}$ of commutative infinite matrices of the form $A:=(a(i-j))_{i,j\in \ZZ}$ with norm $\|A\|_{\mathcal W}=\sum_{j\in \ZZ} |a(j)|$
 is inverse-closed in ${\mathcal B}(\ell^2)$ by the classical Wiener's lemma \cite{wiener}
but it does not admit norm control \cite{nikolski99}.
The  $C^*$-subalgebras ${\mathcal A}$ of ${\mathcal B}$ with a common unit and a  differential norm, i.e.,
\begin{equation}\label{differentialnorm} \|AB\|_{\mathcal A}\le C (\|A\|_{\mathcal A}\|B\|_{\mathcal B}+ \|A\|_{\mathcal B}\|B\|_{\mathcal A})\quad {\rm for \ all} \ A, B\in {\mathcal A}\subset {\mathcal B},
\end{equation}   was shown in \cite{grochenigklotz10} to admit norm control.
The smoothness described in the above differential subalgebra ${\mathcal A}$ of ${\mathcal B}$ has been widely used in operator theory and non-commutative geometry \cite{blackadarcuntz91, kissin94,rieffel10} and also appears in approximation theory \cite{devore93}.
For the  quantitative version of inverse-closedness, we will show in Proposition \ref{wienerlmforinfinitematrices.prop}
%at the end of this section
that a Banach subalgebra ${\mathcal A}$  of ${\mathcal B}(\ell^2)$
 admits norm control in ${\mathcal B}(\ell^2)$ if
 it contains the identity matrix $I$,
it is closed under the transpose operation,  and it satisfies
\begin{equation}\label{paracompactcondition} %{weakparacompactcondition}
\|AB\|_{\mathcal A}\le C_0 \big( \|A\|_{\mathcal A} \|B\|_{\mathcal A}^\theta \|B\|_{{\mathcal B}(\ell^2)}^{1-\theta}
  + \|B\|_{\mathcal A} \|A\|_{\mathcal A}^\theta \|A\|_{{\mathcal B}(\ell^2)}^{1-\theta}\big)
%
%   \Big(\big(\frac{}{\|A\|_{\mathcal A}}\big)^{1-\theta}+ \|B\|_{{\mathcal B}(\ell^2)}^{1-\theta}+
%\|A\|_{\mathcal B(\ell^2)}^{1-\theta}\|B\|_{{\mathcal A}}^{1+\theta}\big)
\quad {\rm for \ all}\
A, B\in {\mathcal A}\subset {\mathcal B}(\ell^2),\end{equation}
 where $C_0\in (0, \infty)$ and  $\theta\in [0, 1)$. The above inequality \eqref{paracompactcondition}, a weak form
of the differential norm property \eqref{differentialnorm},
is satisfied by many families of Banach subalgebras of ${\mathcal B}(\ell^2)$ \cite{gltams06,  jaffard90, suncasp05,
suntams07,  sunca11},  and it will also be  used  later to
 guarantee the  exponential convergence of the  Van-Cittert iteration method, see Theorem \ref{vancittertiteration.tm}.
\smallskip

The main topic of this section is to introduce   Wiener's lemma for strictly monotonic functions, see Theorem \ref{nonlinearwiener.tm}. It states that
 a  strictly monotonic function  on $\ell^2$
 with  continuous bounded  gradient   in some inverse-closed subalgebra
 of ${\mathcal B}(\ell^2)$ has
its inverse with continuous bounded gradient
 in  the same Banach subalgebra.
The above nonlinear Wiener's lemma  will be used later  to establish the  unique extension theorem (Theorem \ref{pinvertibility.tm})
 and the invertibility theorem (Theorem \ref{pinvertibility.tm3}) for functions with
 continuous bounded gradient in a Banach subalgebra of ${\mathcal B}(\ell^2)$.

\begin{thm}\label{nonlinearwiener.tm} %Let ${\mathcal A}$ be an inverse-closed Banach subalgebra of ${\mathcal B}(\ell^2)$.
Let ${\mathcal A}$  be a Banach subalgebra of ${\mathcal B}(\ell^2)$   that admits norm control in ${\mathcal B}(\ell^2)$.
 If
 $f$ is a strictly monotonic function on $\ell^2$ such that its gradient $\nabla f$
 is bounded and continuous in ${\mathcal A}$ (i.e.,
$
\sup_{x\in \ell^2}\|\nabla f(x)\|_{{\mathcal A}}<\infty
$
and
$
\lim_{x'\to x \ {\rm in } \ \ell^2} \|\nabla f(x')-\nabla f(x)\|_{{\mathcal A}}=0$ for all $x\in \ell^2$),
then  $f$ is invertible and  its inverse $f^{-1}$ has continuous bounded gradient
 $\nabla (f^{-1})$ in ${\mathcal A}$.
%
%there exists a positive constant $C_1$
%\begin{equation}\|\nabla f^{-1}(x)\|_{\mathcal A}\le C_1 \quad {\rm for\ all} \ x\in \ell^2\end{equation}
%for some positive constant $C_1$.
\end{thm}

For a relatively-separated subset $\Lambda$ in \eqref{rsset.def} and a  nonnegative number $\beta\ge 0$, define  the {\em Jaffard class} ${\mathcal J}_\beta(\Lambda)$
 of infinite matrices $A=(a(\lambda, \lambda'))_{\lambda, \lambda'\in \Lambda}$ by
 \begin{equation}\label{jaffard.def}
{\mathcal J}_{\beta}(\Lambda): =\Big\{ A\ \Big|\
\|A\|_{{\mathcal J}_\beta(\Lambda)}:= \sup_{\lambda, \lambda'\in \Lambda} (1+|\lambda-\lambda'|)^\beta |a(\lambda, \lambda')|<\infty
\Big\}\end{equation}
\cite{jaffard90, suncasp05}. The Jaffard class can be interpreted as the set of infinite matrices having polynomial off-diagonal decay
of degree $\beta$. It has been established in \cite{suncasp05, suntams07} that the Jaffard class
 ${\mathcal J}_\beta(\Lambda)$ in \eqref{jaffard.def} with $\beta>d$ is closed under the transpose operation,
it contains the identity matrix $I$,
 and
it satisfies the differential norm property %the paracompact  condition
\eqref{paracompactcondition}.
%\eqref{weakparacompactcondition}.
Hence the Jaffard class
 ${\mathcal J}_\beta(\Lambda)$  with $\beta>d$ admits norm control in ${\mathcal B}(\ell^2)$ by Proposition
\ref{wienerlmforinfinitematrices.prop} at the end of this section.
Then applying Theorem \ref{nonlinearwiener.tm}
to the Banach algebra ${\mathcal J}_\beta(\Lambda)$ leads to the following conclusion.

\begin{cor}\label{specialnonlinearwiener.cr}
Let $\beta>d$ and  $\Lambda$ be a relatively-separated subset of $\Rd$.
 If
$f$ is a strictly monotonic function on $\ell^2(\Lambda)$
 such that its gradient $\nabla f$
 is bounded and continuous in ${\mathcal J}_\beta(\Lambda)$,
then  $f$ is invertible and  its inverse $f^{-1}$ has continuous bounded gradient
  in ${\mathcal J}_\beta(\Lambda)$.
\end{cor}

To prove Theorem \ref{nonlinearwiener.tm}, we  recall a textbook version, i.e., ${\mathcal A}={\mathcal B}(\ell^2)$,
about invertibility of a strictly monotonic function on $\ell^2$, see for instance \cite{nonlinearbook}.

\begin{lem}\label{inverse.lm}
Let  $f$ be a strictly monotonic function on $\ell^2$ such that its gradient
$\nabla f$ is  bounded and  continuous  in ${\mathcal B}(\ell^2)$.
Then  $f$ is invertible and its inverse $f^{-1}$  has continuous bounded gradient
 in  ${\mathcal B}(\ell^2)$.
\end{lem}

\begin{proof}[Proof of Theorem \ref{nonlinearwiener.tm}]
Since ${\mathcal A}$ is a Banach subalgebra of ${\mathcal B}(\ell^2)$,
  the gradient $\nabla f$ is also bounded and  continuous in ${\mathcal B}(\ell^2)$ by the assumptions on $f$. Then
 $f$ is invertible and  the gradient of its inverse $f^{-1}$ is bounded  and continuous in ${\mathcal B}(\ell^2)$ by
 Lemma \ref{inverse.lm}. Therefore the gradient $\nabla (f^{-1})$ is
bounded in ${\mathcal A}$ as  ${\mathcal A}$ admits norm control in ${\mathcal B}(\ell^2)$
and
\begin{equation}\label{inverse.lm.pf.eq8}
\nabla f (f^{-1}(x)) \nabla f^{-1}(x)=I
\end{equation}
by  $f(f^{-1}(x))=x$ for all $x\in \ell^2$.
Combining \eqref{inverse.lm.pf.eq8} with
 the boundedness of $\nabla f$ in  ${\mathcal A}$
and the continuity of $f^{-1}$ on $\ell^2$, we obtain that  for any $y_0\in \ell^2$
\begin{eqnarray*}
% & &
  \big\|\nabla f^{-1}(y)-\nabla f^{-1}(y_0)\big\|_{\mathcal A} %\\
\! & = &\!\!
 \big\|\big(\nabla f ( f^{-1}(y))\big)^{-1}-\big(\nabla f( f^{-1}(y_0))\big)^{-1}\big\|_{\mathcal A}\nonumber\\
& \le &\!\! \big(\sup_{z\in \ell^2} \big\| \big(\nabla f (z )\big)^{-1}\big\|_{{\mathcal A}}\big)^2
 \big\|\nabla f ( f^{-1}(y))-\nabla f( f^{-1}(y_0))\big\|_{\mathcal A}\nonumber\\
 & \to & \!\! 0 \quad  {\rm as} \ y\to y_0 \ {\rm in} \ \ell^2.
\end{eqnarray*}
 This  proves the continuity of $\nabla (f^{-1})$ in ${\mathcal A}$.
\end{proof}

Denote by $A^T$ the transpose of an infinite matrix $A$. We conclude this section by showing that  a Banach subalgebra ${\mathcal A}$ of ${\mathcal B}(\ell^2)$
satisfying  %the paracompact condition
\eqref{weakparacompactcondition}, a weak version of the differential norm property \eqref{paracompactcondition},
 admits norm control in ${\mathcal B}(\ell^2)$.

\begin{prop}\label{wienerlmforinfinitematrices.prop}
Let ${\mathcal A}$  be a Banach subalgebra of ${\mathcal B}(\ell^2)$  such that  it contains the identity matrix $I$ and
it is closed under the transpose operation. If
there exist positive constants
$D\in (0,\infty)$ and $\theta\in [0,1)$ such that
\begin{equation}\label{weakparacompactcondition}
\|A^2\|_{\mathcal A}\le D \|A\|_{\mathcal A}^{1+\theta}\|A\|_{{\mathcal B}(\ell^2)}^{1-\theta}\quad {\rm for \ all}\
A\in {\mathcal A},\end{equation}
%a weak version of the inequality \eqref{paracompactcondition}.
then  ${\mathcal A}$ is an inverse-closed Banach subalgebra of ${\mathcal B}(\ell^2)$.
Moreover, for any $A\in {\mathcal A}$
with $A^{-1}\in {\mathcal B}(\ell^2)$,
\begin{equation}\label{wienerlmforinfinitematrices.prop.eq1}
\|A^{-1}\|_{\mathcal A}  \le   \|A^T\|_{\mathcal A} \|A\|_{{\mathcal B}(\ell^2)}^{-2}
\sum_{n=0}^\infty r^n
\big( D ( \|I\|_{\mathcal A}+\|A^T\|_{\mathcal A} \|A\|_{\mathcal A}
\|A^{-1}\|_{{\mathcal B}(\ell^2)}^{-2}) r^{-1}\big)^{\frac{1+\theta}{\theta} n^{\log_2 (1+\theta)}}
\end{equation}
if $\theta\in (0, 1)$, and
\begin{equation}\label{wienerlmforinfinitematrices.prop.eq1prime}
\|A^{-1}\|_{\mathcal A}  \le   \|A^T\|_{\mathcal A} \|A\|_{{\mathcal B}(\ell^2)}^{-2}
\sum_{n=0}^\infty r^n
\big( D ( \|I\|_{\mathcal A}+\|A^T\|_{\mathcal A} \|A\|_{\mathcal A}
\|A^{-1}\|_{{\mathcal B}(\ell^2)}^{-2}) r^{-1}\big)^{1+\log_2 n }
\end{equation}
if $\theta=0$,
 where $r=1-(\|A^{-1}\|_{{\mathcal B}(\ell^2)}\|A\|_{{\mathcal B}(\ell^2)})^{-2}\in [0,1)$.
%and $C_0$ and $\theta$ are constants given in \eqref{weakparacompactcondition}.
\end{prop}

\begin{proof}
We provide only a brief sketch of the proof
as we can mimic the  argument in \cite{suncasp05} where ${\mathcal A}$ is the Jaffard class.
%We provide  a sketch of the proof here only.
Let $A\in {\mathcal A}$
with $A^{-1}\in {\mathcal B}(\ell^2)$.
 Then the product $A^T A$ between  $A$ and its transpose $A^T$
is a positive operator on $\ell^2$ by  the invertibility  assumption on  $A$,
%as ${\mathcal A}$ is a Banach $*$-algebra,
%and   by  the invertibility  assumption on  $A$.
\begin{equation}\label{winerlmforinfinitematrices.lm.pf.eq1} (1- r)  I\le \frac{A^TA}{\|A\|_{{\mathcal B}(\ell^2)}^2}\le  I\end{equation}
for some positive constant $r=1-(\|A^{-1}\|_{{\mathcal B}(\ell^2)}\|A\|_{{\mathcal B}(\ell^2)})^{-2}\in [0,1)$.
% by the assumption that $A^{-1}\in {\mathcal B}(\ell^2)$.
Set $B=I-A^TA/\|A\|_{{\mathcal B}(\ell^2)}^2$. 
Then
\begin{equation}\label{winerlmforinfinitematrices.lm.pf.eq1+}
\|B\|_{\mathcal A}\le \|I\|_{\mathcal A}+\|A^T\|_{\mathcal A} \|A\|_{\mathcal A} \|A\|_{{\mathcal B}(\ell^2)}^{-2}<\infty\end{equation}
as ${\mathcal A}$ is a Banach algebra closed under the transpose operation.
Using \eqref{weakparacompactcondition} and \eqref{winerlmforinfinitematrices.lm.pf.eq1}, and recalling that ${\mathcal A}$ is a Banach algebra,
 we obtain that for all $n\ge 0$,
\begin{equation}\label{winerlmforinfinitematrices.lm.pf.eq5}
\|B^{2n+1}\|_{\mathcal A}\le \|B\|_{{\mathcal A}} \|B^{2n}\|_{{\mathcal A}}
\end{equation}
and
\begin{equation} \label{winerlmforinfinitematrices.lm.pf.eq4}
\|B^{2n}\|_{\mathcal A}\le D \|B^n\|_{{\mathcal B}(\ell^2)}^{1-\theta} \|B^n\|_{{\mathcal A}}^{1+\theta}
\le C_0 r^{n(1-\theta)} \|B^n\|_{{\mathcal A}}^{1+\theta},
\end{equation}
where $D$ and $\theta$ are given in \eqref{weakparacompactcondition}.
% and
%belongs to  the Banach algebra ${\mathcal A}$,
 Applying \eqref{winerlmforinfinitematrices.lm.pf.eq5} and \eqref{winerlmforinfinitematrices.lm.pf.eq4} iteratively, we get
\begin{eqnarray} \label{winerlmforinfinitematrices.lm.pf.eq1++}
\|B^n\|_{{\mathcal A}}  & \le &  D \|B\|_{\mathcal A}^{\epsilon_0} r^{\frac{1-\theta}{2}\sum_{i=1}^k \epsilon_i 2^{i}} \|B^{\sum_{i=1}^k \epsilon_i 2^{i-1}}\|_{\mathcal A}^{1+\theta}\le \cdots
\nonumber\\
& \le & D^{\sum_{j=0}^{k-1} (1+\theta)^j} \|B\|_{\mathcal A}^{\sum_{j=0}^{k-1} \epsilon_j (1+\theta)^j}
r^{\frac{1-\theta}{2} \sum_{j=1}^k \big(\sum_{i=j}^k \epsilon_i 2^{i}\big) (\frac{1+\theta}{2})^{j-1}} \|B^{\epsilon_k}\|_{\mathcal A}^{(1+\theta)^k}\nonumber\\
& \le &
 ( D\|B\|_{{\mathcal A}} r^{-1})^{\frac{1+\theta}{\theta} n^{\log_2 (1+\theta)}} r^n
\end{eqnarray}
if $\theta\in (0,1)$,
and
\begin{equation} \label{winerlmforinfinitematrices.lm.pf.eq1+++}
\|B^n\|_{{\mathcal A}}
 \le   ( D\|B\|_{{\mathcal A}} r^{-1})^{1+\log_2 n} r^n
\end{equation}
if $\theta=0$,
where $n=\sum_{i=0}^k \epsilon_i 2^i$ with $\epsilon_i\in \{0,1\}$ and $\epsilon_k=1$.
%Hence $(A^TA)^{-1}= \|A\|_{{\mathcal B}(\ell^2)}^2\sum_{n=0}^\infty B^n$
%belongs to
%the Banach algebra ${\mathcal A}$.
Then  $A^{-1} \in {\mathcal A}$
as
\begin{equation} \label{winerlmforinfinitematrices.lm.pf.eq1++++}
A^{-1}= (A^TA)^{-1} A^T= \|A\|_{{\mathcal B}(\ell^2)}^{-2}\sum_{n=0}^\infty B^n A^T.\end{equation}
The estimates \eqref{wienerlmforinfinitematrices.prop.eq1} and \eqref{wienerlmforinfinitematrices.prop.eq1prime} hold by
\eqref{winerlmforinfinitematrices.lm.pf.eq1+}, \eqref {winerlmforinfinitematrices.lm.pf.eq1++},
 \eqref{winerlmforinfinitematrices.lm.pf.eq1+++} and \eqref{winerlmforinfinitematrices.lm.pf.eq1++++}.
\end{proof}

%%%%%%%%%%%%%%%%%%%%%%%%%%%%%%%%%%
%%%%%%%%%%%%%%%%%%%%%%%%%%%%%%%%%%%%
%%%%%%%%%%%%%%%%%%%%%%%%%%%%%%%%%%%

\smallskip

\section{Unique extension} % of strictly monotonic functions
\label{extension.section}

  In this section, we  apply   the nonlinear  Wiener's lemma (Theorem \ref{nonlinearwiener.tm}) established in the previous section, and
  show that a strictly monotonic function $f$  on $\ell^2$
  with continuous bounded gradient in some  Banach
algebra
can be uniquely extended to a continuously invertible
function on  $\ell^p, 1\le p\le \infty$.

\begin{thm}\label{pinvertibility.tm} Let   ${\mathcal N}$ be a countable index set,
 ${\mathcal A}$ be a Banach subalgebra of ${\mathcal B}(\ell^2({\mathcal N}))$
that admits norm control in ${\mathcal B}(\ell^2({\mathcal N}))$,
 and let
  $f$ be a  strictly monotonic function on $\ell^2({\mathcal N})$ %with monotonicity constant $m_0>0$ such that
  with the property that $f(0)=0$ and
  the gradient $\nabla f$
 is bounded  and continuous in ${\mathcal A}$.
  Then
 \begin{itemize}
 \item [{(i)}]
the function $f$ can be  uniquely extended to a continuous function $F$  on  $\ell^p({\mathcal N}), 1\le p<\infty$,
 if ${\mathcal A}$ is   a Banach subalgebra of ${\mathcal B}(\ell^p({\mathcal N}))$; and

  \item [{(ii)}] the function $f$  can be uniquely extended to a  continuous function $F$ on  $\ell^\infty({\mathcal N})$ with respect to the weak-star topology, provided that
      ${\mathcal A}$ is  a Banach subalgebra of ${\mathcal B}(\ell^\infty({\mathcal N}))$
and that
  there is a family of nested finite subsets ${\mathcal N}_m, m\ge 1$, of ${\mathcal N}$ such that
      $\cup_{m=1}^\infty {\mathcal N}_m={\mathcal N}$, $ {\mathcal N}_m\subset {\mathcal N}_{m+1}$ for all $m\ge 1$, and
   the   projection operators $P_m, m\ge 1$, defined by
$$P_mx=(x(n)\chi_{{\mathcal N}_m}(n))_{n\in {\mathcal N}} \ {\rm for} \  x=(x(n))_{n\in {\mathcal N}}\in \ell^\infty({\mathcal N}),$$
%from $\ell^\infty({\mathcal N})$ to  $\ell^\infty({\mathcal N}_{m})\subset \ell^\infty ({\mathcal N})$
satisfy
\begin{equation}\label{pinvertibility.tm.eq1}
\lim_{m\to \infty} \sup_{\|A\|_{{\mathcal A}}\le 1}\|P_m A (I-P_{m+1})\|_{{\mathcal B}(\ell^\infty({\mathcal N}))}=0.
 \end{equation}
%where the projections $P_m, m\ge 1$ are
\end{itemize}
Moreover, the  above unique extension $F$
 has continuous inverse $F^{-1}$ on $\ell^p({\mathcal N}), 1\le p\le \infty$,
 and satisfies the following uniform Lipschitz conditions:
\begin{equation}\label{pinvertibility.tm.eq2}
\|F(x)-F(x')\|_{p}\le \big(\sup_{z\in \ell^2({\mathcal N})} \|\nabla f(z)\|_{\mathcal A}\big)
\big (\sup_{\|A\|_{\mathcal  A}\le 1} \|A\|_{{\mathcal B}(\ell^p({\mathcal N}))}\big) \|x-x'\|_p
\end{equation}
and
\begin{equation}\label{pinvertibility.tm.eq3}
\|F^{-1}(x)-F^{-1}(x')\|_{p}\le \big(\sup_{z\in \ell^2({\mathcal N})} \|\nabla f^{-1}(z)\|_{\mathcal A}\big)
\big (\sup_{\|A\|_{\mathcal  A}\le 1} \|A\|_{{\mathcal B}(\ell^p({\mathcal N}))}\big) \|x-x'\|_p
\end{equation}
for all $x, x'\in \ell^p({\mathcal N}), 1\le p\le \infty$.
\end{thm}

The technical assumption \eqref{pinvertibility.tm.eq1} for  a Banach algebra ${\mathcal A}$ of ${\mathcal B}(\ell^\infty({\mathcal N}))$
 can be thought as a weak version of band-limited approximability, as
it holds if
$$\lim_{m\to \infty} \sup_{\|A\|_{\mathcal  A}\le 1} \| A-A_m\|_{{\mathcal B}(\ell^\infty({\mathcal N}))}=0,$$
where $A_m=(a(i,j)\chi_{\rho(i,j)\le K})_{i,j\in {\mathcal N}}$ is the band-limited truncation of  $A=(a(i,j))_{i,j\in {\mathcal N}}
\in {\mathcal A}$, and $\rho$ is a distance function on ${\mathcal N}\times {\mathcal N}$ with $\sup_{i\in {\mathcal N}}
\#\{j\in {\mathcal N}: \rho(i,j)\le m\}<\infty$ for all $m\ge 1$.
We remark that \eqref{pinvertibility.tm.eq1} is {\em invalid} for the classical
Schur class
  $${\mathcal S}:=\Big\{ (a(i,j))_{i,j\in \ZZ}: \ \sup_{i\in \ZZ} \sum_{j\in \ZZ} |a(i,j)|+ \sup_{j\in \ZZ} \sum_{i\in \ZZ} |a(i,j)|<\infty\Big\}.$$

Given $\beta>d$ and a relatively-separated subset $\Lambda$ of $\Rd$, we notice that the Jaffard class ${\mathcal J}_\beta(\Lambda)$ is
a Banach subalgebra of ${\mathcal B}(\ell^p(\Lambda))$ for all $1\le p\le \infty$, and
also that
 \eqref{pinvertibility.tm.eq1} holds   by letting ${\mathcal N}_m=\{\lambda\in \Lambda: |\lambda|\le 2^m\}, m\ge 1$. Therefore
 a strictly monotonic function $f$  on $\ell^2(\Lambda)$
  with continuous bounded gradient in the Jaffard class ${\mathcal J}_\beta(\Lambda)$
can be extended to a continuously invertible
function on  $\ell^p(\Lambda), 1\le p\le \infty$.

\begin{cor}\label{specialextension.cor}
Let $\beta>d, \Lambda$  be a relatively-separated subset of $\Rd$,
 and let  $f$ be a strictly monotonic function on $\ell^2(\Lambda)$ such that  $f(0)=0$ and
  the gradient $\nabla f$
 is bounded  and continuous in ${\mathcal J}_\beta(\Lambda)$.
 Then the function $f$ can be extended to a
 continuously invertible  function   on $\ell^p(\Lambda), 1\le p\le \infty$.
\end{cor}

To prove Theorem \ref{pinvertibility.tm}, we need a technical lemma about the extension
of a  function on $\ell^q, 1\le q<\infty$, which is not necessarily strictly monotonic.

\begin{lem}\label{pinvertibility.lem} Let $1\le q<\infty$,
  ${\mathcal N}$ be a countable index set,
 ${\mathcal A}$ be a Banach subalgebra of ${\mathcal B}(\ell^q({\mathcal N}))$,
%that admits norm control in ${\mathcal B}(\ell^2({\mathcal N}))$,
 and
  $f$ be a continuous function on $\ell^q({\mathcal N})$ such that $f(0)=0$ and
   its gradient $\nabla f$
 is bounded  and continuous in ${\mathcal A}$.
  Then
 \begin{itemize}
 \item [{(i)}]
 the function $f$ can be  uniquely extended to a continuous function $F$ on  $\ell^p({\mathcal N})$   if
 $1\le p<\infty$ and  ${\mathcal A}$ is a Banach subalgebra
  of ${\mathcal B}(\ell^p({\mathcal N}))$; and
\item[{(ii)}]   the function $f$  can be uniquely extended to a
  continuous function $F$ on  $\ell^\infty({\mathcal N})$ with respect to the weak-star topology
  if ${\mathcal A}$ is a Banach subalgebra
  of ${\mathcal B}(\ell^\infty({\mathcal N}))$
 and satisfies \eqref{pinvertibility.tm.eq1}.
\end{itemize}
Furthermore the extension $F$ satisfies the uniform Lipschitz condition
\eqref{pinvertibility.tm.eq2}.
\end{lem}

\begin{proof} (i)\quad  Given  $x\in \ell^p:=\ell^p({\mathcal N}), 1\le p<\infty$, take a Cauchy sequence
$x_n\in \ell^p\cap \ell^q, n\ge 1$,
that converges to $x$ in $\ell^p$ and satisfies $x_1=0$.
%The existence of such a  Cauchy sequence follows from the density of $\ell^p\cap \ell^2$ in $\ell^p$.
 %,  $x_n:=P_nx\in \ell^2\cap \ell^p_w, n\ge 1$ is a Cauchy sequence that . Then
For any $x^{\prime}, x^{\prime\prime}\in \ell^p\cap \ell^q$,
\begin{eqnarray}\label{pinvertibility.tm.pf.eq1}
\|f(x^{\prime})-f(x^{\prime\prime})\|_{p} & = &\!
\Big\| \Big(\int_0^1 \nabla f(sx^{\prime}+(1-s)x^{\prime\prime}) ds\Big) (x^{\prime}-x^{\prime\prime})\Big\|_{p}\nonumber\\
& \le & \!\big(\sup_{0\le s\le 1} \|\nabla f(sx^{\prime}+(1-s)x^{\prime\prime})\|_{{\mathcal B}(\ell^p)}\big)
 \|x^{\prime}-x^{\prime\prime}\|_{p}\nonumber\\
& \le & \!  \big(\sup_{z\in \ell^q} \|\nabla f(z)\|_{\mathcal A}\big) \big( \sup_{\|A\|_{\mathcal A}\le 1} \|A\|_{{\mathcal B}(\ell^p)}\big)
 \|x^{\prime}-x^{\prime\prime}\|_{p}.
\end{eqnarray}
%by the assumption that ${\mathcal A}\subset B(\ell^p_w)$ and boundedness assumption
Thus $f(x_n), n\ge 1$, is a Cauchy sequence in $\ell^p$. Denote the limit of $f(x_n), n\ge 1$, by $F(x)$;
%as the limit of the  sequence $f(x_n), n\ge 1$, in $\ell^p_w$;
i.e.,
%\begin{equation}\label{pinvertibility.tm.pf.eq2}
$F(x):=\lim_{n\to \infty} f(x_n)\ {\rm in }\ \ell^p$.
%\end{equation}
%Obviously
This together with
 \eqref{pinvertibility.tm.pf.eq1}  implies  % and \eqref{pinvertibility.tm.pf.eq2}
 that
 the function $F$ on $\ell^p$ is a continuous extension of
the function $f$ on $\ell^q$ and satisfies  the uniform Lipschitz condition \eqref{pinvertibility.tm.eq2}.
  Furthermore such a continuous extension is unique
due to  the density of $\ell^p\cap \ell^q$ in $\ell^p$.

(ii)\quad  Let $P_m, m\ge 1$, be as in Theorem \ref{pinvertibility.tm}.
For every $m\ge 1$,
\begin{eqnarray}\label{pinvertibility.tm.pf.eq3}
\|P_m f(x)\|_{\infty}  & \le &  \|f(x)-f(0)\|_{\infty}\le \int_0^1
\|  \nabla f(s x)x \|_{\infty}\ ds\nonumber\\
\quad & \le &  % \big(\sup_{z\in \ell^q} \|\nabla f(z)\|_{{\mathcal B}(\ell^\infty_w)}\big) \|x\|_{\ell^\infty_w}
   \big(\sup_{z\in \ell^q} \|\nabla f(z)\|_{\mathcal A}\big) \big(\sup_{\|A\|_{\mathcal A}\le 1} \|A\|_{{\mathcal B}(\ell^\infty)}\big)
\| x\|_{\infty}
\end{eqnarray}
for $x\in \ell^q$, and
\begin{eqnarray} \label{pinvertibility.tm.pf.eq4}
& & \|P_m f(x)-P_m f(x')\|_{\infty}\nonumber\\
 & = & \Big\| \int_0^1
P_m  \nabla f(s x+(1-s) x')(I-P_{k})
 (x-x') ds \Big\|_{\infty}\nonumber\\
\quad & \le &   \big(\sup_{z\in \ell^q} \|\nabla f(z)\|_{\mathcal A}\big)
\big(\sup_{\|A\|_{\mathcal A}\le 1}\|P_m  A(I-P_k)\|_{{\mathcal B}(\ell^\infty)}\big)
\| x-x'\|_{\infty}
%\nonumber\\
%\quad & \to &  0 \quad {\rm as} \ n, n'\to \infty.
\end{eqnarray}
for all $x, x'\in \ell^q$ with $P_kx=P_k x'$.
The above two estimates together with
\eqref{pinvertibility.tm.eq1}  %on infinite matrices in ${\mathcal A}$
imply that
 $P_m f(P_nx), n\ge 1$,  is a Cauchy sequence in $\ell^\infty$ for  every $x\in \ell^\infty$ and $m\ge 1$.
 Denote the limit of the above Cauchy sequence by $F_m(x)$.
Clearly %  $F_m(0)=0$ for all $m\ge 1$ and
$P_{m'} F_{m}(x)=F_{m'}(x)$ for all sequences $x\in \ell^\infty$ and integers $m'$ and $m$ with $1\le m'\le m$.
This implies that there is a
unique sequence, to be denoted by $F(x)$, such that
\begin{equation}\label{pinvertibility.tm.pf.eq5} P_m F(x)=F_m(x) \quad {\rm for \ all} \ x\in \ell^\infty \ {\rm and} \ m\ge 1.\end{equation}
The function $F$ is well-defined on $\ell^\infty$ with $F(0)=0$
because
$$ \|F(x)\|_{\infty}
=\sup_{m\ge 1} \lim_{n\to \infty} \|P_mf(P_nx)\|_{\infty}
 \le  \big(\sup_{z\in \ell^q} \|\nabla f(z)\|_{\mathcal A}\big) \big(\sup_{\|A\|_{\mathcal A}\le 1} \|A\|_{{\mathcal B}(\ell^\infty)}\big)
\| x\|_{\infty}
 $$
by \eqref{pinvertibility.tm.pf.eq3}. Moreover, by
\eqref{pinvertibility.tm.pf.eq3}, \eqref{pinvertibility.tm.pf.eq4} and \eqref{pinvertibility.tm.pf.eq5},
\begin{equation}\label{pinvertibility.tm.pf.eq6}
F(x)=f(x) \quad {\rm for \ all} \ x\in \ell^q;
\end{equation}
% for  all $x\in \ell^\infty\cap \ell^q$;
\begin{equation}\label{pinvertibility.tm.pf.eq7}
\|F(x)-F(x')\|_{\infty}  \le     (\sup_{z\in \ell^q} \|\nabla f(z)\|_{\mathcal A}\big)\big(\sup_{\|A\|_{\mathcal A}\le 1} \|A\|_{{\mathcal B}(\ell^\infty)}\big) \|x-x'\|_{\infty}
\end{equation}
for  all $x,x'\in \ell^\infty$; and
    \begin{eqnarray}\label{pinvertibility.tm.pf.eq8}
 \|P_mF(x)- P_m  F(P_nx)\|_{\infty}  & \le &
 \big(\sup_{\|A\|_{\mathcal A}\le 1}\|P_{n-1}  A(I-P_{n})\|_{{\mathcal B}(\ell^\infty)}\big)
\nonumber\\
 & &  \times
 \big(\sup_{z\in \ell^q} \|\nabla f(z)\|_{\mathcal A}\big)
\| x\|_{\infty}
 \end{eqnarray}
for all  $x\in \ell^\infty$ and $m\le n-1$. %Thus the function $F$ on $\ell^\infty_w$ is a continuous extension of  the function $f$  on $\ell^q$.
Denote by  $\langle \cdot, \cdot\rangle$  the action between  sequences in $\ell^1$ and $\ell^\infty$.
From \eqref{pinvertibility.tm.pf.eq6}, \eqref{pinvertibility.tm.pf.eq7} and \eqref{pinvertibility.tm.pf.eq8} it follows that
\begin{eqnarray}\label{pinvertibility.tm.pf.eq9}
& &
|\langle y, F(x)\rangle-\langle y, F(x')\rangle | \nonumber\\
 & \le &  \|(I-P_m)y\|_{1} \big(\|F(x)\|_{\infty} +
 \|F(x')\|_{\infty}\big)\nonumber\\
& &  +
\|P_my\|_{1} \big(\|P_mF(x)-P_m F(P_{m+1} x)\|_{\infty}\nonumber\\
& & \quad  +
 \|P_mF(x')-P_m F(P_{m+1} x')\|_{\infty}\big)\nonumber\\
& & +
\|P_my\|_{1} \|P_mF(P_{m+1} x)-P_m F(P_{m+1} x')\|_{\infty}\nonumber\\
&\le &
\big(\sup_{z\in \ell^q} \|\nabla f(z)\|_{\mathcal A}\big)
\Big(\|y\|_{1} \sup_{\|A\|_{\mathcal A}\le 1}\|P_{m}  A(I-P_{m+1})\|_{{\mathcal B}(\ell^\infty)}\nonumber\\
 & & \ \ +\|(I-P_m)y\|_{1} \sup_{\|A\|_{\mathcal A}\le 1} \|A\|_{{\mathcal B}(\ell^\infty)}
 \Big)\big(\|x\|_{\infty}+\|x'\|_{\infty}\big)\nonumber\\
 & &\quad  + \big(\sup_{z\in \ell^q} \|\nabla f(z)\|_{\mathcal A}\big) \big(\sup_{\|A\|_{\mathcal A}\le 1} \|A\|_{{\mathcal B}(\ell^\infty)}\big)
 \|y\|_{1}
\|P_{m+1}(x-x')\|_{\infty}
\end{eqnarray}
for all $y\in \ell^1, x, x'\in \ell^\infty$ and $m\ge 1$.
This together with the off-diagonal decay assumption \eqref{pinvertibility.tm.eq1}  proves
the continuity of the function $F$ with respect to the  weak-star topology. % of $\ell^\infty$.
Thus $F$ is an  extension of the function $f$ on $\ell^q$ which is continuous with respect to the weak-star topology and satisfies
the uniform Lipschitz condition
\eqref{pinvertibility.tm.eq2}. The uniqueness of such an extension
follows from the fact that   $\ell^\infty\cap \ell^q$ is dense in  $\ell^\infty$
with respect to the  weak-star topology. 
\end{proof}

\begin{proof}[Proof of Theorem \ref{pinvertibility.tm}] Let $F$ be the unique extension of the function $f$
to $\ell^p$ in Lemma \ref{pinvertibility.lem}. Then $F$ satisfies \eqref{pinvertibility.tm.eq2}
by Lemma \ref{pinvertibility.lem}.
%Then it suffices to prove that it is invertible on $\ell^p$
% and its inverse  satisfies \eqref{pinvertibility.tm.eq3}.
By  Theorem \ref{nonlinearwiener.tm}, $\nabla f^{-1}$ is continuous and bounded in ${\mathcal A}$.
%$\sup_{z\in \ell^q} \|\nabla f^{-1}(z)\|_{\mathcal A}<\infty$.
Applying Lemma \ref{pinvertibility.lem} to the function $f^{-1}$ on $\ell^2$
leads to the
 unique extension $G$  of the function $f^{-1}$
to $\ell^p$.
%which is continuous
%in strong topology when $1\le p<\infty$ and continuous in weak-star  topology when $p=\infty$.
By
 Lemma \ref{pinvertibility.lem}, $G$ is the  continuous inverse of the function $F$ on $\ell^p$
that satisfies \eqref{pinvertibility.tm.eq3}, because
$\ell^p\cap \ell^2$ is dense in  $\ell^p$ with respect to the strong topology if $1\le p<\infty$ and
 the  weak-star topology of $\ell^\infty$ if $p=\infty$.
\end{proof}

\section{Invertibility}\label{invertibility.section}

In this section, we  
apply the unique Lipschitz extension theorem (Theorem \ref{pinvertibility.tm}) established in the
previous section, and  show that the restriction of
a function $f$ on $\ell^{\infty}$ to $\ell^p\subset \ell^\infty, 1\le p\le \infty$, is
 invertible on $\ell^p$ if its gradient $\nabla f$ is bounded and continuous in some Banach algebra and has the strict monotonicity property
  \eqref{pinvertibility3.tm.eq1}.
  %A key observation is that
%.

\begin{thm}\label{pinvertibility.tm3}
Let $2\le p_0\le \infty$,  ${\mathcal A}$ be a subalgebra of ${\mathcal B}(\ell^q)$ for all $1\le q\le \infty$
with the property that it admits norm control in ${\mathcal B}(\ell^2)$ and
 satisfies \eqref{pinvertibility.tm.eq1}. Assume that  $f$ is a continuous function on $\ell^{p_0}$ such that $f(0)=0$,
  its gradient $\nabla f$ is bounded and continuous in ${\mathcal A}$ and satisfies
  \begin{equation} \label{pinvertibility3.tm.eq1}
 c^T \nabla f(x) c\ge m_0 c^T c\quad {\rm for \ all} \ c\in \ell^2\ {\rm and}\  x\in \ell^{p_0}
  \end{equation}
where $m_0>0$. Then
\begin{itemize}
\item[{(i)}] for every $1\le p\le p_0$, the restriction of $f$ to $\ell^p$ is a continuously invertible function  on $\ell^p$ and the gradient of its inverse
  is  bounded and continuous in ${\mathcal A}$; and
  \item[{(ii)}]  for every $p_0<p\le \infty$,
  the function $f$ can be extended to a continuously invertible function $F$
 on $\ell^p$ that  satisfies % Lipschitz conditions
 \begin{equation}\label{pinvertibility3.tm.eq2} \|F(x)-F(x')\|_p\le \big(\sup_{z\in \ell^{p_0}}\|\nabla f(z)\|_{\mathcal A}\big)
 \big(\sup_{\|A\|_{\mathcal A}\le 1}\|A\|_{{\mathcal B}(\ell^p)}\big) \|x-x'\|_p
 \end{equation}
 and
 \begin{equation}\label{pinvertibility3.tm.eq3}
 \|F^{-1}(x)-F^{-1}(x')\|_p\le \big(\sup_{z\in \ell^{p_0}}\|\nabla f^{-1}(z)\|_{\mathcal A}\big)
 \big(\sup_{\|A\|_{\mathcal A}\le 1}\|A\|_{{\mathcal B}(\ell^p)}\big) \|x-x'\|_p
 \end{equation}
 for all $x, x'\in \ell^p$.
  \end{itemize}
\end{thm}

We  do not know whether or not the conclusions in the above invertibility theorem hold  when  $1\le p_0<2$.
Taking $p=p_0$ in Theorem \ref{pinvertibility.tm3} leads to the following  generalization of the well known result for
 invertibility  of a function on $\ell^2$, see \cite{minty62, nonlinearbook} or Lemma \ref{inverse.lm}.

 % highlighted in the introduction section.

\begin{cor}\label{p0invertibility.cor}
Let $2\le p\le \infty$, and ${\mathcal A}$ be a subalgebra of ${\mathcal B}(\ell^q)$ for all $1\le q\le \infty$
with the property that it admits norm control in ${\mathcal B}(\ell^2)$ and
 satisfies \eqref{pinvertibility.tm.eq1}. If  $f$ is a continuous function on $\ell^{p}$ such that $f(0)=0$,
  its gradient $\nabla f$ is bounded and continuous in ${\mathcal A}$ and satisfies
  \begin{equation*} % \label{pinvertibility3.tm.eq1}
 c^T \nabla f(x) c\ge m_0 c^T c\quad {\rm for \ all} \ c\in \ell^2\ {\rm and}\  x\in \ell^{p},
  \end{equation*}
 then
 $f$ has continuous inverse in $\ell^p$ and  the inverse $f^{-1}$ has continuous bounded  gradient in ${\mathcal A}$.
\end{cor}

\begin{rem}\label{linearequation.rem}{\rm  Applying Corollary \ref{p0invertibility.cor} to a linear function $f(x)=Ax$, where $A$ is an  infinite matrix in ${\mathcal A}$, we conclude that if $A$ is positive definite, i.e., $C_1I\le A\le C_2 I$ for some positive constants $C_1$ and $C_2$, then $A$ has bounded inverse in ${\mathcal B}(\ell^p)$ for all $1\le p\le \infty$. The above result follows from Wiener's lemma for infinite matrices in ${\mathcal A}$ if
we assume further that
${\mathcal A}$ is  closed under transpose operation  and has the identity $I$ as its unit \cite{baskakov90, grochenigsurvey, grochenigklotz10, grochenigklotz12, gltams06, Krishtal11, shincjfa09, suntams07, sunca11}.}
\end{rem}

Taking   the  Jaffard class ${\mathcal J}_{\beta}(\Lambda)$
as the Banach algebra ${\mathcal A}$ in  Theorem \ref{pinvertibility.tm3},
we  obtain %the following:

\begin{cor}\label{specialpinvertibility.cor}
Let $ \beta>d, \Lambda$ be a relatively-separated subset of $\Rd$, and  let $f$ be a continuous function on $\ell^{\infty}(\Lambda)$ such that $f(0)=0$, and
  its gradient $\nabla f$ is bounded and continuous in ${\mathcal J}_\beta(\Lambda)$ and satisfies
 \eqref{pinvertibility3.tm.eq1}. Then
 for every $1\le p\le \infty$, the restriction of $f$ to $\ell^p(\Lambda)$ is a continuously invertible function  on $\ell^p(\Lambda)$, and the gradient of its inverse is bounded and continuous in ${\mathcal J}_\beta(\Lambda)$.
 \end{cor}

\begin{proof}[Proof of Theorem \ref{pinvertibility.tm3}] (i)\quad
First we prove that  for every $1\le p\le p_0$ the restriction of the function $f$ on $\ell^p$,  denoted by $F_p$, is well-defined.
Let $P_m, m\ge 1$, be as in Theorem \ref{pinvertibility.tm}. Take $x\in \ell^p\subset \ell^{p_0}$, then
$P_m f(x)\in \ell^p$  and
 \begin{eqnarray}\label{pinvertibility.tm3.pf.eq3}
\|P_m f(x)\|_p & = & \|P_m f(x)-P_m f(0)\|_p\le
\int_0^1 \|P_m \nabla f(sx) x\|_{p}\ \! ds\nonumber\\
& \le  & \big(\sup_{z\in \ell^{p_0}} \|\nabla f(z)\|_{{\mathcal B}(\ell^p)}\big)
\|x\|_p \quad {\rm for \ all} \ m\ge 1,
 \end{eqnarray}
%for all $m\ge 1$.
which implies that
\begin{equation*} %\label{pinvertibility.tm3.pf.eq4}
\|f(x)\|_p=\sup_{m\ge 1} \|P_m f(x)\|_p\le
 \big(\sup_{z\in \ell^{p_0}} \|\nabla f(z)\|_{{\mathcal A}}\big) \big(\sup_{\|A\|_{\mathcal A}\le 1} \|A\|_{{\mathcal B}(\ell^p)}\big)
\|x\|_p
\end{equation*}
for all $x\in \ell^p$. Thus the restriction of $f$ on $\ell^p$ is well-defined.

Next we prove that for every $1\le p\le p_0$, the function $F_p$ is
continuous
on $\ell^p$ and has continuous bounded gradient in ${\mathcal A}$.
Similar to the argument in
\eqref{pinvertibility.tm3.pf.eq3}, we obtain
\begin{equation*} %\label{pinvertibility.tm3.pf.eq5}
\|F_p(x')-F_p(x)\|_p %=\sup_{n\ge 1} \|P_n f(x)-P_nf(x')\|_p
\le
\big(\sup_{z\in \ell^{p_0}} \|\nabla f(z)\|_{{\mathcal A}}\big)
\big(\sup_{\|A\|_{\mathcal A}\le 1} \|A\|_{{\mathcal B}(\ell^p)}\big)
\|x-x'\|_p
\end{equation*}
and
\begin{eqnarray*} %\label{pinvertibility.tm3.pf.eq6}
  & & \|F_p(x')-F_p(x)-\nabla f(x) (x'-x)\|_p  %=\sup_{n\ge 1} \|P_n f(x)-P_nf(x')\|_p
 \le  \big(\sup_{\|A\|_{\mathcal A}\le 1} \|A\|_{{\mathcal B}(\ell^p)}\big)\nonumber\\
 & & \qquad \qquad \times
\big(\sup_{0\le s\le 1} \|\nabla f(x+s(x'-x))-\nabla f(x)\|_{{\mathcal A}}\big) %\nonumber\\
%& & \times
%\big(\sup_{\|A\|_{\mathcal A}\le 1} \|A\|_{{\mathcal B}(\ell^p)}\big)
\|x-x'\|_p
\end{eqnarray*}
for all $x, x'\in \ell^p$. This implies that $F_p$  is continuous on $\ell^p$ and
\begin{equation}\label{pinvertibility.tm3.pf.eq7}\nabla F_p(x)=\nabla f(x)\quad {\rm  for\ all}\ 1\le p\le p_0.\end{equation}
Hence $F_p$ has its gradient $\nabla F_p(x)$ in ${\mathcal A}$.
The  continuity of $\nabla F_p(x), x\in \ell^p$, in ${\mathcal A}$
then follows from the continuity of the gradient $\nabla f(x), x\in \ell^{p_0}$, in ${\mathcal A}$
and  the continuous embedding of $\ell^p$ in $\ell^{p_0}$ for every $1\le p\le p_0$.

Then  we prove the continuous invertibility of the function $F_p$ on $\ell^p, 1\le p\le p_0$.
We notice that the function $F_2$ satisfies all requirements
for the function $f$ in Theorem \ref{pinvertibility.tm}
 because it has   continuous gradient in ${\mathcal A}$ and  strictly monotonic property  by
\eqref{pinvertibility3.tm.eq1}, \eqref{pinvertibility.tm3.pf.eq7} and the assumption $p_0\ge 2$, and it
   satisfies $F_2(0)=0$ by the assumption $f(0)=0$.
Thus
by Theorem \ref{pinvertibility.tm} it suffices to show that
 the function $F_p$ is a continuous extension of the function $F_2$ on
$\ell^p$ with respect to the strong topology when $1\le p<\infty$ and the weak-star topology when $p=\infty$.
The  above extension property follows
as   both $F_2$ and $F_p$ are the restriction of the function $f$
on $\ell^{p_0}$. The continuity of  the function $F_p$ with respect to the strong topology when $1\le p<\infty$ follows as
$F_p$ is continuous on $\ell^p$. To establish the continuity of  the function $F_p$ with respect to the weak-star topology when $p=\infty$,
we notice that  $p_0=\infty$ and $F_p=f$ in this case.
Moreover for $p=\infty$, similar to the argument used to  establish \eqref{pinvertibility.tm.pf.eq9} we obtain
\begin{eqnarray*}
& & |\langle  y, F_p(x)\rangle-\langle y, F_p(x')\rangle|\nonumber\\
&\le &
\big(\sup_{z\in \ell^\infty} \|\nabla f(z)\|_{\mathcal A}\big)
\Big(\|y\|_{1} \sup_{\|A\|_{\mathcal A}\le 1}\|P_{m}  A(I-P_{m+1})\|_{{\mathcal B}(\ell^\infty)}\nonumber\\
 & & \ \ +\|(I-P_m)y\|_{1} \sup_{\|A\|_{\mathcal A}\le 1} \|A\|_{{\mathcal B}(\ell^\infty)}
 \Big)\big( \|x\|_{\infty}+\|x'\|_{\infty}\big)\nonumber\\
 & &\quad  + \big(\sup_{z\in \ell^\infty} \|\nabla f(z)\|_{\mathcal A}\big) \big(\sup_{\|A\|_{\mathcal A}\le 1} \|A\|_{{\mathcal B}(\ell^\infty)}\big)
 \|y\|_{1}
\|P_{m+1}(x-x')\|_{\infty}
\end{eqnarray*}
for all $y\in \ell^1, x, x'\in \ell^\infty$ and $m\ge 1$, which  implies that $F_p$ is
continuous on $\ell^\infty$  with respect to  the weak-star topology.

Finally we prove that for every $1\le p\le p_0$, the inverse of the function $F_p$ has its gradient being bounded and continuous in  ${\mathcal A}$.
Take $x\in \ell^p$, and let  $x_n\in \ell^p\cap \ell^2, n\ge 1$, converge to $x$ with respect to the strong topology of $\ell^p$ when $1\le p<\infty$ and
the weak-star topology of $\ell^\infty$ when $p=\infty$.
Recall that the function $F_p$ is a continuous extension of the function $F_2$ on
$\ell^p$ with respect to the strong topology when $1\le p<\infty$ and the weak-star topology when $p=\infty$.
Then
\begin{equation*}
\lim_{n\to \infty} c^T (F_2(x_n+tc)-F_2(x_n))= c^T (F_p(x+tc)-F_p(x)),
\end{equation*}
where $t\in \RR$ and $c\in \ell^0$ (the space of all finitely supported sequences).
This together with \eqref{pinvertibility3.tm.eq1} and \eqref{pinvertibility.tm3.pf.eq7}
%the strict monotonicity of the function $F_2$ on $\ell^2$
implies that
\begin{eqnarray*}
c^T \nabla F_p(x) c & = &  \lim_{t\to 0} t^{-1}
c^T (F_p(x+tc)-F_p(x))\nonumber\\
& = & \lim_{t\to 0} \lim_{n\to \infty}
 t^{-1} c^T (F_2(x_n+tc)-F_2(x_n))\nonumber\\
& = &  \lim_{t\to 0} \lim_{n\to \infty}
\int_{0}^1 c^T \nabla F_2(x_n+st c) c \ \! ds\ge  m_0 c^T c \quad {\rm for \ all} \ c\in \ell^0.
\end{eqnarray*}
%where $m_0$ is the strict monotonicity constant of the function $f$ on $\ell^2$.
Thus
\begin{equation*}
c^T \nabla F_p(x) c\ge m_0 c^T c \quad {\rm for \ all} \ c\in \ell^2\  {\rm and} \   x\in \ell^p
\end{equation*}
by \eqref{pinvertibility3.tm.eq1}, \eqref{pinvertibility.tm3.pf.eq7}, the density of $\ell^0$ in $\ell^2$,  and
${\mathcal A}$ is a Banach subalgebra of ${\mathcal B}(\ell^2)$.
Hence $\nabla F_p(x)$ is invertible on $\ell^2$  and
$\|(\nabla F_p(x))^{-1}\|_{{\mathcal B}(\ell^2)}\le 1/m_0$ for all $x\in \ell^p$.
Notice that $\nabla F_p(x), x\in \ell^p$, is bounded and continuous in ${\mathcal A}$ by \eqref{pinvertibility.tm3.pf.eq7} and
the assumption that $\nabla f(x), x\in \ell^{p_0}$, is bounded and continuous in  ${\mathcal A}$. Then
 we obtain from the inverse-closedness of the subalgebra ${\mathcal A}$ in ${\mathcal B}(\ell^2)$ that
$(\nabla F_p)^{-1}$ is bounded and continuous in ${\mathcal A}$.
This together with \eqref{inverse.lm.pf.eq8}  proves that the gradient $\nabla (F_p^{-1})$ of the inverse function $F_p^{-1}$ is bounded in ${\mathcal A}$.
 The continuity of $\nabla (F_p^{-1})$ in ${\mathcal A}$
follows from \eqref{inverse.lm.pf.eq8}, the continuity of $\nabla F_p$ in ${\mathcal A}$ and the boundedness of
$\nabla (F_p^{-1})$ in ${\mathcal A}$.

\bigskip

(ii) \quad We recall that the function $F_2$ satisfies all requirements
for the function $f$ in Theorem \ref{pinvertibility.tm}. Then by Theorem \ref{pinvertibility.tm}, there exists
a unique extension $F$ of the function $F_2$ on $\ell^p, 2\le p\le \infty$,  which is continuous with respect to the strong topology of $\ell^p$ when $1\le p<\infty$ and the weak-star topology of $\ell^p$ when $p=\infty$. As shown earlier, $F_p, 2\le p\le p_0$,
are the restriction of the function $f$ on $\ell^{p_0}$ and are continuous with respect to the strong topology of $\ell^p$ when $1\le p<\infty$
 and the weak-star topology of $\ell^p$ when $p=\infty$. Hence $F=F_p$ when $2\le p\le p_0$, which in turn implies that for $p_0<p\le \infty$, the function $F$ is a continuous  extension of the function  $f$ on $\ell^{p_0}$. The Lipschitz conditions \eqref{pinvertibility3.tm.eq2} and \eqref{pinvertibility3.tm.eq3} for the function $F$ follow from \eqref{inverse.lm.pf.eq8}, Theorem \ref{pinvertibility.tm}
 and the inclusion $\ell^2\subset \ell^{p_0}$.
\end{proof}

\section{Error estimate}\label{error.section}

 In this section, we apply % the unique extension theorem (Theorem \ref{pinvertibility.tm}) and
 the invertibility theorem (Theorem  \ref{pinvertibility.tm3}) established in the previous section, and discuss
  the   $\ell^p$-error estimate
for solving the  nonlinear functional equation
\eqref{nfe.def} in the presence of noises in $\ell^p, 1\le p\le \infty$. %, which is highlighted in the introduction section.

\begin{thm}\label{errorestimate.tm}
Let $2\le p_0\le \infty, 1\le p\le p_0$,  ${\mathcal A}$ be a subalgebra of ${\mathcal B}(\ell^q)$ for all $1\le q\le \infty$
with the property that it admits norm control in ${\mathcal B}(\ell^2)$ and
 satisfies \eqref{pinvertibility.tm.eq1}, and let
  $f$ be a continuous function on $\ell^{p_0}$ such that $f(0)=0$,
  its gradient $\nabla f$ is bounded and continuous in ${\mathcal A}$ and satisfies \eqref{pinvertibility3.tm.eq1}.
If $x_0$ and $x_\epsilon$
 are the solutions of the nonlinear functional equation \eqref{nfe.def}
 from the noiseless observed data $y\in \ell^p$ and the noisy observed data $y+\epsilon$ in $\ell^p$ respectively (i.e.,
 $f(x_0)=y$ and $f(x_\epsilon)=y+\epsilon$), then
 \begin{equation}\label{errorestimate.tm.eq1}
\|x_\epsilon-x_0\|_p\le  \big(\sup_{z\in \ell^{p_0}} \|\nabla f^{-1}(z)\|_{\mathcal A}\big)
\big(\sup_{\|A\|_{\mathcal A}\le 1} \|A\|_{{\mathcal B}(\ell^p)}\big)
\|\epsilon\|_p,
\end{equation}
and
 \begin{equation}\label{errorestimate.tm.eq2}
\frac{\|x_\epsilon-x_0\|_p}{\|x_0\|_p}
\le \kappa_{\mathcal A}(f)
\big(\sup_{\|A\|_{\mathcal A}\le 1} \|A\|_{{\mathcal B}(\ell^p)}\big)^2
\frac{\|\epsilon\|_p}{\|y\|_p},
\end{equation}
 where $\kappa_{\mathcal A}(f)= \big(\sup_{z\in \ell^{p_0}} \|\nabla f(z)\|_{\mathcal A}\big)
 \big(\sup_{z\in \ell^{p_0}} \|\nabla f^{-1}(z)\|_{\mathcal A}\big)$.
\end{thm}

\begin{rem}{\rm If ${\mathcal A}={\mathcal B}(\ell^2)$ and
$f$ is a linear function on $\ell^{p_0}$ (i.e, $f(x)=Ax$ for some matrix $A$),
 then
$\kappa_{\mathcal A}(f)$ in Theorem \ref{errorestimate.tm}
 becomes the condition number  for the matrix $A$. So the quantity
$\kappa_{\mathcal A}(f)$   can be thought as the condition number of the function $f$ in the Banach algebra ${\mathcal A}$.
}\end{rem}

\begin{rem} {\rm
Estimates in Theorems \ref{pinvertibility.tm}, \ref{pinvertibility.tm3}
 and \ref{errorestimate.tm}
% \eqref{pinvertibility.tm.eq2}, \eqref{pinvertibility.tm.eq3}, \eqref{pinvertibility3.tm.eq2}, \eqref{pinvertibility3.tm.eq3},
%\eqref{errorestimate.tm.eq1} and \eqref{errorestimate.tm.eq2}
 could be improved by replacing the ${\mathcal A}$-norm by the conventional ${\mathcal B}(\ell^p)$-norm. For instance, we can  replace \eqref{errorestimate.tm.eq1} and \eqref{errorestimate.tm.eq2}
in Theorem \ref{errorestimate.tm}
by the following better  estimates:
 \begin{equation}\label{errorestimate.rem2.eq1}
\|x_\epsilon-x_0\|_p\le  \big(\sup_{z\in \ell^{p_0}} \|\nabla f^{-1}(z)\|_{{\mathcal B}(\ell^p)}\big)
\|\epsilon\|_p,
\end{equation}
and
 \begin{equation}\label{errorestimate.rem.eq2}
\frac{\|x_\epsilon-x_0\|_p}{\|x_0\|_p}
\le  \big(\sup_{z\in \ell^{p_0}} \|\nabla f(z)\|_{{\mathcal B}(\ell^p)}\big)
 \big(\sup_{z\in \ell^{p_0}} \|\nabla f^{-1}(z)\|_{{\mathcal B}(\ell^p)}\big)
\frac{\|\epsilon\|_p}{\|y\|_p}.
\end{equation}
 We select the  suboptimal
estimates in Theorems \ref{pinvertibility.tm}, \ref{pinvertibility.tm3}
 and \ref{errorestimate.tm}
for emphasizing the fundamental role
 the inverse-closed Banach algebra ${\mathcal A}$  has played,
and for increasing  computational feasibility, c.f. the arguments used to prove Theorems \ref{vancittertiteration.tm},
\ref{frirecovery.tm1} and \ref{friidentification.tm}.
%to use ${\mathcal A}$-norm instead of the conventional ${\mathcal B}(\ell^p)$-norm.
On the other hand,  the  replacement  of ${\mathcal A}$-norm by the ordinary ${\mathcal B}(\ell^p)$-norm
does not always work well.
In our study of global linear convergence of the Van-Cittert iteration method (Theorem \ref{vancittertiteration.tm}),
 we cannot replace the ${\mathcal A}$-norm  in \eqref{vancittertiteration.tm.pf.eq13-0}
by the ordinary ${\mathcal B}(\ell^p)$-norm. The reasons are that the differential norm  property
\eqref{paracompactcondition} does not hold when  ${\mathcal A}$ is replaced by ${\mathcal B}(\ell^p), p\ne 2$,
and that the new $b_j$ with the ${\mathcal A}$-norm  replaced by
the ${\mathcal B}(\ell^p)$-norm may not satisfy the crucial iterative properties \eqref{vancittertiteration.tm.pf.eq13} and
\eqref{vancittertiteration.tm.pf.eq14}.
} \end{rem}

\begin{rem}\label{sfm.rem} {\rm The  mean squared error, as a signal fidelity measure,
is widely criticized in the field of signal processing for its failure %serious shortcoming
to deal with
perceptually important signals such as speech and images.
An interesting alternative is the $\ell^p$-norm with $p\ne 2$,
 possibly with adaptive spatial weighting, see the recent review paper \cite{wang09} on the signal fidelity measure.
This is one of our motivations to consider solving the nonlinear functional equation \eqref{nfe.def}
in the presence of noises in $\ell^p$ with $p\ne 2$.
}\end{rem}

\begin{proof}[Proof of Theorem \ref{errorestimate.tm}]
By Theorem \ref{pinvertibility3.tm.eq1},
the function $f$ is invertible on $\ell^p$ and hence the
nonlinear functional equation \eqref{nfe.def}
is solvable in $\ell^p, 1\le p\le p_0$. Thus $x, x_\epsilon\in \ell^p$.

Again by Theorem \ref{pinvertibility3.tm.eq1},
the inverse   of the function $f$ on $\ell^p$ has its gradient bounded and continuous in ${\mathcal A}$.
Thus
\begin{eqnarray*}%\label{errorestimate.tm.pf.eq1}
\|f^{-1}(y_1)-f^{-1}(y_2)\|_p %& \le &  \big(\sup_{z\in \ell^p} \|\nabla f^{-1}(z)\|_{\mathcal A}\big)
%\big(\sup_{\|A\|_{\mathcal A}\le 1} \|A\|_{{\mathcal B}(\ell^p)}\big) \|y_1-y_2\|_p\nonumber\\
 & \le &  \big(\sup_{z\in \ell^{p_0}} \|\nabla f^{-1}(z)\|_{\mathcal A}\big)
\big(\sup_{\|A\|_{\mathcal A}\le 1} \|A\|_{{\mathcal B}(\ell^p)}\big) \|y_1-y_2\|_p
\end{eqnarray*}
for all $y_1, y_2\in \ell^p$.
 Applying the above estimate %\eqref{errorestimate.tm.pf.eq1}
  with $y_1, y_2$ replaced by $y, y+\epsilon\in \ell^p$
proves  the absolute error estimate \eqref{errorestimate.tm.eq1}.

Using the above argument with the function $f^{-1}$ replaced by the function $f$ leads to
\begin{equation}\label{errorestimate.tm.pf.eq2}
\|f(x_1)-f(x_2)\|_p\le \big(\sup_{z\in \ell^{p_0}} \|\nabla f(z)\|_{\mathcal A}\big)
\big(\sup_{\|A\|_{\mathcal A}\le 1} \|A\|_{{\mathcal B}(\ell^p)}\big) \|x_1-x_2\|_p
\end{equation}
for all $x_1, x_2\in \ell^p$. Taking $x_1=x_0$ and $x_2=0$ in \eqref{errorestimate.tm.pf.eq2} and using $f(0)=0$ gives
 \begin{equation}\label{errorestimate.tm.pf.eq3}
\|y\|_p=\|f(x_0)\|_p\le \big(\sup_{z\in \ell^{p_0}} \|\nabla f(z)\|_{\mathcal A}\big)
\big(\sup_{\|A\|_{\mathcal A}\le 1} \|A\|_{{\mathcal B}(\ell^p)}\big) \|x_0\|_p.
\end{equation}
Hence the relative error estimate \eqref{errorestimate.tm.eq2} follows from \eqref{errorestimate.tm.eq1} and \eqref{errorestimate.tm.pf.eq3}.
\end{proof}

\section{Convergence of iteration methods}\label{algorithm.section}

In this section, we consider numerically solving the nonlinear functional  equation \eqref{nfe.def}.
We show in Theorem \ref{vancittertiteration.tm}  that
the Van-Cittert iteration  method has global  exponential convergence in $\ell^p$,  and
in Theorem \ref{newton.tm} that
the quasi-Newton  iteration method has local quadratic convergence in $\ell^p, 1\le p\le \infty$.
The  Van-Cittert method is the {\em only} iteration method, that we found, which has global exponential  convergence in $\ell^p, p\ne 2$,
for solving infinite-dimensional nonlinear functional equation \eqref{nfe.def}.
We are seeking to derive other iteration methods to solve the nonlinear functional
equation \eqref{nfe.def}, that are easily implementable and stable in the  presence of noises,  and  that  have global (or local) convergence in $\ell^p$ with $p\ne 2$.

\subsection{Global linear convergence of the Van-Cittert iteration method}

\begin{thm}\label{vancittertiteration.tm}
Let $2\le p_0\le \infty$, and let ${\mathcal A}$ be a  Banach algebra of infinite matrices with the property that it is
 a subalgebra of ${\mathcal B}(\ell^q)$ for all $1\le q\le \infty$,
  it contains the identity matrix  $I$, it is closed under the transpose operation,
  %$\|A^T\|_{\mathcal A}=\|A\|_{\mathcal A}$ for all $A\in {\mathcal A}$,
   and
it satisfies \eqref{pinvertibility.tm.eq1} and
 the differential norm inequality \eqref{paracompactcondition} for some $C_0\in (0, \infty)$ and $\theta\in [0, 1)$.
 Assume that  $f$ is a continuous function on $\ell^{p_0}$ such that $f(0)=0$,
  its gradient $\nabla f$ is bounded and continuous in ${\mathcal A}$ and satisfies
 the strict monotonicity property \eqref{pinvertibility3.tm.eq1}.
Set $L=\sup_{z\in \ell^{p_0}}\|\nabla f(z)\|_{{\mathcal B}(\ell^2)}$ and let $m_0$ be the strict monotonicity constant in \eqref{pinvertibility3.tm.eq1}.
Given $1\le p\le p_0$ and $y\in \ell^p$,  take  an initial guess $x_0\in \ell^p$ and define
 $x_n\in \ell^p, n\ge 1$, iteratively by
\begin{equation}\label{vancittert.def}
x_{n}=x_{n-1}-\alpha (f(x_{n-1})-y),\ n\ge 1.
\end{equation}
Then for  any  $0<\alpha<m_0/(L+L^2)$, the sequence
$x_n, n\ge 0$, converges exponentially in $\ell^p$ to the unique true solution $f^{-1}(y)$ of the nonlinear
functional equation \eqref{nfe.def}. % from noiseless observed data $y$.
Moreover
\begin{equation}
\|x_n-f^{-1}(y)\|_{p}\le C_1  \|x_0-f^{-1}(y)\|_p\ r^n,\ \ n\ge 0,
\end{equation}
where
 $r>(1-\alpha+ \alpha^2 L+\alpha^2L^2)/(1-\alpha+\alpha m_0)$ and $C_1$ is a positive
 constant depending only on $\alpha, r,  p$ and ${\mathcal A}$.
\end{thm}

\begin{rem}{\rm  We may rewrite
 the  Van-Cittert iteration
\eqref{vancittert.def}   as  $x_n=Tx_{n-1}$, where  $Tx=x-\alpha (f(x)-y), x\in \ell^p$.
Thus  the Van-Cittert iteration \eqref{vancittert.def} is a Picard iteration to solve the nonlinear functional equation \eqref{nfe.def}.
The logic behind the Van-Cittert iteration
\eqref{vancittert.def}
 is that we define
 the new guess  by adding the difference $y-f(x_0)$   weighted by a relaxation factor $\alpha$ to the initial guess $x_0$
 if  the difference $y-f(x_0)$ between the observation $f(x_0)$ of the initial guess $x_0$  and the observed data
 $y$ is above a threshold,
and  we repeat the above iteration  until
the difference  between the observation  of the  guess  and the observed data
is below a threshold.
The  Van-Cittert iteration \eqref{vancittert.def} can be easily implemented and
 is fairly stable in the presence of noises in $\ell^p$, see  Section \ref{fri3.subsection} for  numerical simulations.
}\end{rem}

It has been established in \cite{suncasp05, suntams07} that  the Jaffard class ${\mathcal J}_\beta(\Lambda)$ in \eqref{jaffard.def}
satisfies the
paracompact condition \eqref{weakparacompactcondition}.
We can mimic the proof there  to  show that
 the Jaffard class ${\mathcal J}_\beta(\Lambda)$
satisfies the  differential norm inequality \eqref{paracompactcondition}.
Then  we have the following corollary by  Theorem  \ref{vancittertiteration.tm}.

\begin{cor}\label{specialalgorithm.cor}
  Let $\beta>d, \Lambda$ be a relatively-separated subset of $\Rd$, and
 let $p_0, p, f, L, \alpha$ be as in Theorem \ref{vancittertiteration.tm} with
 the Banach algebra ${\mathcal A}$ replaced by the Jaffard class ${\mathcal J}_\beta(\Lambda)$.
Then the sequence $x_n, n\ge 0$, in  the Van-Cittert iteration \eqref{vancittert.def}
converges exponentially in $\ell^p(\Lambda)$ to the true  solution
of the nonlinear functional equation \eqref{nfe.def} with the observed data $y$.
\end{cor}

\begin{proof}[Proof of Theorem \ref{vancittertiteration.tm}]\
By Proposition \ref{wienerlmforinfinitematrices.prop},  ${\mathcal A}$ is an inverse-closed Banach subalgebra of ${\mathcal B}(\ell^2)$
that admits norm control.
This together with  Theorem \ref{pinvertibility.tm3} implies that $f$ is invertible on $\ell^p$. Hence $x^*:=f^{-1}(y)\in \ell^p$.
%The existence of such a sequence $x^*$ follows from  Theorem \ref{pinvertibility.tm}.
For the Van-Cittert iteration \eqref{vancittert.def},
\begin{eqnarray*} \label{vancittertiteration.tm.pf.eq1}% & &
x_{n+1}-x^*
%\nonumber\\
% & = &  x_n-x^*-\alpha (f(x_n)-f(x^*))\nonumber\\
%& = & (1-\alpha) (x_n-x^*)+\alpha (x_n-x^*-f(x_n)+f(x^*))\\
& = & (1-\alpha) (x_n-x^*)+\alpha \big(x_{n+1}- x^*-f(x_{n+1})+f(x^*)\big )\nonumber\\
& & -
\alpha \big(x_{n+1}-x_n-f(x_{n+1})+f(x_n)\big)\nonumber\\
& = & (1-\alpha) (x_n-x^*)+\alpha \big(x_{n+1}- x^*-f(x_{n+1})+f(x^*)\big)\nonumber\\
& & +\alpha^2\big (f(x_{n})-f(x^*)\big)+
\alpha \Big(f\big(x_n-\alpha\big(f(x_n)-f(x^*)\big)\big)-f(x_n)\Big).
\end{eqnarray*}
This leads to the following crucial equation:
\begin{equation} \label{vancittertiteration.tm.pf.eq2}
A_n (x_{n+1}-x^*)= B_n (x_n-x^*)  \quad{\rm for\  all} \  n\ge 0,
\end{equation}
where
\begin{equation} \label{vancittertiteration.tm.pf.eq3}
A_n=(1-\alpha)I+\alpha \int_0^1 \nabla f(t x_{n+1}+(1-t)x^*)dt,
\end{equation}
and
\begin{eqnarray}\label{vancittertiteration.tm.pf.eq4}
B_n & = & (1-\alpha)I+\alpha^2  \Big(I-\int_0^1 \nabla f\big(x_n- s \alpha(f(x_n)-f(x^*))\big) ds\Big)\nonumber\\
& & \times\Big(\int_0^1 \nabla f(tx_n+(1-t) x^*) dt\Big), \quad n\ge 0.
\end{eqnarray}

For infinite matrices $A_n$ and $B_n, n\ge 0$, we obtain from   \eqref{pinvertibility3.tm.eq1}, \eqref{vancittertiteration.tm.pf.eq3}
and \eqref{vancittertiteration.tm.pf.eq4}  that
\begin{equation} \label{vancittertiteration.tm.pf.eq5}
\left\{\begin{array}
{l}\|A_n\|_{{\mathcal B}(\ell^2)}\le 1+(L-1)\alpha, \\ %\sup_{z\in \ell^\infty} \|\nabla f(z)\|_{{\mathcal B}(\ell^2)},\\
\|A_n\|_{{\mathcal A}}\le (1-\alpha)\|I\|_{\mathcal A}+\alpha  \sup_{z\in \ell^{p_0}} \|\nabla f(z)\|_{{\mathcal A}}, \\
c^T A_n c\ge (1-\alpha +\alpha m_0) c^Tc\quad {\rm for \ all} \  c\in \ell^2,
\end{array}\right.
\end{equation}
and
\begin{equation} \label{vancittertiteration.tm.pf.eq6}
\left\{\begin{array}{l}
\|B_n\|_{{\mathcal B}(\ell^2)}\le 1-\alpha+ \alpha^2 L(1+L), \\
\|B_n\|_{{\mathcal A}}\le (1-\alpha)\|I\|_{\mathcal A}+\alpha^2 \big(
\sup_{z\in \ell^{p_0}} \|\nabla f(z)\|_{{\mathcal A}} \big)\\
\qquad \qquad \quad \times \big( \|I\|_{\mathcal A} +\sup_{z\in \ell^{p_0}} \|\nabla f(z)\|_{{\mathcal A}}\big).
\end{array}\right.\end{equation}
 For every $n\ge 0$,
 it follows from \eqref{vancittertiteration.tm.pf.eq5} that
  $A_n$ is invertible in $\ell^2$ and
  the operator norm  $\|(A_n)^{-1}\|_{{\mathcal B}(\ell^2)}$
  of its inverse  $(A_n)^{-1}$ is bounded by $(1-\alpha +\alpha m_0)^{-1}$.
The above property for infinite matrices $A_n, n\ge 0$, together with
\eqref{vancittertiteration.tm.pf.eq6}  and
   the norm control property of the Banach algebra ${\mathcal A}$ given in Proposition \ref{wienerlmforinfinitematrices.prop}, implies that
\begin{equation}\label{vancittertiteration.tm.pf.eq9}
\|A_n^{-1} B_n\|_{{\mathcal B}(\ell^2)}=\frac{1-\alpha+ \alpha^2 L(1+L)}{1-\alpha+\alpha m_0}:=r_1<1
\end{equation}
and
\begin{equation}\label{vancittertiteration.tm.pf.eq10}
\|A_n^{-1} B_n\|_{{\mathcal A}}\le D_1<\infty\quad {\rm for\ all} \ n\ge 0,\end{equation}
 where $D_1$ is a positive constant independent of $n\ge 0$.

Now we rewrite \eqref{vancittertiteration.tm.pf.eq2} as
\begin{equation*} %\label{vancittertiteration.tm.pf.eq11}
x_{n+1}-x^*= A_n^{-1} B_n (x_n-x^*).
\end{equation*}
Applying the above formula iteratively gives
\begin{equation*}%\label{vancittertiteration.tm.pf.eq12}
x_{n+1}-x^*=A_n^{-1} B_n A_{n-1}^{-1} B_{n-1}\cdots A_0^{-1} B_0 (x_0-x^*).
\end{equation*}
Then the proof of  the exponential convergence of  $x_n, n\ge 0$, to $x^*$ reduces to showing that
there exist positive constant $C\in (0, +\infty)$ and $r\in (0,1)$ such that
\begin{equation}\label{vancittertiteration.tm.pf.eq13-}
\|A_n^{-1} B_n A_{n-1}^{-1} B_{n-1}\cdots A_0^{-1} B_0 \|_{\mathcal A}\le C r^n, \ \ n\ge 0.
\end{equation}
For $\theta \in (0, 1)$, set
\begin{equation}\label{vancittertiteration.tm.pf.eq13-0}
b_j= (2 C_0)^{1/\theta} r_1^{-j} \sup_{k\ge 0} \|A_{j+k-1}^{-1} B_{j+k-1} \cdots A_k^{-1} B_k\|_{{\mathcal A}},\end{equation}
where $C_0$ and $r_1$ are the constants in \eqref{paracompactcondition} and \eqref{vancittertiteration.tm.pf.eq9} respectively.
Then we obtain from  \eqref{vancittertiteration.tm.pf.eq9},
 \eqref{vancittertiteration.tm.pf.eq10}  and the differential norm property \eqref{paracompactcondition} that
\begin{eqnarray}\label{vancittertiteration.tm.pf.eq13}
b_{2j}& = &  (2C_0)^{1/\theta} r_1^{-2j} \sup_{k\ge 0}
\|A_{2j+k-1}^{-1} B_{2j+k-1} \cdots A_k^{-1} B_k\|_{{\mathcal A}}\nonumber\\
& \le &  (2 C_0)^{(1+\theta)/\theta} r_1^{-2j}
\max\big( \|A_{2j+k-1}^{-1} B_{2j+k-1}\cdots A_{j+k}^{-1} B_{j+k}\|_{{\mathcal A}}^{1+\theta},\nonumber\\
& & \qquad
\|A_{j+k-1}^{-1} B_{j+k-1}\cdots A_k^{-1} B_k\|_{{\mathcal A}}^{1+\theta}\big)\nonumber\\
& & \times \max\big( \|A_{2j+k-1}^{-1} B_{2j+k-1}\cdots A_{j+k}^{-1} B_{j+k}\|_{{\mathcal B}(\ell^2)}^{1-\theta},\nonumber\\
& & \qquad
\|A_{j+k-1}^{-1} B_{j+k-1}\cdots A_k^{-1} B_k\|_{{\mathcal B}(\ell^2)}^{1-\theta}\big)\nonumber\\
&  \le &    b_j^{1+\theta}
\end{eqnarray}
and
\begin{eqnarray}\label{vancittertiteration.tm.pf.eq14}
b_{2j+1} & \le  & (2C_0)^{1/\theta} r_1^{-2j-1} \sup_{k\ge 0}
\|A_{2j+k}^{-1} B_{2j+k}\|_{\mathcal A} \| A_{2j+k-1}^{-1} B_{2j+k-1} \cdots A_k^{-1} B_k\|_{{\mathcal A}}\nonumber\\
& \le &  (D_1/r_1) b_{2j}\le (D_1/r_1) b_j^{1+\theta}
\end{eqnarray}
for all $j\ge 1$.
Applying \eqref{vancittertiteration.tm.pf.eq13} and \eqref{vancittertiteration.tm.pf.eq14} repeatedly yields
\begin{eqnarray*}
b_n & \le  & (D_1/r)^{\epsilon_0} b_{\sum_{i=1}^l \epsilon_i 2^{i-1}}^{1+\theta} \le  \cdots\le  (D_1/r)^{\sum_{i=0}^{l-1} \epsilon_i (1+\theta)^i} b_{\epsilon_l}^{(1+\theta)^l}\nonumber\\
& \le &  (D_1/r)^{(1+\theta)^l/\theta} ( (2C_0)^{1/\theta} D_1/r)^{(1+\theta)^l}\le (2C_0D_1/r_1)^{(1+\theta)n^{\log_2(1+\theta)}/\theta},
\end{eqnarray*}
where $n=\sum_{i=0}^l \epsilon_i 2^i$  with $\epsilon_i\in \{0, 1\}$ and $\epsilon_l=1$.
Therefore for $\theta\in (0, 1)$,
\begin{eqnarray}\label{vancittertiteration.tm.pf.eq15}
\|A_n^{-1} B_n A_{n-1}^{-1} B_{n-1}\cdots A_0^{-1} B_0 \|_{\mathcal A}   & \le &    (2C_0)^{-1/\theta} r_1^n b_n\nonumber\\
& \le &
(2C_0)^{-1/\theta} r_1^n(C_0D_1/r_1)^{(1+\theta)n^{\log_2(1+\theta)}/\theta}
\end{eqnarray}
for all $n\ge 1$.

For $\theta=0$, we can mimic the above argument to obtain that
\begin{equation}\label{vancittertiteration.tm.pf.eq16}
\|A_n^{-1} B_n A_{n-1}^{-1} B_{n-1}\cdots A_0^{-1} B_0 \|_{\mathcal A}  \le
 r_1^n(2C_0D_1/r_1)^{1+\log_2 n} \quad {\rm for \ all}\ n\ge 1.
\end{equation}
 Hence \eqref{vancittertiteration.tm.pf.eq13-} follows  from \eqref{vancittertiteration.tm.pf.eq15}
and \eqref{vancittertiteration.tm.pf.eq16}
 by letting  $r\in (r_1, 1)$,
and the exponential convergence of $x_n$ in $\ell^p$ is established.
\end{proof}

\subsection{Local quadratic convergence of the quasi-Newton method}
In this subsection, we show that
the quasi-Newton iterative method \eqref{newtonalgorithm.def} has local R-quadratic convergence in $\ell^p$.

\begin{thm}
\label{newton.tm}
Let $2\le p_0\le \infty$,  let ${\mathcal A}$ be
a subalgebra of ${\mathcal B}(\ell^q)$ for all $1\le q\le \infty$ and an  subalgebra of ${\mathcal B}(\ell^2)$ that admits norm control,
 and  satisfy \eqref{pinvertibility.tm.eq1}, and let $f$ be a continuous function on $\ell^{p_0}$ such that $f(0)=0$,
  its gradient $\nabla f$ is bounded and continuous in ${\mathcal A}$, and satisfies
  the strictly monotonic property
  \eqref{pinvertibility3.tm.eq1} and
\begin{equation}\label{newton.tm.eq1}
\|\nabla f(x)-\nabla f(x')\|_{{\mathcal A}}\le C_1 \|x-x'\|_{p_0} \quad {\rm for \ all} \  x, x'\in \ell^{p_0}
\end{equation}
where $C_1\in (0, \infty)$. Given an observed data $y\in \ell^p$ and an initial guess $x_0\in \ell^p, 1\le p\le p_0$, define
the quasi-Newton iteration  by
\begin{equation}\label{newtonalgorithm.def}
 x_{n+1}=x_{n}- (\nabla f(x_{n}))^{-1} (f(x_{n})-y),\quad  n\ge 1.
\end{equation}
Then $x_n, n\ge 1$, converges R-quadratically  to  the solution $f^{-1}(y)$
of the nonlinear functional equation \eqref{nfe.def} with the observed data $y$ as there exists a positive constant $C$ such that
\begin{equation}
\|x_{n+1}-x_n\|_p\le C \|x_n-x_{n-1}\|_p^2\quad {\rm for \ all} \  n\ge 1,\end{equation}
provided that the initial guess $x_0$ is sufficiently close to the true solution $f^{-1}(y)$.
 \end{thm}

By Theorems \ref{pinvertibility.tm3} and
\ref{newton.tm}, we have the following corollary.

\begin{cor}\label{specialnewtonalgorithm.cor}
  Let $\beta>d, \Lambda$ be a relatively-separated subset of $\Rd$,
   ${\mathcal A}$ be the Jaffard class ${\mathcal J}_\beta(\Lambda)$,
 and let
$p_0$ and $f$ be  as in Theorem \ref{newton.tm}.
Then  the quasi-Newton iteration \eqref{newtonalgorithm.def}
has local R-quadratical convergence in $\ell^p, 1\le p\le p_0$.
\end{cor}

\begin{proof} [Proof of Theorem \ref{newton.tm}]
By \eqref{newton.tm.eq1} and  \eqref{newtonalgorithm.def},
\begin{equation*}
\|x_1-x_0\|_p\le \big(\sup_{z\in \ell^{p_0}} \|(\nabla f(z))^{-1}\|_{\mathcal A}\big)
\big(\sup_{\|A\|_{\mathcal A}\le 1}\|A\|_{{\mathcal B}(\ell^p)}\big)
 \|f(x_0)-y\|_p
\end{equation*}
and
\begin{eqnarray*}   \|x_{n+1}-x_n\|_p  &  = &   \big\| (\nabla f(x_n))^{-1} (f(x_n)-y)\big\|_p\nonumber\\
& \le &   \big(\sup_{z\in \ell^{p_0}} \|(\nabla f(z))^{-1}\|_{\mathcal A}\big)
\big(\sup_{\|A\|_{\mathcal A}\le 1}\|A\|_{{\mathcal B}(\ell^p)}\big)\\
& & \times \|  f(x_n)-f(x_{n-1})-\nabla f(x_{n-1}) (x_n-x_{n-1})\big\|_p\nonumber\\
& \le & C_1\big(\sup_{z\in \ell^{p_0}} \|(\nabla f(z))^{-1}\|_{\mathcal A}\big)
\big(\sup_{\|A\|_{\mathcal A}\le 1}\|A\|_{{\mathcal B}(\ell^p)}\big)
 \|x_{n}-x_{n-1}\|_p^2
\end{eqnarray*}
for all $n\ge 1$. This proves the local R-quadratic convergence
of the quasi-Newton iteration method.
\end{proof}

\section{Instantaneous companding and  average sampling}\label{fri.section}

In this section, the theory established in  previous sections for the nonlinear functional equation
\eqref{nfe.def} will be used  to handle the  nonlinear sampling  procedure
 \eqref{sfpsi.def} of instantaneous companding and subsequently average sampling. In the first subsection,
we  apply the invertibility theorem (Theorem \ref{pinvertibility.tm3}) to establish  $\ell^p$-stability
of the nonlinear sampling  procedure
 \eqref{sfpsi.def},  while in the second subsection we use
the Van-Cittert iteration method in Theorem \ref{vancittert.def}
and the  quasi-Newton iteration method
in Theorem \ref{newton.tm} to recover
  signals  from
 averaging samples of their instantaneous companding.
We present some numerical simulations  in the last subsection to demonstrate the stable signal recovery in the presence of bounded noises.

\subsection{$\ell^p$-stability of the nonlinear sampling procedure \eqref{sfpsi.def}}

We say that
 $\Phi=(\phi_\lambda)_{\lambda\in \Lambda}$  is a {\em Riesz basis} for the space $V_p(\Phi)$  in \eqref{vp.def} if
       \begin{equation}\label{riesz.def}
0< \inf_{\|c\|_p=1} \|c^T\Phi\|_p\le \sup_{\|c\|_p=1} \|c^T\Phi\|_p<\infty.
  \end{equation}
Given two closed subspaces $V$ and $W$ of $L^2$, define
the {\em gap}
 $\delta (V, W)$ from  $V$ to $W$ by
\begin{equation}\label{angle.def}
\delta (V, W)=\sup_{\|h\|_2\le 1, h\in V} \inf_{\tilde h\in W} \|h-\tilde h\|_2=
\sup_{\|h\|_2\le 1, h\in V} \|h-P_W h\|_2,
\end{equation}
where $P_W$ is the projection operator onto $W$.

\begin{thm}\label{frirecovery.tm1}
Let  $\Lambda, \Gamma$ be  two relatively-separated subsets of $\Rd$, and let
 $\Phi:=(\phi_\lambda)_{\lambda\in \Lambda}$
  and $\Psi:=(\psi_\gamma)_{\gamma\in \Gamma}$  have polynomial decay
of order $\beta>d$, i.e.,
$$\sup_{\lambda\in \Lambda, x\in \Rd} |\phi(x)| (1+|x-\lambda|)^\beta +
\sup_{\gamma\in \Gamma, x\in \Rd}|\psi_\gamma (x)| (1+|x-\gamma|)^\beta<\infty,$$
 and
generate Riesz bases for $V_2(\Phi)$
    and $V_2(\Psi)$ %\subset L^2$
    respectively.
    If the gap $\delta(V_2(\Phi), V_2(\Psi))$ from $V_2(\Phi)$ to $V_2(\Psi)$ is strictly less than one,
 %\begin{equation} \label{asmonotonicity.tm.eq5}
%r\|f\|_2^2\le \|P_{V_2(\Psi)} f\|_2\le \|f\|_{2}^2, \quad \
%f\in V_2(\Phi)\end{equation}
%for some positive constant $r\in (0,1)$,
 and the companding function  $F$ is a continuously differential function on $\RR$ satisfying $F(0)=0$ and
\begin{equation}\label{frirecovery.tm.eq4}
\mu:=\sup_{t\in \RR} |1-mF'(t)|<
\frac{\sqrt{1-(\delta(V_2(\Phi), V_2(\Psi)))^2}}{ \delta(V_2(\Phi), V_2(\Psi))+\sqrt{1-(\delta(V_2(\Phi), V_2(\Psi)))^2}}
%\big(1+\tan \theta (V_2(\Phi), V_2(\Psi))\big)^{-1}
\end{equation}
for some nonzero constant $m$, then
 the nonlinear sampling procedure $S_{F, \Psi}$ in \eqref{sfpsi.def}
 is stable on $V_p(\Phi), 1\le p\le \infty$, i.e.,
 there exist positive constants $C_1$ and $C_2$ such that
\begin{equation}\label{frirecovery.tm.stability}
 C_1 \|h_1-h_2\|_p\le \| S_{F, \Psi}(h_1) - S_{F, \Psi} (h_2)\|_p\le C_2\|h_1-h_2\|_p
\end{equation}
for all $h_1, h_2\in  V_p(\Phi)$.
\end{thm}

\begin{rem}{\rm
We  remark that, in Theorem \ref{frirecovery.tm1},
 the assumption that the gap $\delta(V_2(\Phi), V_2(\Psi))$ from $V_2(\Phi)$ to $V_2(\Psi)$  is strictly less than one
is a necessary and sufficient condition for the $\ell^2$-stability of the   sampling procedure \eqref{sfpsi.def}
 without instantaneous companding (i.e., $F(t)\equiv t$) on $V_2(\Phi)$.
% This is because it is equivalent to
% the existence of positive constants $A$ and $B$ such that
%\begin{equation}A\|f\|_2\le \|\langle f, \Psi\rangle\|_{\ell^2(\Gamma)}\le B \|f\|_2\quad {\rm for \ all}\ f\in V_2(\Phi),
%\end{equation}
The above equivalence follows from
 \eqref{frirecovery.lem1.pf.eq2}, \eqref{frirecovery.lem1.pf.eq4} and
 \eqref{frirecovery.tm.pf.eq4}. For the sampling procedure \eqref{sfpsi.def}
 without instantaneous companding, the readers may refer to \cite{bns09, sunsiam06} and references therein.
}
\end{rem}

As a stable nonlinear sampling procedure is one-to-one,  from Theorem \ref{frirecovery.tm1} we obtain the  uniqueness
of the nonlinear sampling procedure \eqref{sfpsi.def},
which is established in  \cite{Faktor10} for $p=2$ when
$\mu<\frac{1-\delta(V_2(\Phi), V_2(\Psi))}{1+\delta(V_2(\Phi), V_2(\Psi))}$, a stronger assumption on $F$ than \eqref{frirecovery.tm.eq4}
in Theorem \ref{frirecovery.tm1}.

\begin{cor}
Let  $
 \Phi, \Psi$  and  $F$ be as in Theorem \ref{frirecovery.tm1}. Then any signal $h\in V_p(\Phi), 1\le p\le \infty$, is uniquely determined by its  nonlinear sampling
data $\langle F(h), \Psi\rangle$.
\end{cor}

Notice that $\delta(V, V)=0$ for any  closed subspace  $V$ of $L^2$. Then applying Theorem \ref{frirecovery.tm1} with
 the generator $\Phi$
of the space $V_2(\Phi)$  being taken as the average sampler $\Psi$ in the nonlinear sampling procedure \eqref{sfpsi.def}, we obtain
the following result.

\begin{cor} \label{phi=psi.cor}
Let  $\Lambda$ be a relatively-separated subset of $\Rd$, and let
 $\Phi:=(\phi_\lambda)_{\lambda\in \Lambda}$
   have polynomial decay
of order $\beta>d$ and
generate a Riesz basis for $V_2(\Phi)$.
    If   the companding function  $F$ has its derivative bounded away from zero and infinity, then
 the nonlinear sampling procedure $S_{F, \Psi}$  in \eqref{sfpsi.def}
  with $\Psi=\Phi$
 is stable on $V_p(\Phi), 1\le p\le \infty$.
\end{cor}

\begin{rem}{\rm Notice that
 \begin{eqnarray*}
 \big(\inf_{s\in \RR} F'(s)\big) \|c_1-c_2\|_2^2 & \le &
(c_1-c_2)^T (S_{F, \Phi}(c_1^T\Phi)-S_{F, \Phi} (c_2^T\Phi))\\
& =& \int_{\Rd}\big( F(c_1^T\Phi(t))-F(c_2^T\Phi(t))\big) (c_1^T\Phi(t)-c_2^T\Phi(t)\big)dt\nonumber\\
 & \le &  \big(\sup_{s\in \RR} F'(s)\big) \|c_1-c_2\|_2^2
\end{eqnarray*}
for all $c_1, c_2\in V_2(\Phi)$.
Hence  the stability conclusion in Corollary \ref{phi=psi.cor} holds for $p=2$
 without the polynomial decay assumption on $\Phi$. This is established for  signals living in  the Paley-Wiener space %$B_{\pi\Omega}$
%(the space of  all
%   square-integrable functions bandlimited to $[-\pi\Omega, \pi\Omega]$)
 \cite{landau61, sandberg94} or a    shift-invariant space \cite{Faktor10}.
}\end{rem}

%\subsection{Proof of  Theorem \ref{frirecovery.tm}}
Given relatively-separated subsets $\Lambda$ and $\Gamma$, let the Jaffard class ${\mathcal J}_\beta(\Gamma, \Lambda)$
be the  family of infinite matrices $A:=(a(\gamma, \lambda))_{\gamma\in \Gamma, \lambda\in \Lambda}$ such that
$$\|A\|_{{\mathcal J}_\beta(\Gamma, \Lambda)}:=\sup_{\gamma\in \Gamma, \lambda} (1+|\gamma-\lambda|)^\beta  |a(\gamma, \lambda)| <\infty.$$
%Clearly ${\mathcal J}_\beta (\Lambda, \Lambda)$ is the same as the Jaffard class ${\mathcal J}_\beta(\Lambda)$.
 Given two vectors $\Phi:=(\phi_\lambda)_{\lambda\in \Lambda}$
  and $\Psi:=(\psi_\gamma)_{\gamma\in \Gamma}$  of square-integrable functions, we define their {\em inter-correlation matrix}
  $A_{\Phi, \Psi}$ by
%\begin{equation}
$A_{\Phi, \Psi}=\big(\langle \phi_{\lambda}, \psi_\gamma\rangle\big )_{\lambda\in \Lambda, \gamma\in \Gamma}$.
To prove
Theorem \ref{frirecovery.tm1}, we need a technical lemma.

%\end{equation}

\begin{lem}\label{frirecovery.lem1}
 Let $\Lambda, \Gamma$ be relatively-separated subsets of $\Rd$,
 $\Phi:=(\phi_\lambda)_{\lambda\in \Lambda}$
  and $\Psi:=(\psi_\gamma)_{\gamma\in \Gamma}$  have polynomial decay
of order $\beta>d$ and
generate Riesz bases of $V_2(\Phi)$
    and $V_2(\Psi)$ %\subset L^2$
    respectively, and let the function $F$  satisfy $F(0)=0$ and have its derivative being continuous and bounded.
Then
\begin{itemize}
\item [{(i)}]  $\Phi$ and $\Psi$ generate Riesz bases for $V_p(\Phi)$ and $V_p(\Psi)$ respectively, where $1\le p\le\infty$.

\item [{(ii)}] The function $f_{F, \Phi, \Psi}$ in \eqref{ffphipsi.def} is well-defined on $\ell^p$ for all $1\le p\le \infty$, and
has its gradient $\nabla f_{F, \Phi, \Psi}$ being continuous and bounded in the Jaffard class ${\mathcal J}_\beta(\Gamma, \Lambda)$.
\end{itemize}
\end{lem}

\begin{proof} %[Proof of Theorem \ref{frirecovery.tm}]
(i) \quad We follow the arguments in \cite{sunaicm08}.  By the polynomial decay property for $\Phi$ and relatively-separatedness of the index set $\Lambda$, we have that
\begin{equation} \label{frirecovery.lem1.pf.eq1}
\sup_{\|c\|_p=1} \|c^T\Phi\|_p<\infty. %\quad {\rm and} \quad \sup_{\|d\|_p=1} \|d^T\Psi\|_p<\infty .
\end{equation}
By  Riesz basis property for $V_2(\Phi)$ and the polynomial decay property for $\Phi$,
%can be reformulated as
% positive-definiteness properties for
%the autocorrelation matrices $A_{\Phi, \Phi}$ and $A_{\Psi, \Psi}$:
%\begin{equation}\label{frirecovery.tm.pf.eq1}
% 0<\inf_{\|c\|_2=1} c^T A_{\Phi, \Phi} c\le \sup_{\|c\|_2=1} c^T A_{\Phi, \Phi} c<\infty,
%\end{equation}
\begin{equation}\label{frirecovery.lem1.pf.eq2} (A_{\Phi, \Phi})^{-1}\in {\mathcal B}(\ell^2(\Lambda)) \quad {\rm and} \quad A_{\Phi, \Phi}\in {\mathcal J}_\beta(\Lambda).
\end{equation}
Hence  $(A_{\Phi, \Phi})^{-1}$ belong to the same
Jaffard class ${\mathcal J}_\beta(\Lambda)$ by
 the inverse-closedness of the Jaffard class ${\mathcal J}_\beta(\Lambda)$ in ${\mathcal B}(\ell^2(\Lambda))$, which
 together with the polynomial decay property of $\Phi$ implies that the dual Riesz basis
  $(A_{\Phi, \Phi})^{-1} \Phi$ has the polynomial decay of same order $\beta$.
Hence  $c=\langle c^T\Phi, (A_{\Phi, \Phi})^{-1} \Phi\rangle$ for all $c\in \ell^p$, which yields
 \begin{equation} \label{frirecovery.lem1.pf.eq3}
\inf_{\|c\|_p=1} \|c^T\Phi\|_p>0.
\end{equation}
Combining \eqref{frirecovery.lem1.pf.eq1} and \eqref{frirecovery.lem1.pf.eq3} proves the Riesz property for $V_p(\Phi)$.

The Riesz property for $V_p(\Psi)$ can be proved by using the  same argument to establish
\eqref{frirecovery.lem1.pf.eq1} and \eqref{frirecovery.lem1.pf.eq3} except with  $\Phi$ replaced by $\Psi$, $\Lambda$ by $\Gamma$,
\eqref{frirecovery.lem1.pf.eq2}
by
 \begin{equation}\label{frirecovery.lem1.pf.eq4}
( A_{\Psi, \Psi})^{-1}\in {\mathcal B}(\ell^2(\Gamma)) \quad {\rm and} \quad   A_{\Psi, \Psi}\in {\mathcal J}_\beta(\Gamma),
\end{equation}
which follows from   Riesz basis  property for $V_2(\Psi)$  and the polynomial decay property for $\Psi$.

(ii)\quad For $c=(c(\lambda))_{\lambda\in \Lambda}\in \ell^p(\Lambda)$, write $f_{F, \Phi, \Psi}(c)= (d(\gamma))_{\gamma\in \Gamma}$.
Then %it follows that
\begin{eqnarray*}
|d(\gamma)| & \le &  \int_{\Rd} |F(c^T\Phi(t)) \psi_{\gamma}(t)| dt \\
& \le & \|F'\|_\infty
 \|\Phi\|_{\infty, \beta}
\|\Psi\|_{\infty, \beta} \sum_{\lambda\in \Lambda} |c(\lambda)| \int_{\Rd}  (1+|t-\lambda|)^{-\beta} (1+|t-\gamma|)^{-\beta} dt\\
& \le &  C \|F'\|_\infty \|\Phi\|_{\infty, \beta}
\|\Psi\|_{\infty, \beta}
\sum_{\lambda\in \Lambda} |c(\lambda)|  (1+|\lambda-\gamma|)^{-\beta}\quad {\rm for \ all} \ \gamma\in \Gamma,
\end{eqnarray*}
where $C$ is  a positive constant. Thus
 $ f_{F, \Phi, \Psi}(c)\in \ell^p(\Gamma)$ for all $c\in \ell^p(\Lambda)$, which proves that the function $f_{F, \Phi, \Psi}$ is well-defined on $\ell^p(\Lambda)$.

 By direct calculation, we have that
 \begin{equation} \label{frirecovery.lem1.pf.eq5} \nabla f_{F, \Phi, \Psi}(c)= (\langle F'(c^T\Phi) \phi_\lambda, \psi_\gamma\rangle)_{\gamma\in \Gamma, \lambda\in \Lambda}\quad  {\rm for \ all} \ c\in\ell^p,
 \end{equation}
and
    $\lim_{\tilde c\to c \ {\rm in} \ \ell^p} \|{\tilde c}^T\Phi- c^T\Phi\|_\infty=0$.
 Thus
% en it follows from the  boundedness and continuity of the function $F$,
%  the polynomial decay property for $\Phi$ and $\Psi$ and the relatively-separatedness for $\Lambda$ and $\Gamma$ that
   $\nabla f_{F, \Phi, \Psi}(c), c\in \ell^p$, is bounded and continuous
 in %the Jaffard class
  ${\mathcal J}_\beta(\Gamma, \Lambda)$.
 \end{proof}

\begin{proof}[Proof of Theorem \ref{frirecovery.tm1}] By the polynomial decay property for $\Phi$ and $\Psi$,
$A_{\Phi, \Phi}\in {\mathcal J}_\beta(\Lambda, \Lambda), A_{\Psi, \Psi}\in {\mathcal J}_{\beta}(\Gamma, \Gamma),
A_{\Psi, \Phi}\in {\mathcal J}_\beta(\Gamma, \Lambda)$ and $A_{\Phi, \Psi}\in {\mathcal J}_\beta(\Lambda, \Gamma)$.
From the argument in Lemma \ref{frirecovery.lem1}, we have  that
$(A_{\Phi, \Phi})^{-1}\in {\mathcal J}_\beta(\Lambda, \Lambda)$ and
 $(A_{\Psi, \Psi})^{-1}\in {\mathcal J}_\beta(\Gamma, \Gamma)$.
Thus
\begin{equation}\label{frirecovery.tm.pf.eq1}
A_{\Phi, \Psi}  (A_{\Psi, \Psi})^{-1}A_{\Psi, \Phi}\in {\mathcal J}_\beta(\Lambda, \Lambda).\end{equation}
By the  definition of the gap $\delta (V_2(\Phi), V_2(\Psi))$  from $V_2(\Phi)$ to $V_2(\Psi)$,
%\begin{equation} (1-\sin\theta(V_2(\Phi), V_2(\Psi)) x^T A_{\Phi, \Phi} x \le x^T A_{ \Phi,  \Psi} (A_{\Psi, \Psi})^{-1} A_{\Psi,  \Phi} x\le x^T A_{\Phi, \Phi} x
%\end{equation}
%%for all $f\in V_2(\Phi)$, or equivalently
\begin{equation} \label{frirecovery.tm.pf.eq2}
%%& &
\big(1-\big( \delta(V_2(\Phi), V_2(\Psi))\big)^2\big)\ \!
c^T A_{\Phi, \Phi} c %\\
%& = & (1-\sin\theta(V_2(\Psi), V_2(\Phi))^2\|f\|_2^2\nonumber\\
% & \le &  %\|P_{V_2(\Psi)} f\|_2^2=
 \le c^T A_{\Phi, \Psi}  (A_{\Psi, \Psi})^{-1}A_{\Psi, \Phi} c %\nonumber\\
% & \le &
%(1+\sin\theta(V_2(\Psi), V_2(\Phi))^2 \|f\|_2^2\nonumber\\
%& \le &  % (1+\sin\theta(V_2(\Psi), V_2(\Phi))^2
\le c^T A_{\Phi, \Phi} c
\end{equation}
for  all  $c\in \ell^2(\Lambda)$. Combining
\eqref{frirecovery.lem1.pf.eq2}, \eqref{frirecovery.tm.pf.eq1} and \eqref{frirecovery.tm.pf.eq2}, and applying
 Wiener's lemma for infinite matrices in  the Jaffard class ${\mathcal J}_\beta(\Lambda, \Lambda)$ \cite{jaffard90,suncasp05, suntams07}, we obtain
that $(A_{\Phi, \Psi}  (A_{\Psi, \Psi})^{-1}A_{\Psi, \Phi})^{-1}\in {\mathcal J}_\beta(\Lambda, \Lambda)$.
Thus the reconstruction matrix
 \begin{equation}
\label{frirecovery.tm.pf.eq4}
 R_{\Phi, \Psi}:=A_{\Phi, \Phi} \big(A_{\Phi, \Psi} (A_{\Psi, \Psi})^{-1} A_{\Psi, \Phi}\big)^{-1} A_{\Phi, \Psi} (A_{\Psi, \Psi})^{-1} \in {\mathcal J}_\beta(\Lambda, \Gamma).\end{equation}
Define a function $g_{F, \Phi, \Psi}$ on $\ell^p(\Lambda)$ by
\begin{equation}\label{frirecovery.tm.pf.eq5} g_{F, \Phi, \Psi}(x):=R_{\Phi, \Psi} f_{F, \Phi, \Psi}(x), x\in \ell^p,\end{equation}
where $f_{F, \Phi, \Psi}$ is given in \eqref{ffphipsi.def}.
Then the function $g_{F, \Phi, \Psi}$ is well-defined on $\ell^p(\Lambda)$ and satisfies $g(0)=0$, and its gradient
\begin{equation}\label{frirecovery.tm.pf.eq5+}
\nabla g_{F, \Phi, \Psi}(x)=
 R_{\Phi, \Psi} \langle \Psi, F'(x^T\Phi) \Phi\rangle, x\in \ell^p,\end{equation} is continuous and bounded in the Jaffard class ${\mathcal J}_\beta(\Lambda, \Lambda)$ by Lemma \ref{frirecovery.lem1}.
 From \eqref{frirecovery.tm.pf.eq2} it follow that
\begin{eqnarray} \label{frirecovery.tm.pf.eq6}
 c^T A_{\Phi, \Phi}  c & \le &  c^T A_{\Phi, \Phi} \big(A_{\Phi, \Psi} (A_{\Psi, \Psi})^{-1}A_{\Psi, \Phi}\big)^{-1}  A_{\Phi, \Phi}  c\nonumber
 \\
 & = &  c^T R_{\Phi, \Psi} A_{\Psi, \Psi} (R_{\Phi, \Psi})^T c =  \|c^TR_{\Phi, \Psi}\Psi\|_2^2\nonumber\\
 & \le & \big(1-\big(\delta(V_2(\Phi), V_2(\Psi))\big)^{2}\big)^{-1}\ \!
c^T A_{\Phi, \Phi}  c %\nonumber
\quad
  {\rm for \ all}\ c\in \ell^2(\Lambda).
  \end{eqnarray}
  This together with \eqref{frirecovery.tm.eq4} implies that
  \begin{eqnarray*} \label{frirecovery.tm.pf.eq6+}
   d^T \nabla g_{F, \Phi, \Psi} (x) c   & = &   \langle F'(x^T\Phi) (c^T\Phi), d^TR_{\Phi,\Psi}\Psi\rangle\nonumber\\
   & \le &  m^{-1}(1+\mu) \|c^T\Phi\|_2 \|d^TR_{\Phi, \Psi}\Psi\|_2\nonumber\\
%& \le & m^{-1}(1+\mu)  \big(\delta(V_2(\Phi), V_2(\Psi))\big)^{-1}
%(c^TA_{\Phi, \Phi} c)^{1/2} (d^TA_{\Phi, \Phi} d)^{1/2}\nonumber\\
& \le &
m^{-1} B_1 (1+\mu)  \big(1-\big(\delta(V_2(\Phi), V_2(\Psi))\big)^{2}\big)^{-1/2}
\ \! \|c\|_2\|d\|_2
\end{eqnarray*}
for  any $x\in \ell^p(\Lambda)$ and $c, d\in \ell^2(\Lambda)$, where $B_1=\sup_{\|c\|_2=1} \| c^T\Phi\|_2$.
This proves that
\begin{equation}\label{frirecovery.tm.pf.eq7}
  \|\nabla g_{F, \Phi, \Psi}(x)\|_{{\mathcal B}(\ell^2(\Lambda))} \le m^{-1} B_1 \frac{1+\mu}{\sqrt{1-(\delta(V_2(\Phi), V_2(\Psi)))^2}}
\quad {\rm for \ all} \ x\in \ell^p(\Lambda).
\end{equation}
Notice that
\begin{equation*} \label{frirecovery.tm.pf.eq8}
P_{V_2(\Phi)} (c^T R_{\Phi, \Psi}\Psi)=c^T R_{\Phi, \Psi} A_{\Psi, \Phi} A_{\Phi, \Phi}^{-1} \Phi=c^T\Phi
\end{equation*}
for all $c\in \ell^2$ by direct calculation.
%
%Set $g=x^T S\Psi$ and $f=x^T\Phi$
%for any $x\in \ell^2(\Lambda)$.
%Then $g\in V_2(\Psi), f\in V_2(\Phi)$ and
%$P_{V_2(\Phi)} g=x^T S A_{\Psi, \Phi} A_{\Phi, \Phi}^{-1} \Phi=f$.
Hence for any $x\in \ell^p$ and $c\in \ell^2$,
\begin{eqnarray}\label{frirecovery.tm.pf.eq9}
   c^T \nabla g_{F, \Phi, \Psi}(x) c   & = & \langle F'(x^T\Phi) (c^T\Phi), c^T\Phi\rangle +  \langle F'(x^T\Phi) (c^T\Phi), c^T R_{\Phi, \Psi}\Psi-c^T\Phi\rangle\nonumber\\
    & \ge &    m^{-1}\big (1-\mu) \|c^T\Phi\|_2^2-m^{-1}\mu \|c^T\Phi\|_2 \|c^T R_{\Phi, \Psi}\Psi-c^T\Phi\|_2\nonumber\\
    &\ge  & m^{-1} (1-\mu-\mu \tan\theta)\ \!
    c^TA_{\Phi, \Phi} c\nonumber\\
    &\ge  & m^{-1} A_1 (1-\mu-\mu \tan \theta)\ \!
    c^T  c
\end{eqnarray}
by \eqref{frirecovery.tm.eq4} and  \eqref{frirecovery.tm.pf.eq5}, where $A_1=\inf_{\|c\|_2=1} \|c^T\Phi\|_2$ and
$$\tan \theta=\delta(V_2(\Phi), V_2(\Psi))/ \sqrt{ 1-(\delta(V_2(\Phi), V_2(\Psi)))^2} .$$
By the above argument, we see that the function $g_{F, \Phi, \Psi}$ satisfies all requirements for the function $f$ in
Theorem \ref{pinvertibility.tm3}.
Thus the function  $g_{F, \Phi, \Psi}$  is invertible in $\ell^p$ and the gradient of its inverse is bounded and continuous in ${\mathcal J}_\beta(\Lambda)$. Then there exists a positive constant $C_0$ such that
\begin{equation}\label{frirecovery.tm.pf.eq10}
\|c_1-c_2\|_p\le C_0 \|g_{F, \Phi, \Psi}(c_1)-g_{F, \Phi, \Psi}(c_2)\|_p\quad  {\rm for \ all} \ c_1, c_2\in \ell^p.
\end{equation}
Hence for any $h_1=c_1^T\Phi$ and $h_2=c_2^T\Phi\in V_p(\Phi)$,
\begin{eqnarray*}
\|h_1-h_2\|_p & \le &  \big(\sup_{\|c\|_p=1} \|c^T\Phi\|_p\big)
\|c_1-c_2\|_p
 \le   C
\|R_{\Phi, \Psi} (S_{F, \Psi} (h_1)-S_{F, \Psi} (h_2))\|_p   \\
& \le &
C
\|S_{F, \Psi} (h_1)-S_{F, \Psi} (h_2)\|_p
\end{eqnarray*}
by  \eqref{frirecovery.tm.pf.eq4}, \eqref{frirecovery.tm.pf.eq10} and Lemma \ref{frirecovery.lem1}.
%where $C$ is a positive constant which could be different at different occurrences.
This proves the first inequality in \eqref{frirecovery.tm.stability}.

 The second inequality in \eqref{frirecovery.tm.stability} follows from
\eqref{frirecovery.tm.pf.eq4} and Lemma \ref{frirecovery.lem1}.
\end{proof}

\subsection{Reconstructing  signals  from
 averaging samples of their instantaneous companding}

\begin{thm}\label{frirecovery.tm2}
Let  $\Lambda, \Gamma$ be  two relatively-separated subsets of $\Rd$,
 $\Phi:=(\phi_\lambda)_{\lambda\in \Lambda}$
  and $\Psi:=(\psi_\gamma)_{\gamma\in \Gamma}$  have polynomial decay
of order $\beta>d$ and
generate Riesz bases for $V_2(\Phi)$
    and $V_2(\Psi)$ %\subset L^2$
    respectively,  the gap $\delta(V_2(\Phi), V_2(\Psi))$ from $V_2(\Phi)$ to $V_2(\Psi)$ be strictly less than one,
 %\begin{equation} \label{asmonotonicity.tm.eq5}
%r\|f\|_2^2\le \|P_{V_2(\Psi)} f\|_2\le \|f\|_{2}^2, \quad \
%f\in V_2(\Phi)\end{equation}
%for some positive constant $r\in (0,1)$,
  the companding function  $F$ be a continuously differential function on $\RR$ satisfying $F(0)=0$ and
 \eqref{frirecovery.tm.eq4}
for some nonzero constant $m$, and let $R_{\Phi, \Psi}$ be as in
 \eqref{frirecovery.tm.pf.eq4}. Set
$$ m_0= m^{-1} \big(\inf_{\|c\|_2=1} \|c^T\Phi\|_2\big)
 \Big(1-\mu-\mu \frac{\delta(V_2(\Phi), V_2(\Psi))}{\sqrt{ 1-(\delta(V_2(\Phi), V_2(\Psi)))^2}}\Big)$$ and
 $$L=m^{-1}
 \big(\sup_{\|c\|_2=1} \|c^T\Phi\|_2\big) \frac{1+\mu}{\sqrt{ 1- (\delta(V_2(\Phi), V_2(\Psi)))^2}}.$$
Then
\begin{itemize}

\item [{(i)}]
Given any sampling data $y=\langle F(x_\infty^T\Phi), \Psi\rangle $
of a signal $x_\infty\in \ell^p$ and any   initial guess $x_0\in \ell^p(\Lambda), 1\le p\le \infty$, the Van-Cittert iteration
 $x_n, n\ge 0$, defined by
\begin{equation}\label{frirecovery.tm2.eq1}
x_{n+1}=x_n-\alpha R_{\Phi, \Psi} (\langle F(x_n^T\Phi), \Psi\rangle -y)
\end{equation}
 converges exponentially  to $x_\infty$ in $\ell^p(\Lambda)$ when the relaxation  factor $\alpha$ satisfies
$0<\alpha<m_0/(L+L^2)$.
%Furthermore  if   $y$ is the nonlinear sampling data of
%a signal $x_\infty\in \ell^p$, i.e., $y=\langle F(x_\infty^T\Phi), \Psi\rangle$, then
%\begin{equation}\label{frirecovery.tm.eq3}
%\lim_{n\to \infty} x_n=x_\infty \quad {\rm in } \ \ell^p.\end{equation}

\item [{(ii)}]  If $F''$ is continuous and bounded,
then
the quasi-Newton iteration defined  by
\begin{equation} \label{frirecovery.tm2.eq2}
 x_{n+1}=x_{n}- ( R_{\Phi, \Psi} \langle \Psi,  F'(x_{n}^T\Phi)\Phi\rangle)^{-1} R_{\Phi, \Psi} (\langle F(x_{n}^T\Phi), \Psi\rangle-y),\quad  n\ge 1,
\end{equation}
 converges quadratically  if the initial guess $x_0$ is so chosen  that $\|\langle F(x_{0}^T\Phi), \Psi\rangle-y\|_p$
 is sufficiently small.

\end{itemize}

\end{thm}

\begin{rem} \label{vcl2.rem}
{\rm In the proof of Theorem \ref{frirecovery.tm2}, we show that the Van-Cittert iteration \eqref{frirecovery.tm2.eq1}
converges exponentially for any data $y\in \ell^p(\Gamma)$,  not necessarily the observed data $\langle F(x^T\Phi), \Psi\rangle$
of a signal $x\in \ell^p(\Lambda)$.
Unlike local R-quadratic convergence of the quasi-Newton iteration \eqref{frirecovery.tm2.eq2} established in Theorem \ref{frirecovery.tm2},  the quasi-Newton iteration
\begin{equation} \label{frirecovery.tm2.eq2****}
 x_{n+1}=x_{n}- \alpha_n ( R_{\Phi, \Psi} \langle \Psi,  F'(x_{n}^T\Phi)\Phi\rangle)^{-1} R_{\Phi, \Psi} (\langle F(x_{n}^T\Phi), \Psi\rangle-y),\quad  n\ge 1,
\end{equation}
with appropriate varying step size $\alpha_n$  could have global convergence when $p=2$, see \cite{eldar08}.
For $p=2$, we also remark that
 the following modified Van-Cittert iteration
 \begin{equation}\label{vcl2.rem.eq1}
 x_{n+1}=x_n- m \tilde R
\big (\langle \Psi, F(x_n^T\Phi)\rangle -y\big), n\ge 0,
 \end{equation}
has exponential convergence
in $\ell^2$  with the assumption \eqref{frirecovery.tm.eq4} replaced by a weak condition:
\begin{equation}\label{campandingstability.tm.eq6}
\mu<
\sqrt{1-(\delta(V_2(\Phi), V_2(\Psi)))^2},
\end{equation}
where $\tilde R=
\big(A_{\Phi, \Psi} (A_{\Psi, \Psi})^{-1} A_{\Psi, \Phi}\big)^{-1} A_{\Phi, \Psi} (A_{\Psi, \Psi})^{-1}$.
The above exponential convergence can be proved  by following the argument in \cite{Faktor10}.
Without loss of generality, we assume that
 $\Phi$ (resp. $\Psi$) is an orthonormal basis for $V_2(\Phi)$ (resp. $V_2(\Psi)$).
Let $x_n^1, x_n^2, n\ge 0$, be  sequences in the Van-Cittert algorithm \eqref{vcl2.rem.eq1} with the observed data $y$ replaced by
$y_1, y_2\in \ell^2$. Then
it follows from  \eqref{frirecovery.tm.pf.eq4} and \eqref{vcl2.rem.eq1}  that
\begin{eqnarray}\label{compandingstability.tm.eq2}
\|x_{n+1}^1-x_{n+1}^2\|_2
&  \le &  \|x_n^1-x_{n}^2-  m  \tilde R ( \langle  F((x_n^1)^T\Phi), \Psi\rangle -\langle F((x_n^2)^T\Phi),\Psi\rangle ))\|_2\nonumber\\
& &  \quad+ |m| \|y_1-y_2\|_2 \nonumber\\
& \le &  \| \tilde R \langle (x_n^1-x_n^2)^T\Phi-m F((x_n^1)^T\Phi)+ m F((x_n^2)^T\Phi), \Psi\rangle\|_2\nonumber\\
& &  \quad+ |m| \|y_1-y_2\|_2\nonumber\\
& \le &  |m| \|y_1-y_2\|_2 + \big(1-(\delta(V_2(\Phi), V_2(\Psi)))^2\big)^{-1/2}\nonumber\\
& &\quad \times \|(x_n^1-x_n^2)^T\Phi-m F((x_n^1)^T\Phi)+m F((x_n^2)^T\Phi)\|_2\nonumber\\
%& \le & \big(1-(\delta(V_2(\Phi), V_2(\Psi)))^2\big)^{-1/2}\mu \| (x_n^1-x_n^2)^T\Phi\|_2+
%|m| \|y_1-y_2\|_2\nonumber\\
&  \le &  \big(1-(\delta(V_2(\Phi), V_2(\Psi)))^2\big)^{-1/2} \mu \| x_n^1-x_n^2\|_2+
|m| \|y_1-y_2\|_2
\end{eqnarray}
for all $n\ge 0$.
Taking $y_2=y_1=y\in \ell^2$ in
\eqref{compandingstability.tm.eq2}  and using \eqref{campandingstability.tm.eq6}
shows that the map $V(x):=x- m \tilde R
\big (\langle \Psi, F(x^T\Phi)\rangle -y\big)$ is contractive on $\ell^2$ and hence
has a  unique fixed point $V^{-1}(y)$. Applying
\eqref{compandingstability.tm.eq2} with $y_1=y_2=y$ and $x_0^2=V^{-1}(y)$
leads to  the exponential convergence of  the Van-Cittert algorithm \eqref{vcl2.rem.eq1},
\begin{equation*}
\|x_n-V^{-1}(y)\|_2\le \mu^n \big(1-(\delta(V_2(\Phi), V_2(\Psi)))^2\big)^{-n/2}
\|x_0-V^{-1}(y)\|_2, \ n\ge 0.
\end{equation*}
Also  taking limit in \eqref{compandingstability.tm.eq2} shows that
 the  Van-Cittert algorithm \eqref{vcl2.rem.eq1}
has $\ell^2$-stability:
\begin{equation*}
\|V^{-1}(y_1)-V^{-1}(y_2)\|_2\le  \frac{\sqrt{1-(\delta(V_2(\Phi), V_2(\Psi)))^2}}{\sqrt{1-(\delta(V_2(\Phi), V_2(\Psi)))^2}- \mu}
\|y_1-y_2\|_2
\end{equation*}
for any $y_1, y_2\in \ell^2$.
From the above argument, we see that the limit  $V^{-1}(y)$ of the Van-Cittert algorithm
\eqref{vcl2.rem.eq1} satisfies
 %\begin{equation*}
$\tilde R (\langle \Psi, F( (T^{-1}(y))^T\Phi)\rangle-y)=0$.
%\end{equation*}
The  above consistence condition  is established in \cite{Faktor10} under additional assumption that $\tilde R$ is invertible.
}
\end{rem}

\begin{proof}[Proof of Theorem \ref{frirecovery.tm2}]
(i) \quad Let $g_{F, \Phi, \Psi}$ be the function on $\ell^p(\Lambda), 1\le p\le \infty$, defined in \eqref{frirecovery.tm.pf.eq5}.
By the argument in the proof of Theorem \ref{frirecovery.tm1},
 the function $g_{F, \Phi, \Psi}$  satisfies all requirements
 in Theorem \ref{vancittertiteration.tm}. Then applying Theorem \ref{vancittertiteration.tm} gives the exponential convergence
 of the sequence $x_n, n\ge 0$, in the Van-Cittert iteration  \eqref{frirecovery.tm2.eq1}, and
 the limit  follows from the invertibility of the function $g_{F, \Phi, \Psi}$ in $\ell^p$.

(ii)\quad
By \eqref{frirecovery.tm.pf.eq4} and  \eqref{frirecovery.tm.pf.eq5+}, we obtain
\begin{eqnarray}
& & \|\nabla g_{F, \Phi, \Psi}(x)-\nabla g_{F, \Phi, \Psi}(\tilde x)\|_{{\mathcal J}_\beta(\Lambda)}\nonumber\\
&\le  &  C \big\|\big\langle \Psi, \big(F'(x^T\Phi)-F'(\tilde x^T\Phi)\big) \Phi\big\rangle \big\|_{{\mathcal J}_\beta(\Gamma, \Lambda)}\nonumber\\
% & \le &  C \sup_{\lambda\in \Lambda, \gamma\in \Gamma}
%|\langle (F'(x^T\Phi)-F'(\tilde x^T\Phi))\phi_\lambda, \psi_\gamma\rangle| (1+|\lambda-\gamma|)^\beta\nonumber\\
 & \le & C \|F^{\prime\prime}\|_\infty \|x^T\Phi-{\tilde x}^T\Phi\|_\infty \le C \|x-\tilde x\|_\infty
 \le C \|x-\tilde x\|_p %\nonumber\\
% & \le &  C \|x-\tilde x\|_p\quad {\rm for \ all} \  x, \tilde x\in \ell^p,
\end{eqnarray}
for  all $x, \tilde x\in \ell^p$. %, where $C$ is a positive constant which may be different at different occurrences.
Then the local R-quadratic convergence of the quasi-Newton iteration \eqref{frirecovery.tm2.eq2} follows from Theorem \ref{newton.tm}.
\end{proof}

\subsection{Numerical simulation}\label{fri3.subsection}
In this subsection, we  present some numerical simulations for  demonstration of the theoretical  results in Theorem \ref{frirecovery.tm2}.
%on the recovery of signals with finite rate of innovation from their nonlinear sample data.
Let
 $\Lambda=\{t_i\}_{i=1}^{40}$ contain
 randomly selected knots $t_i\in [-2, 2], 1\le i\le 40$, satisfying
$0.05\le \min_{1\le i\le 39} t_{i+1}-t_i \le \max_{1\le i\le 39} t_{i+1}-t_i \le 0.15$,
$\Gamma=\{-2+i/20\}_{i=1}^{80}$
be the set of uniform knots on $[-2, 2]$,
$\Phi=(\phi_{t_i})_{t_i\in \Lambda}$ be the column vector of interpolating cubic splines with knots $\Lambda$ satisfying
$\phi_{t_i}(t_j)=\delta_{ij}$ for all $1\le i, j\le 40$  where $\delta_{ij}$ is the Kronecker symbol,
  $\Psi=\{10\chi_{[-2+(i-1)/20, -2+i/20)}\}_{i=1}^{80}$ be the column vector of uniform sampling functionals on $\Gamma$,
  and $F(t)=\sin (\pi t/2)$ be the companding function. In the simulation
  we use
    $x_\infty=(c_1, \ldots, c_{40})^T/\max_{1\le i\le 40} |c_i|$
    as the original sequence to be recovered from its nonlinear samples $\langle F(x_\infty^T\Phi), \Psi\rangle$
    (plotted in Figure \ref{fig1.fig1}(b)),
where  $c_i\in [-1/2, 1/2], 1\le i\le 40$. %, are randomly selected.
The nonuniform cubic spline $x_\infty^T\Phi$ %with knots in  $\Lambda$
is plotted in Figure \ref{fig1.fig1}(a) in  a continuous line. We add piecewise random  noise $\epsilon=
(r(1) a(1), \ldots, r(80) a(80))^T \|\langle F(x_\infty^T\Phi), \Psi\rangle\|_\infty$ of noise level from 0\% to $2.5\%$, plotted in Figure \ref{fig1.fig1}(b)
in a dashed line,
where $r(i)\in [-0.05, 0.05], 1\le i\le 80$, are randomly selected and
$a(i), 1\le i\le 80$, take
value one when $i\in [1, 16)\cup [72, 80]$, two when $i\in [24, 40)\cup [48, 64)$,
and zero otherwise. %zero when $i\in [16, 24)\cup[40, 48)\cup [64, 72)$,
We reconstruct the signal $x_\infty^\epsilon$ from the noisy  nonlinear samples $\langle F(x_\infty^T\Phi), \Psi\rangle+\epsilon$
by applying the Van-Cittert  iterative method \eqref{frirecovery.tm2.eq1}
with $x_0=0$ and $y=\langle F(x_\infty^T\Phi), \Psi\rangle+\epsilon$,
and plot the difference between  cubic splines $(x_\infty^\epsilon)^T \Phi$ and $x_\infty^T\Phi$ in Figure \ref{fig1.fig1}(a)
in a dashed line.
The numerical result shows that $\|x_\infty^\epsilon-x_\infty\|_\infty=0.003$, which illustrates
 the $\ell^\infty$-stability of the reconstruction procedure
in Theorem \ref{frirecovery.tm2}.
The nonlinear sampling procedure $S:x\longmapsto \langle F(x^T\Phi), \Psi\rangle$ is locally behaved as  entries in the column vectors $\Phi$ and $\Psi$ are
locally supported, which can be observed by comparing the shape
of the original cubic spline  $x_\infty^T\Phi$ in Figure \ref{fig1.fig1}(a) and the nonlinear samples $\langle F(x^T\Phi), \Psi\rangle$ in Figure \ref{fig1.fig1}(b).
For the sampling procedure without instantaneous companding (i.e. $F(t)=t$),
it is shown in \cite{bns09, sunsiam06} that the reconstruction procedure is also locally behaved, that is,
the amplitude of a signal  at any position is
essentially  determined by the adjacent  sampling values. We notice that
 the cubic splines $(x_\infty^\epsilon)^T\Phi$ and $x_\infty^T\Phi$ are almost perfectly matched
in the interval $[0, 0.4]$  where no noise is added to adjacent nonlinear samples, which
indicates that  the reconstruction procedure from the  nonlinear samples
is  locally behaved. % This observation may suggest to
%  implement the iteration method locally and then patch  them together.

\begin{figure}[hbt]
\centering
\begin{tabular}{ccc}
   \includegraphics[width=80mm]{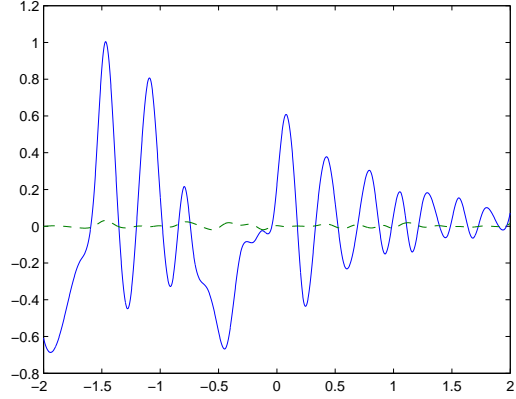}\\
   (a) \\
         \includegraphics[width=80mm]{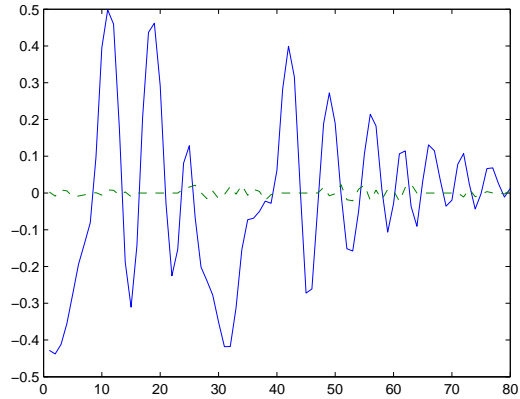} \\
       %& \includegraphics[width=38mm]{./sample1.eps}\\
       %& \includegraphics[width=60mm]{./gaussiansample2.eps} \\
   %   {\rm \large Gaussian window $\exp(-|x|^2/2\sigma^2)$ with  $\sigma=1/6$} &
      % & \includegraphics[width=60mm]{./originalsignal.eps} \\
 (b) % & (c)
%   {\rm \large (a) Delta impulse signal with innovative rate 2}\\
%    \includegraphics[width=50mm]{./sample1.eps}\\ % & \includegraphics[width=60mm]{./sample2.eps} \\
% {\rm \large samples on uniform grid  with  sampling rate $6$.}
   \end{tabular}
\caption{\small (a) The cubic spline $x_\infty^T\Phi$ is  plotted in a  continuous line, while it is
plotted in   a dashed line
 the difference between the original cubic spline  $x_\infty^T\Phi$
and the reconstructed cubic spline $x_\infty^\epsilon \Phi$ from noisy nonlinear sampled data
$\langle F(x_\infty^T\Phi), \Psi\rangle+\epsilon$.
 %at innovative locations
(b) Nonlinear  sampled data $\langle F(x_\infty^T\Phi), \Psi\rangle$  is plotted in a continuous line, and
 the piecewise random noise $\epsilon$ with noise level between 0\% and 2.5\% is in a dashed line.}
\label{fig1.fig1}
\end{figure}

We apply the Van-Cittert iterative method \eqref{frirecovery.tm2.eq1} with zero initial guess $x_0=0$ and relaxation factor $\alpha=0.3$
 to recover the signal $x_\infty$
 from its nonlinear samples $\langle F(x_\infty^T\Phi), \Psi\rangle$.
The corresponding numerical error
%using the Van-Cittert iterative method \eqref{frirecovery.tm.eq7} with zero initial
is presented in  Table
\ref{cubicspline1.table}, where listed in the first column is the number of iteration, listed
in the second, third and fourth columns are the $\ell^\infty$-error $\|x_n-x_\infty\|_\infty$,
 the mean squared error $\|x_n-x_\infty\|_2$,
and the $\ell^1$-error $\|x_n-x_\infty\|_1$ respectively,
and listed in the last column is the $\ell^\infty$-error $\|\langle F(x_n^T\Phi), \Psi\rangle-
\langle F(x_\infty^T\Phi), \Psi\rangle\|_\infty$ of nonlinear sample data.

\begin{table}[ht]
\caption{Reconstruction error via the Van-Cittert iteration}
\begin{tabular}{ccccc}
\hline\hline
 Iteration& $\ell^\infty$-error&
$\frac{\ell^2-{\rm error}} {\small \sqrt{40}}$ & $\frac{ \ell^1-{\rm error}}{40}$
 & $\ell^\infty$-data  \\
\hline
% 0 &    10.2209  &     2.1839 &     1.0000   &   0.4982 \\
5 &    1.9893 &     0.5049  &    0.3334    &  0.1168 \\
10 &     0.5230 &     0.1832 &     0.1533 &     0.0425\\
15 &     0.2102 &     0.0931 &    0.0846  &   0.0222\\
20 &    0.1053 &     0.0535  &   0.0503  &    0.0128\\
30 &     0.0343  &   0.0201 &   0.0194 &    0.0048\\
40 &    0.0128  &   0.0081 &    0.0078 &   0.0019\\
50 &   0.0050 &   0.0033 &    0.0032 &    0.0008
\\
\hline \hline
\end{tabular}
\label{cubicspline1.table}

\end{table}

A difficult problem is how to select the  relaxation factor $\alpha$. We tried to
change the relaxation factor in every iteration, c.f. \eqref{frirecovery.tm2.eq2****}. The ideas are to use smaller  relaxation factor if
the $\ell^\infty$-error
$\|\langle F(x_n^T\Phi), \Psi\rangle-
y\|_\infty$ between
the nonlinear sample data of the $n$-th iteration $x_n$ and  the given nonlinear sample data $y$
increases, and to use  larger relaxation factor when the ratio between
$\|\langle F(x_n^T\Phi), \Psi\rangle-
y \|_\infty$  and  $\|y\|_\infty$ gets smaller. The above modification of the Van-Cittert iteration method
 does provide faster
convergence, but no significant improvement is observed. We also apply the  combination of the Van-Cittert iteration method
and the quasi-Newton  iteration method.
As  shown in  Theorem  \ref{frirecovery.tm2},
the Van-Cittert iteration method  has slow global convergence and
the quasi-Newton iteration method has fast local convergence.
So  we may use the Van-Cittert iteration at the beginning
and the quasi-Newton iteration when the ratio between
$\|\langle F(x_n^T\Phi), \Psi\rangle-
y \|_\infty$  and  $\|y\|_\infty$ is below certain threshold, say $10\%$. The
numerical  error is
 shown in Table \ref{cubicspline2.table}, where
  the quasi-Newton iteration  method \eqref{frirecovery.tm2.eq2} is used from the 6-th iteration.
% in 6th-10th iterations.
Comparing  the numerical results in
Table \ref{cubicspline1.table} and Table \ref{cubicspline2.table},
we may conclude  that, in order to recover a signal from its nonlinear sample data,
it is much better to use the combination of the
Van-Cittert iteration method  and the quasi-Newton iteration method
than to use the
Van-Cittert iteration method solely. %The cost  to pay is   to compute the gradient
%and to solve a linear system in each quasi-Newton iteration.
%, which are usually expensive in %very
% high dimensions.

  \begin{table}[ht]
\caption{Reconstruction error via a combination of the Van-Cittert iteration  and the quasi-Newton iteration}
\begin{tabular}{ccccc}
\hline\hline
 Iteration& $\ell^\infty$-error& $\frac{\ell^2-{\rm error}} {\small \sqrt{40}}$ & $\frac{ \ell^1-{\rm error}}{40}$ & $\ell^\infty$-data  \\
\hline
%0 &    10.2209  &     2.1839 &     1.0000   &   0.4982 \\
%1 &  6.1418  &    1.3677 &     0.6969 &     0.4982\\
%2 &  4.0152  &    0.9295  &   0.5209  &   0.2749\\
%3 & 2.7917  &  0.6702 &    0.4097  &    0.1690\\
5 & 1.9893 &   0.5049  &  0.3334 &   0.1168 \\
6 &   0.2372  &  0.1175 &   0.1073 &   0.0896\\
7 &   0.2034 &    0.1033 &     0.0954 &     0.0250\\
8 &    0.0232 &     0.0170 &     0.0168 &     0.0222\\
9 &   0.0010 &    0.0008 &     0.0007 &     0.0037\\
10&    0.0000 &     0.0000 &     0.0000  &    0.0002
\\
\hline \hline
\end{tabular}
\label{cubicspline2.table}
\end{table}

\section{Local identification of innovation positions and qualification of amplitudes}\label{friinnovation.section}

In this section, we apply the theory established for the nonlinear functional equation \eqref{nfe.def} with $p=\infty$ in the first two parts of this paper indirectly
to  local identification of innovation positions  and qualification of amplitudes  of a signal
$h$ having the parametric representation \eqref{secondmodel.def}. In the first subsection, we show
that innovation positions and amplitudes of
a signal $h$ having the parametric representation \eqref{secondmodel.def}
can be {\bf precisely} identified and qualified provided that the linearization at approximate innovation positions and amplitudes is stable, see Theorem \ref{friidentification.tm}.
In the second subsection, we show that any signal $h$  in a perturbed shift-invariant space  with unknown (but small) perturbations
can be recovered exactly from its average samples $\langle h, \psi_m(\cdot-k)\rangle, 1\le m \le M, k\in \Zd$, provided that the generator  $\varphi$ of the perturbed
shift-invariant space and the average samplers $\psi_m, 1\le m\le M$, satisfy \eqref{blindcondition}.
In the last subsection, we  present some numerical simulations to demonstrate the  precise identification and  accurate qualification.

\subsection{Precise identification of innovation positions  and accurate qualification of amplitudes}
 %of signals  having the parametric representation \eqref{secondmodel.def}}
\begin{thm}\label{friidentification.tm}
Let    $\Lambda_0$ and $\Gamma$ be  relatively-separated subsets of
$\Rd$,  $c_0:=(c_0(\lambda))_{\lambda\in \Lambda_0}\in \ell^\infty(\Lambda_0)$, $\Psi=(\psi_\gamma)_{\gamma\in \Gamma}$ have polynomial decay of order $\beta>d$,
the impulse response $\varphi\in C^2(\RR^d)$ satisfy
\begin{equation}\label{friidentification.tm.eq2}
\|\varphi(\cdot) (1+|\cdot|)^\beta\|_\infty +\|\nabla \varphi(\cdot) (1+|\cdot|)^\beta\|_\infty+\|\nabla^2 \varphi (\cdot) (1+|\cdot|)^\beta\|_\infty<\infty,
\end{equation}
and let
 the companding function $F\in C^2(\RR)$  satisfy
\begin{equation} \label{friidentification.tm.eq3}
\|F'\|_\infty+\|F^{\prime\prime}\|_\infty<\infty.
\end{equation}
If %that the initial  guess $\Lambda_0$ of  the innovation position set  is  relatively-separated,
% the initial guess $c_0=(c_0(\lambda))_{\lambda\in \Lambda_0}$ of
%the amplitudes belongs to $\ell^\infty(\Lambda_0)$, and
  the nonlinear sampling process \eqref{localnonlinearsampling.def}
has its linearization  at $(\Lambda_0, c_0)$ being stable (i.e.,
the matrix $S_{ \Lambda_0, c_0}$ in \eqref{slambda0.def} satisfies \eqref{slambda0.stability}),
then there exists a positive constant $\delta_0$ such that any signal $h$, that has the parametric representation $h(t)=\sum_{\lambda\in \Lambda_0}
(c_{0}(\lambda)+c(\lambda))\varphi(t-\lambda-\sigma(\lambda))$
for some vectors $c:=(c(\lambda))_{\lambda\in \Lambda_0}$ %$\in \ell^\infty(\Lambda_0)$
 and $\sigma:=(\sigma(\lambda))_{\lambda\in \Lambda_0}$ %\in (\ell^\infty(\Lambda_0))^d$ with
 with $\max(\|c\|_\infty, \|\sigma\|_\infty)\le \delta_0$,
 can be stably recovered from
its %nonlinear
sampling data $\langle F(h), \Psi\rangle$.
\end{thm}

\begin{proof} %[Proof of Theorem \ref{friidentification.tm}]
To prove the desired stable recovery, we need to find a small number $\delta_0$ and two positive constants $C_1$ and $C_2$  such that
\begin{eqnarray}\label{friidentification.tm.eq5}
C_1 (\|c_1- c_2\|_\infty+\|\sigma_1-\sigma_2\|_\infty)
%\big(\sup_{\lambda\in \Lambda_0} |c_{1}(\lambda)- c_{2}(\lambda)|+|\sigma_{1}(\lambda)-\sigma_{2})(\lambda)|\big)
& \le &
\| \langle F(h_1)-F(h_2), \Psi\rangle \|_\infty\nonumber\\
 & \le &
C_2 (\|c_1- c_2\|_\infty+\|\sigma_1-\sigma_2\|_\infty)
 %\big(\sup_{\lambda\in \Lambda_0} |c_{1}(\lambda)- c_{2}(\lambda)|+|\sigma_{1}(\lambda)-\sigma_{2})(\lambda)|\big)
%C_2(\|c_1- c_2\|_\infty+\|\sigma_1-\sigma_2\|_\infty)
\end{eqnarray}
for all  vectors  $c_i:=(c_{i}(\lambda))_{\lambda\in \Lambda_0}$ and $\sigma_i:=(\sigma_{i}(\lambda))_{\lambda\in \Lambda_0}$
with $\max(\|c_i\|_\infty,\|\sigma_i\|_\infty)\le \delta_0$,
where $h_i=\sum_{\lambda\in \Lambda_0}
(c_{0}(\lambda)+c_{i}(\lambda)) \varphi (\cdot-\lambda-\sigma_{i}(\lambda)), i=1, 2$.

By  \eqref{friidentification.tm.eq3} and the
polynomial decay property for the average sampler $\Psi$, we obtain  % we have that
\begin{eqnarray*} & &
\| \langle F(h_1)-F(h_2), \Psi\rangle\|_\infty\nonumber\\
% & \le &
%  \int_0^1 \|\langle F'(sh_1+(1-s) h_2) (h_1-h_2), \Psi\rangle\|_\infty ds \nonumber\\
%  &\le &
%  \|F'\|_\infty \|h_1-h_2\|_\infty  \sup_{\gamma\in \Gamma} \| \psi_\gamma\|_1 \nonumber\\
& \le &  C \|F'\|_\infty \|\Psi\|_{\infty, \beta}
\|c_0+c_2\|_\infty
\Big\|\sum_{\lambda\in \Lambda_0}   \big|\varphi\big(\cdot-\lambda-\sigma_1(\lambda)\big)-\varphi\big(\cdot-\lambda-\sigma_2(\lambda)\big)\big|\Big\|_\infty
\nonumber\\
& & + C \|F'\|_\infty \|\Psi\|_{\infty, \beta} \|c_1-c_2\|_\infty\Big\|\sum_{\lambda\in \Lambda_0}  |\varphi(\cdot-\lambda-\sigma_1(\lambda))|\Big\|_\infty\nonumber\\
&\le  &   C (1+ \delta_0)^\beta (\|c_0\|_\infty+\delta_0)  \|F'\|_\infty  \|\Psi\|_{\infty, \beta} \|\nabla \varphi(\cdot) (1+|\cdot|)^\beta \|_\infty
  \|\sigma_1-\sigma_2\|_\infty
\nonumber\\
& &   +  C (1+ \delta_0)^\beta \|F'\|_\infty \|\Psi\|_{\infty, \beta}  \|\varphi(\cdot)(1+|\cdot|)^\beta\|_\infty  \|c_1-c_2\|_\infty.
%
%
%     \|F'\|_\infty ( \|c_0\|_\infty+\|c_2\|_\infty)  \|\nabla \varphi(\cdot) (1+|\cdot|)^\beta\|_\infty
%      (1+\|\sigma_1\|_\infty+\|\sigma_2\|_\infty)^\beta \|\sigma_1-\sigma_2\|_\infty
%      \|\Psi\|_{\infty, \beta}
%      \sup_{\gamma\in \Gamma} \sum_{\lambda\in \Lambda_0}  \int_{\Rd}
%    |\varphi(t-\lambda-\sigma_1(\lambda)-\varphi(t-\lambda-\sigma_2(\lambda) | |\psi_\gamma(t)| dt\Big)\nonumber\\
%  & &  + \|F'\|_\infty
%    \sup_{\gamma\in \Gamma} \Big(\sum_{\lambda\in \Lambda_0}  |c_1(\lambda)-c_2(\lambda)| \int_{\Rd}
%    |\varphi(t-\lambda-\sigma_1(\lambda)| |\psi_\gamma(t)|\big) dt\nonumber\\
%   %  |c_{0, \lambda}+c_{1, \lambda}|\\
%% \int_0^1 \big|\big
%% \langle F'(t x_1+(1-t) x_2) \big(\varphi(\cdot-\lambda-\sigma_{1, \lambda})-\varphi(\cdot-\lambda-\sigma_{2, \lambda})\big) , \psi_\gamma\big\rangle\big| dt\nonumber\\
%% & & + \sup_{\gamma\in \Gamma} \sum_{\lambda\in \Lambda_0} |c_{1, \lambda}-c_{2, \lambda}|
%% \int_0^1 \big|\big
%% \langle F'(t x_1+(1-t) x_2) \varphi(\cdot-\lambda-\sigma_{1, \lambda}), \psi_\gamma\big\rangle\big| dt\nonumber\\
%%  \\
%& \le  &
%C \big(\sup_{\lambda\in \Lambda_0} |c_{1, \lambda}- c_{2, \lambda}|+|\sigma_{1, \lambda}-\sigma_{2, \lambda}|\big)
\end{eqnarray*}
This proves the second inequality in
 \eqref{friidentification.tm.eq5}.

%\bigskip

Now we prove the first inequality in
 \eqref{friidentification.tm.eq5}.
 From   \eqref{friidentification.tm.eq2}, \eqref{friidentification.tm.eq3} and the polynomial decay property for the average sampler $\Psi$, it follows that
   \begin{eqnarray}\label{friidentification.tm.pf.eq1+}
  & & |\langle F'(h) \varphi(\cdot-\lambda-\sigma(\lambda)), \psi_\gamma\rangle|+|\langle F'(h)  \nabla \varphi(\cdot-\lambda-\sigma(\lambda)), \psi_\gamma\rangle|\nonumber\\
% & \le &  \|F'\|_\infty\big( \|\varphi(\cdot)(1+|\cdot|)^\beta\|_\infty +\|\nabla \varphi(\cdot) (1+|\cdot|)^\beta\|_\infty\big)
%  \|\Psi\|_{\infty, \beta} (1+\|\sigma\|_\infty)^\beta\nonumber\\
%  & & \quad \times
%  \int_{\Rd} (1+|t-\lambda|)^{-\beta} (1+|t-\gamma|)^{-\beta} dt\nonumber\\
  & \le & C \|F'\|_\infty\big( \|\varphi(\cdot)(1+|\cdot|)^\beta\|_\infty +\|\nabla \varphi(\cdot) (1+|\cdot|)^\beta\|_\infty\big)
  \|\Psi\|_{\infty, \beta} (1+\|\sigma\|_\infty)^\beta\nonumber\\
& & \times
(1+|\gamma-\lambda|)^{-\beta} \quad {\rm for \ all} \ \gamma\in \Gamma \ {\rm and} \ \lambda\in \Lambda_0.
 \end{eqnarray}
Hence
 the matrix\begin{equation*} \label{friidentification.tm.pf.eq1}
A\Big(\begin{matrix} \sigma\\ c\end{matrix}\Big):=\left(\begin{array} {c}
 - (c_0(\lambda)+c(\lambda))\langle F'(h)  \nabla \varphi(\cdot-\lambda-\sigma(\lambda)), \psi_\gamma\rangle\\
 \langle F'(h) \varphi(\cdot-\lambda-\sigma(\lambda)), \psi_\gamma\rangle\\\end{array}\right)_{\gamma\in \Gamma, \lambda\in \Lambda_0}
\end{equation*}
is well-defined for any $c:=(c(\lambda))_{\lambda\in \Lambda_0}\in \ell^\infty(\Lambda_0)$ and $\sigma:=(\sigma(\lambda))_{\lambda\in \Lambda_0}\in (\ell^\infty(\Lambda_0))^d$,
   where  $h=\sum_{\lambda\in \Lambda_0} (c_{0}(\lambda)+c(\lambda)) \varphi (\cdot-\lambda-\sigma(\lambda))$.
  Moreover
  $\big (A\big(\begin{matrix} \sigma\\ c\end{matrix}\big)\big)^T A\big(\begin{matrix} \sigma\\ c\end{matrix}\big)$ belongs to the Jaffard class
 $({\mathcal J}_\beta(\Lambda_0))^{d+1}$,
 and
% \begin{eqnarray} \label{friidentification.tm.pf.eq4}
%  \| S_{c_0, \Lambda_0}(c, \sigma) (S_{c_0, \Lambda_0}(c, \sigma))^T \|_{({\mathcal J}_\beta(\Lambda_0))^{d+1}} &  \le &
% C \|F'\|_\infty^2 \big( \|\varphi(\cdot)(1+|\cdot|)^\beta\|_\infty^2 +|\nabla \varphi(\cdot) (1+|\cdot|)^\beta\|_\infty^2\big)
%  \nonumber\\
%  & & \times \|\Psi\|_{\infty, \beta}^2 (1+\|\sigma\|_\infty)^{2\beta}.
% \end{eqnarray}
  \begin{equation} \label{friidentification.tm.pf.eq4}
 \Big \| \Big (A\Big(\begin{matrix} \sigma\\ c\end{matrix}\Big)\Big)^T A\Big(\begin{matrix} \sigma\\ c\end{matrix}\Big)\Big \|_{({\mathcal J}_\beta(\Lambda_0))^{d+1}}  \le  C
 (1+\|\sigma\|_\infty)^{2\beta} (1+\|c\|_\infty)^2
 \end{equation}
   by
 \eqref{friidentification.tm.pf.eq1+} and the relatively-separatedness of the sets $\Lambda_0$ and $\Gamma$.
Notice that $A\big(\begin{matrix} \sigma\\ c\end{matrix}\big)$ becomes the matrix $S_{\Lambda_0, c_0}$ in \eqref{slambda0.def}
when $\sigma=0$ and $c=0$.
Then
$ (S_{\Lambda_0, c_0})^T S_{\Lambda_0, c_0}$ belongs to
the Jaffard class $({\mathcal J}_\beta(\Lambda_0))^{d+1}$ by  \eqref{friidentification.tm.pf.eq4}. This, together with
the stability condition
\eqref{slambda0.stability} and  the Wiener's lemma   for infinite matrices in  the Jaffard class \cite{jaffard90, suncasp05},
proves that $( (S_{\Lambda_0, c_0})^TS_{\Lambda_0, c_0})^{-1} $
 belongs to the Jaffard class
$({\mathcal J}_\beta(\Lambda_0))^{d+1}$; i.e.,
\begin{equation} \label{friidentification.tm.pf.eq6}
\|( (S_{\Lambda_0, c_0})^T S_{\Lambda_0, c_0})^{-1}\|_{({\mathcal J}_\beta(\Lambda_0))^{d+1}}\le D_1
\end{equation}
for some positive constant $D_1$. 
Using similar argument to prove
\eqref{friidentification.tm.pf.eq1+}, we get
\begin{eqnarray} \label{friidentification.tm.pf.eq8}
& &  |\langle F'(h) \varphi(\cdot-\lambda-\sigma(\lambda)), \psi_\gamma\rangle-\langle F'(h_0) \varphi(\cdot-\lambda), \psi_\gamma\rangle|
\nonumber\\
& & + |\langle F'(h) \nabla \varphi(\cdot-\lambda-\sigma(\lambda)), \psi_\gamma\rangle-\langle F'(h_0) \nabla \varphi(\cdot-\lambda), \psi_\gamma\rangle|\nonumber\\
%& \le &
%\|F^{\prime\prime}\|_\infty \|h-h_0\|_\infty
%\int_{\Rd} |\varphi(t-\lambda)||\psi_\gamma(t)| dt\nonumber\\
%& & + \|F'\|_\infty
%\int_{\Rd} |\varphi(t-\lambda-\sigma(\lambda))-\varphi(t-\lambda)||\psi_\gamma(t)| dt\nonumber\\
% \|x-x_0\|_\infty (1+|\lambda-\gamma|)^{-\beta}\nonumber\\
%& & + C_\beta^1 (1+\|\sigma\|_\infty)^\beta \|\sigma\|_\infty \|F'\|_\infty \|\nabla\phi\|_{\infty, \beta} \|\Psi\|_{\infty, \beta}
%(1+|\lambda-\gamma|)^{-\beta}
%\nonumber\\
& \le &  C
 (1+\|\sigma\|_\infty)^\beta (\|\sigma\|_\infty+\|c\|_\infty)
(1+|\lambda-\gamma|)^{-\beta}
\end{eqnarray}
for all $\lambda\in \Lambda_0$ and $\gamma\in \Gamma$.
Therefore
\begin{eqnarray} \label{friidentification.tm.pf.eq11}
  & & \big\| ( (S_{\Lambda_0, c_0})^T S_{\Lambda_0, c_0})^{-1} S_{\Lambda_0, c_0}^T A\big(\begin{matrix} \sigma\\ c\end{matrix}\big)-I\big\|_{({\mathcal J}_\beta(\Lambda_0))^{d+1}}\nonumber\\
& \le &   C  (1+\|\sigma\|_\infty)^{\beta}  (1+\|c\|_\infty) (\|c\|_\infty + \|\sigma\|_\infty),
\end{eqnarray}
where $I$ is the identity matrix of appropriate size.
For the function $f_{\Lambda_0, c_0, \Psi}$ in \eqref{flambda0.def}, we notice that
$$\nabla f_{\Lambda_0, c_0,\Psi} \big(\begin{matrix} \sigma\\ c\end{matrix}\big)= ( (S_{\Lambda_0, c_0})^T S_{\Lambda_0, c_0})^{-1} S_{\Lambda_0, c_0}^T A\big(\begin{matrix} \sigma\\ c\end{matrix}\big).$$
Then
 $\nabla f_{\Lambda_0, c_0, \Psi} \big(\begin{matrix} \sigma\\ c\end{matrix}\big)$ belongs to the Jaffard class $({\mathcal J}_\beta(\Lambda_0))^{d+1}$
  and
\begin{equation}\label{friidentification.tm.pf.eq14}
 \big\|\nabla f_{\Lambda_0, c_0, \Psi}\big(\begin{matrix} \sigma\\ c\end{matrix}\big)-I\big\|_{({\mathcal J}_\beta(\Lambda_0))^{d+1}}\le  C (1+\|\sigma\|)^\beta (1+\|c\|_\infty)
(\|c\|_\infty+\|\sigma\|_\infty)
\end{equation}
  by   \eqref{friidentification.tm.pf.eq11}.
Recall that any infinite matrix in the
 Jaffard class $({\mathcal J}_\beta(\Lambda_0))^{d+1}$ is a bounded operator on $(\ell^\infty(\Lambda))^{d+1}$.
Thus there exists a positive number $\delta_0$ such that
\begin{equation}\label{friidentification.tm.pf.eq15-}
 \big\|\nabla f_{\Lambda_0, c_0, \Psi} \big(\begin{matrix} \sigma\\ c\end{matrix}\big)-I\big\|_{{\mathcal B}((\ell^\infty(\Lambda_0))^{d+1})}\le \frac{1}{3}
\end{equation}
for vectors $\sigma\in (\ell^\infty(\Lambda_0))^d$ and $c\in \ell^\infty(\Lambda_0)$   with
 $\max(\|c\|_\infty, \| \sigma\|_\infty)\le \delta_0$.
This implies that  there exists a positive constant  $\delta_0$ such that
\begin{eqnarray} \label{friidentification.tm.pf.eq15}
\frac{2}{3} \max\big(\|\sigma_1-\sigma_2\|_\infty, \|c_1-c_2\|_\infty\big)  & \le &
\big\|f_{\Lambda_0, c_0, \Psi}\big(\begin{matrix} \sigma_1\\ c_1\end{matrix}\big)-f_{\Lambda_0, c_0, \Psi}
\big(\begin{matrix} \sigma_2\\ c_2\end{matrix}\big) \big\|_\infty\nonumber\\
& \le &  \frac{4}{3} \max\big(\|\sigma_1-\sigma_2\|_\infty, \|c_1-c_2\|_\infty\big) \end{eqnarray}
for vectors $\sigma_i\in (\ell^\infty(\Lambda_0))^d$ and $c_i\in \ell^\infty(\Lambda_0)$   with
 $\max(\|c_i\|_\infty, \| \sigma_i\|_\infty)\le \delta_0, i=1,2$.
 The first inequality in \eqref{friidentification.tm.eq5} then follows from
\eqref{friidentification.tm.pf.eq1+}, \eqref{friidentification.tm.pf.eq6} and
 \eqref{friidentification.tm.pf.eq15}.
\end{proof}

\begin{thm}\label{friidentification.vancittert.tm} Let $\varphi, \Psi, F, \Lambda_0, c_0,  \delta_0$ be as in  Theorem \ref{friidentification.tm}, and $f_{\Lambda_0, c_0, \Psi}$  as in \eqref{flambda0.def}.
 Given  the  sampling data   $y=\langle F(h), \Psi\rangle$
 of a signal $h=\sum_{\lambda\in \Lambda_0} (c_0(\lambda)+c(\lambda))    \varphi(\cdot-\lambda-\sigma(\lambda))$ with
$c=(c(\lambda))_{\lambda\in \Lambda_0}$ and $\sigma=(\sigma(\lambda))_{\lambda\in \Lambda_0}$ satisfying
$\max\big(\|c\|_\infty, \|\sigma\|_\infty)\le \delta_0/2$,
we define the Van-Cittert iteration by
\begin{equation}\label{friidentification.vancittert}
\Big(\begin{matrix} \sigma_{n+1}\\
d_{n+1}\end{matrix}\Big)=
\Big(\begin{matrix} \sigma_{n}\\
d_{n}\end{matrix}\Big)
 -\alpha \Big(f_{\Lambda_0, c_0, \Psi} \big(\begin{matrix} \sigma_n\\ d_n\end{matrix}\big)
-  z_0 \Big),\quad  n\ge 0,
\end{equation}
with  $\big(\begin{matrix} \sigma_{0}\\
d_{0}\end{matrix}\big)=0, \alpha\in (0,1)$, and $z_0=(S_{\Lambda_0, c_0})^T S_{\Lambda_0, c_0})^{-1} S_{\Lambda_0, c_0}^T (y-\langle F(h_0),\Psi\rangle)$. %=f_{\Lambda_0, c_0, \Psi} \big(\begin{matrix} \sigma\\ c\end{matrix}\big)$.
Then  $\big(\begin{matrix}\sigma_n \\ d_n\end{matrix}\big), n\ge 0$,  converges  to
the solution $\big(\begin{matrix}\sigma \\ c\end{matrix}\big)$   exponentially.
\end{thm}

\begin{proof}
By \eqref{friidentification.tm.pf.eq15-}, \eqref{friidentification.tm.pf.eq15} and
\eqref{friidentification.vancittert}, we have
\begin{equation*}
\max (\|\sigma_1\|_\infty, \|d_1\|_\infty)\le \frac{4\alpha}{3} \max (\|\sigma\|_\infty, \|c\|_\infty)
\end{equation*}
and
\begin{equation*}
\max (\|\sigma_{n+1}-\sigma_n\|_\infty, \|d_{n+1}-d_n\|_\infty)\le \frac{3-2\alpha}{3}
 \max (\|\sigma_{n}-\sigma_{n-1}\|_\infty, \|d_n-d_{n-1}\|_\infty)
\end{equation*}
for all $n\ge 1$. Hence
$\max(\|\sigma_n\|_\infty, \|d_n\|_\infty)\le \delta_0$ and
the sequence $\big(\begin{matrix}\sigma_n \\ d_n\end{matrix}\big), n\ge 0$, in the
Van-Cittert iteration \eqref{friidentification.vancittert} converges  to
the solution $\big(\begin{matrix}\sigma \\ c\end{matrix}\big)$   exponentially.
\end{proof}

\subsection{Sampling in a perturbed shift-invariant space with unknown perturbation}

\begin{thm} \label{blindsampling.tm}
Let $\varphi$  satisfy \eqref{friidentification.tm.eq2} and $\psi_1, \ldots, \psi_M$ satisfy
\begin{equation}\label{blindsampling.tm.eq1}
 \sum_{1\le m\le M} \|\psi_m(\cdot) (1+|\cdot|)^\beta\|_\infty
<\infty.
\end{equation}
If \eqref{blindcondition} holds,
then  for any $L\ge 1$ there exists  a positive number $\delta_1\in (0,1/2)$ such that
any signal $h(t)=\sum_{k\in \Zd} c(k) \varphi(t-k-\sigma(k))$
with  $\|(\sigma(k))_{k\in \Zd}\|_\infty\le \delta_1$ and $\|(c(k))_{k\in \Zd}\|_{\ell^\infty_\oslash}:=\sup_{c(k)\ne 0} |c(k)|+|c(k)|^{-1}
\le L$
 can be recovered
 from its average sample data
 $\langle  h , \psi_m(\cdot-k)\rangle, 1\le m\le M,  k\in \Zd$, in a stable way.
%Noisy estimate.
\end{thm}

\begin{rem}\label{blindsampling.averagesamplingremark}
{\rm Let $\varphi$   and $\psi_1, \ldots, \psi_M$  be as in Theorem \ref{blindsampling.tm},  and let
 % $V_2(\varphi_0, \varphi_1, \cdots, \varphi_d)$ be the shift-invariant space spanned by $\varphi_0, \ldots, \varphi_d$; i.e.,
\begin{equation*}
V_2(\varphi, \nabla\varphi)=\Big\{  \sum_{k\in \Zd} \big( e_o(k) \varphi(\cdot-k)+ (e_g(k))^T\nabla\varphi (\cdot-k)\big) \Big|
\sum_{k\in \Zd} |e_0(k)|^2+|e_g(k)|^2<\infty\Big\}
\end{equation*}
be the shift-invariant space generated by $\varphi$ and its first-order partial derivatives.
Then an equivalent formulation of the condition
\eqref{blindcondition}  is that
the integer-shifts of $\varphi$ and its first-order partial derivatives
 form a Riesz basis for the  space  $V_2(\varphi, \nabla \varphi)$ and
 signals $h$ living in  $V_2(\varphi, \nabla \varphi)$
can be recovered stably from their samples $\langle h, \psi_m(\cdot-k)\rangle, 1\le m\le M, k\in \Zd$.
%
%; i.e.; there exist positive constants $A_1$ and $A_2$ such that
%\begin{equation}
%A_1\|h\|_2^2\le \sum_{m=1}^M \|(\langle h, \psi_m(\cdot-k)\rangle)_{k\in \Zd}\|_2^2\le A_2 \|h\|_2^2\quad {\rm for \ all}
%\ h\in V_2(\varphi, \nabla \varphi)
%\end{equation}
The readers may refer to \cite{akramgrochenigsiam, akramst05, sunaicm10, unsersurvey} %unserakram94, walter92}
and references therein for sampling and reconstruction of signals in  a shift-invariant space.
}\end{rem}

\begin{proof}[Proof of Theorem \ref{blindsampling.tm}]
By \eqref{blindcondition}, the vector  $A(\xi):=\big([\widehat {\varphi}, \widehat \psi_1](\xi), \cdots, ([\widehat {\varphi}, \widehat \psi_M](\xi)\big)$
is  nonzero for all $\xi\in \Rd$.  Here the bracket product is defined as follows: $[f, g](\xi)=\sum_{l\in \Zd}
f(\xi+2l\pi) \overline{g(\xi+2l\pi)}$.
Hence \begin{equation}\label{blindsampling.tm.pf.eq1}
R(\xi):=\sum_{m=1}^M |[\widehat {\varphi}, \widehat \psi_m](\xi)|^2>0\quad  {\rm  for\  all} \ \xi\in \Rd.\end{equation}
By \eqref{friidentification.tm.eq2} and  \eqref{blindsampling.tm.eq1}, the Fourier coefficients of
$[\widehat {\varphi}, \widehat \psi_m](\xi), 1\le m\le M$,
have  polynomial decay of order $\beta>d$. Here  a periodic function  $f(\xi)=\sum_{k\in \Zd} c_f(k) e^{- i k\xi}$ is said to have Fourier coefficients with polynomial decay of order $\beta\ge 0$ if
$\sup_{k\in \Zd} | c_f(k) (1+|k|)^\beta<\infty$. This, together with
\eqref{blindsampling.tm.pf.eq1} and Wiener's lemma for infinite matrices in the Jaffard class \cite{jaffard90, suncasp05}, implies that
$(R(\xi))^{-1}$  has Fourier coefficients  with polynomial decay of order $\beta>d$.
So does
\begin{equation}\label{blindsampling.tm.pf.eq1+}(R(\xi))^{-1} A(\xi):=\Big(\sum_{k\in \Zd} r_1(k)e^{-ik\xi}, \cdots, \sum_{k\in \Zd} r_M(k)e^{-ik\xi}\Big).\end{equation}
Thus
\begin{equation}\label{blindsampling.tm.pf.eq2}
D_2:=\sup_{1\le m\le M} \sum_{k\in \Zd} |r_m(k)|<\infty.
\end{equation}

Given the sampling data
 $y_m(k):=\langle h, \psi_m(\cdot-k)\rangle, {1\le m\le M, k\in \Zd}$, of a signal  $h(t)=
\sum_{k\in \Zd} c(k) \varphi(\cdot-k-\sigma(k))$   satisfying
 $\|(c(k))_{k\in \Zd}\|_{\ell^\infty_\oslash}\le L$ and $\|(\sigma(k))_{k\in \Zd}\|_\infty\le \delta_1$.
Define  $\tilde c:=(\tilde c(k))_{k\in \Zd}$ by
$\tilde c(k)=\sum_{m=1}^M\sum_{k'\in\Zd} r_{m}(k-k') y_m(k') ,  k\in\Zd$.
From the construction of vectors $(r_{m}(k))_{k\in \Zd}$, we see that
$c(k)=\sum_{m=1}^M \sum_{k'\in \Zd}r_{m}(k-k') \langle \tilde h, \psi_m(\cdot-k')\rangle$ for all $k\in \Zd$, where
 $\tilde h=
\sum_{k\in \Zd} c(k) \varphi(\cdot-k)$. This together with
\eqref{blindsampling.tm.eq1}  gives that
%\begin{eqnarray} \label{blindsampling.tm.pf.eq3}
$$ \|\tilde c-c\|_\infty % & \le &  D_2 \big( \sum_{m=1}^M  \|\psi_m\|_1\big) \|h-\tilde h\|_\infty\nonumber\\
 \le   D_2 \Big( \sum_{m=1}^M  \|\psi_m\|_1\Big)
\|c\|_\infty \Big\|  \sum_{k\in \Zd}   |\varphi(\cdot-k-\sigma(k))-\varphi(\cdot-k)|\Big\|_\infty\nonumber\\
  \le  C_0\|c\|_\infty \|\sigma\|_\infty $$
%\end{eqnarray}
for some positive constant $C_0$. Therefore, if  $\delta_1\le 1/(3 C_0L^2)$, the vector
 $\tilde c_0:=(\tilde c_0(k))_{k\in \Zd}$ defined  by
\begin{equation}\label{blindsampling.co.def}
\tilde c_0(k)=\left\{\begin{array}{ll}  \tilde c(k) & {\rm if}\  |\tilde c(k)| \ge 1/(2L)\\
0 & {\rm otherwise}\end{array}\right.\end{equation}
has the same support, say $\Lambda_0$, as the one of the amplitude vector $c$ of the signal $h$ to be recovered,   and
satisfies
\begin{equation}\label{blindsampling.tm.pf.eq4} \|\tilde c_0-c\|_\infty\le C_0L \|\sigma\|_\infty.\end{equation}
So we may consider
$\sum_{k\in \Lambda_0} \tilde c_0(k) \varphi(\cdot-k)$  as an approximation of the signal  $h=\sum_{k\in \Lambda_0}
c(k) \varphi(\cdot-k-\sigma(k))=\sum_{k\in \Zd} c(k) \varphi(\cdot-k-\sigma(k)) $ to be recovered.

Let $c_0=(\tilde c_0(\lambda))_{\lambda\in \Lambda_0}$,  $\Psi=\{\psi_m(\cdot-k)\}_{1\le m\le M, k\in \Zd}$,
and $S_{\Lambda_0, c_0}$ be the
 the linearization matrix  in \eqref{slambda0.def}.
Then
for any $e=(e(\lambda))_{\lambda\in \Lambda_0}\in (\ell^2(\Lambda_0))^{d+1}$,
\begin{eqnarray*} \|S_{\Lambda_0, c_0} e\|_2^2 & = & \sum_{m=1}^M \sum_{k\in\Zd}
\Big|\big\langle \psi_m(\cdot-k), \sum_{\lambda\in \Lambda_0} \big(-\tilde c_0(\lambda) e_g(\lambda)^T \nabla \varphi(\cdot-\lambda)+ e_o(\lambda)\varphi(\cdot-\lambda)\big)\big\rangle\Big|^2\nonumber\\
   & \ge &    C\Big( \sum_{\lambda\in \Lambda_0} |\tilde c_0(\lambda)|^2 |e_g(\lambda)|^2 + |e_o(\lambda)|^2\Big) \ge  C L^{-2} \|e\|_{(\ell^2(\Lambda_0))^{d+1}}^2
\end{eqnarray*}
and
\begin{eqnarray*} \|S_{\Lambda_0, c_0} e\|_2^2  & \le  &
C \Big\|\sum_{\lambda\in \Lambda_0} \big(-\tilde c_0(\lambda) e_g(\lambda)^T \nabla \varphi(\cdot-\lambda)+ e_o(\lambda)\varphi(\cdot-\lambda)\big)\Big\|_2^2\nonumber\\
& \le &  C \|e\|_{(\ell^2(\Lambda_0))^{d+1}}^2  (1+\|\tilde c_0\|_{\ell^\infty(\Lambda_0)})^2
\le C L^2 \|e\|_{(\ell^2(\Lambda_0))^{d+1}}^2,
\end{eqnarray*}
by \eqref{blindcondition}, \eqref{friidentification.tm.eq2}, \eqref{blindsampling.tm.eq1} and \eqref{blindsampling.tm.pf.eq4}, where $e(\lambda)=\big(\begin{matrix} e_g(\lambda)\\ e_o(\lambda)\end{matrix}\big), \lambda\in \Lambda_0$.
The above estimates lead to the conclusion that the linearization matrix $S_{\Lambda_0, c_0}$ satisfies the stability condition \eqref{slambda0.stability}.
This together with \eqref{blindsampling.tm.pf.eq4} and Theorem  \ref{friidentification.tm} completes the proof.
\end{proof}

\subsection{Numerical simulation}
In the following simulation,  the signal $h$ to be recovered has the  following parametric representation
\begin{equation}\label{fri.signal.example}
h(t)=\sum_{i=1}^{20} a_i \varphi_0(t-t_i),\end{equation}
 where  $\varphi_0(t)=\exp(-4(\pi/2)^{2/3}t^2)$,
%the innovation position vector $\Lambda=(\lambda_i)_{1\le i\le 20}$
%
the innovation positions $t_i, 1\le i\le 20$,
are randomly selected numbers in $[0.5, 19.5]$ satisfying $t_1=0.5, t_{20}=19.5$ and  $0.5\le \min_{1\le i\le 19} t_{i+1}-t_i\le
\max_{1\le i\le 19} t_{i+1}-t_i
\le 1.5$, and the amplitudes
$a_i, 1\le i\le 20$,  are random selected numbers in $[-1,-0.1]\cap [0.1, 1]$.
We consider the recovery of  such a signal $h$ as it can be used to model diffusion of  point sources of pollution
and  the generating signal $\varphi_0$ does not satisfy the Strang-Fix condition. %, as $h=\varphi_0*(\sum_{i=1}^{20} a_i \delta(\cdot-t_i))$.

\begin{figure}[hbt]
\centering
\begin{tabular} {c}
   \includegraphics[width=90mm]{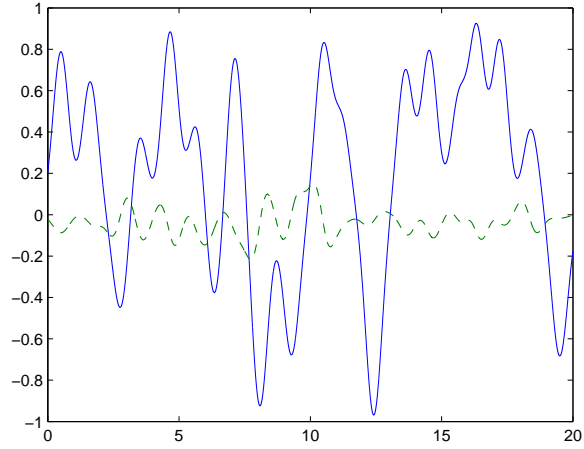} %\\
%   (a) \\
%         \includegraphics[width=90mm]{./squarerootsampler.eps} \\
       %& \includegraphics[width=38mm]{./sample1.eps}\\
       %& \includegraphics[width=60mm]{./gaussiansample2.eps} \\
   %   {\rm \large Gaussian window $\exp(-|x|^2/2\sigma^2)$ with  $\sigma=1/6$} &
      % & \includegraphics[width=60mm]{./originalsignal.eps} \\
% (b)  & (c)
%   {\rm \large (a) Delta impulse signal with innovative rate 2}\\
%    \includegraphics[width=50mm]{./sample1.eps}\\ % & \includegraphics[width=60mm]{./sample2.eps} \\
% {\rm \large samples on uniform grid  with  sampling rate $6$.}
   \end{tabular}
\caption{\small The original signal $h(t)=\sum_{i=1}^{20} a_i \varphi_0(t-t_i)$ to be recovered is plotted in a continuous line, while
the difference between the original signal $h(t)$ and the approximate signal
$h_{\rm app}(t)=\sum_{i=1}^{20} a_{0, i} \varphi_0(t-t_{0,i})$
is plotted in a dashed line.}
\label{frisignal.fig}
\end{figure}
The given samples
$y_j, 1\le j\le 41$, of the signal $h$ are obtained by average sampling the companded  signal $\sin (\pi h(t)/2)$
by  the shifted square-root sampler $\psi_{0}$,
\begin{equation} \label{noisedata.table2} y_j=\langle \sin(\pi h(\cdot)/2)\chi_{[0, 20]}, \psi_{0}(\cdot+1-j/2)\rangle, \quad 1\le j\le 41,\end{equation}
where %the square-root sampler $s$ is given by
\begin{equation}\label{squareroorsampler.def}
\psi_0(t)=\left\{\begin{array}{ll} \max(1/4-|2t-1/4|, 0) & {\rm if} \ 0\le t\le 1/4\\
4t-1 & {\rm if} \ 1/4\le t\le 1/2\\
1  &  {\rm if} \ 1/2\le t\le 1\\
0 & {\rm otherwise}\end{array}\right.
\end{equation}
is the square-root sampling kernel, see Figure \ref{squrerootsampler.fig}.
The
 approximate innovation positions  $t_{0,i}$  and amplitudes  $c_{0, i}, 1\le i\le 20$,
 are given by
$$a_{0,i}=\lfloor 10 a_i\rfloor /10\quad {\rm  and} \quad t_{0,i}=\lfloor 10 t_i\rfloor /10, \ 1\le i\le 20,$$
where $\lfloor a\rfloor $ is the largest integer less than or equal to $a\in \RR$.
\begin{figure}[hbt]
\centering
\begin{tabular}{c}
         \includegraphics[width=90mm]{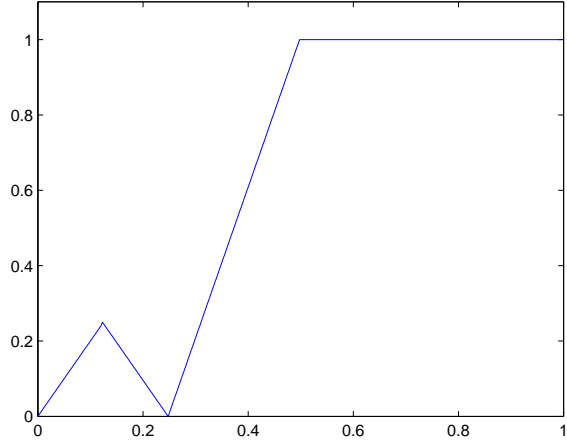}
   \end{tabular}
\caption{\small The square-root sample $\psi_0$ in \eqref{squareroorsampler.def}.}
\label{squrerootsampler.fig}
\end{figure}
We use  the  Van-Cittert iteration  \eqref{friidentification.vancittert} in Theorem \ref{friidentification.vancittert.tm} to recover the exact innovation positions and amplitudes
with $c_0=\sigma_0=0$ as initials  and  $\alpha=0.5$ as the relaxation factor.
  In Tables \ref{identification.tab} below, we list  the amplitude errors $\|c_n-c_\infty\|_\infty$
in the second column,
 the innovation position errors $\|\sigma_n-\sigma_\infty\|_\infty$  in the third column, and
the sampling errors  $\|\langle F(h_n)-F(h), \Psi\rangle\|_\infty$ in the forth column,
 where $F(t)=\sin (\pi t/2)$ and $\Psi=(\psi_0(\cdot+1-j/2)\chi_{[0,20]})_{1\le j\le 41}$.
 \begin{table}[h]
\caption{Local identification of innovation  positions and qualification of amplitudes via the Van-Cittert iteration method}
\begin{tabular}{ccccc}
\hline\hline
 n %Iteration number
 &  $\|c_n-c_\infty\|_\infty$ %Amplitude error
& $\|\sigma_n-\sigma_\infty\|_\infty$ %Innovation position error
 &  $\|\langle F(h_n)-F(h), \Psi\rangle\|_\infty$ %e_{n}$ % Sampling error
 \\
\hline
1 &   0.1465   &   0.4085 &     0.2805\\
%2 &   0.1909 &     0.4134 &     0.3296\\
% 3 &    0.2542 &     0.2954 &     0.1820\\
%4 &     0.2331 &     0.0980 &     0.1520\\
5 &     0.1459 &     0.0292   &    0.1590\\
%6 &     0.0799 &     0.0201  &   0.1486\\
%7 &     0.0453  &    0.0134  &    0.0913\\
%8 &     0.0268 &     0.0099 &     0.0512\\
%9 &     0.0154 &     0.0069 &     0.0273 \\
10 &     0.0088 &     0.0045 &     0.0140\\
%11 &     0.0049  &    0.0029   &   0.0070\\
%12 &     0.0028  &    0.0018  &   0.0039\\
%13 &     0.0015 &     0.0011 &     0.0022\\
%14 &     0.0009  &    0.0007  &    0.0012\\
15 &     0.0005  &    0.0004  &    0.0007\\
%16 &     0.0003  &    0.0002  &    0.0004\\
%17 &     0.0001  &    0.0001   &    0.0002\\
%18 &     0.0001   &   0.0001   &    0.0001\\
%19&     0.0000   &   0.0000    &  0.0001\\
$\ge 20$ &    0.0000   &   0.0000    &  0.0000
\\
\hline \hline
\end{tabular}
\label{identification.tab}
\end{table}
We observe that the sequences $\sigma_n$ and $c_n, n\ge 0$,
in the Van-Cittert iteration  \eqref{friidentification.vancittert} %\eqref{localiteration.def}
converges and  their limits,  denoted by $\sigma_\infty$ and $c_\infty$ respectively,
satisfy
$\sigma_\infty=(t_i-t_{0,i})_{1\le i\le 20}$ and
$c_\infty=(a_i-a_{0,i})_{1\le i\le 20}$. Hence we reach to the desired conclusion that  the signal $h$
in \eqref{fri.signal.example}
 can be recovered {\bf accurately}  from its samples via
the Van-Cittert iteration \eqref{friidentification.vancittert} when its innovation positions and amplitudes  are given approximately.
Moreover, the above recovery process is stable in the presence of bounded noise in Theorem \ref{errorestimate.tm}.
In  Table \ref{identification.tab2} below, we list  numerical results when we apply the Van-Citert iteration \eqref{friidentification.vancittert}
with  exact samples $y_j, 1\le j\le 41$, in \eqref{noisedata.table2} replaced by the  samples $\tilde y_j, 1\le j\le 41$, corrupted by $5\%$ random additive noise,
i.e., $$\tilde y_j= y_j+0.05*\epsilon_j*(\max_{1\le k\le 41}|y_k|), \ 1\le j\le 41,$$
where $\epsilon_j\in [-1, 1], 1\le j\le 41$, are random noise.

  \begin{table}[h]
\caption{Local identification of innovation  positions and qualification of amplitudes via the Van-Cittert iteration method in  the presence of
$5\%$ bounded random noise}
\begin{tabular}{ccccc}
\hline\hline
 n %Iteration number
 &  $\|c_n-c_\infty\|_\infty$ %Amplitude error
& $\|\sigma_n-\sigma_\infty\|_\infty$ %Innovation position error
 &  $\|\langle F(h_n)-F(h), \Psi\rangle\|_\infty$ %e_{n}$ % Sampling error
 \\
\hline
1 & 0.1725 &   0.4593 &   0.2805\\
%2&   0.2577 &    0.5275 &    0.4366\\
% 3&   0.3817 &   0.4474  &   0.3763\\
%4&    0.4191 &   0.1645 &    0.1871\\
5&    0.3059  &  0.0818 &    0.3769\\
%6&    0.1863 &   0.0687 &    0.3518\\
%7  &  0.1118 &    0.0668 &    0.2348\\
%8&    0.1108 &    0.0662 &    0.1591\\
%9&    0.1109 &    0.0662 &    0.1106\\
10 &    0.1110  &   0.0716 &    0.0820\\
%11&    0.1110 &    0.0784 &    0.0670\\
%12&    0.1110  &   0.0820 &    0.0648\\
%13&    0.1110 &    0.0836 &    0.0648\\
%14 &    0.1110 &    0.0843 &    0.0648\\
15&    0.1110 &   0.0847 &    0.0648\\
% 16 &  0.1110  &   0.0850 &    0.0648\\
%17 &    0.1110 &    0.0851 &    0.0648\\
%18&    0.1110 &   0.0852 &    0.0648\\
 $\ge$ 19&   0.1110  &   0.0853 &    0.0648\\
%20&    0.1110 &    0.0853  &   0.0648
\hline \hline
\end{tabular}
\label{identification.tab2}
\end{table}

{\bf Acknowledgement} \  The author would like to thank  anonymous reviewers for their valuable comments for the improvement of the manuscript.

\end{document}